\documentclass{amsart}
\usepackage{amsmath,amssymb,bm}
\usepackage{amsmath,amsthm,amssymb,cases}
\usepackage[all]{xy}
\usepackage{amsmath}
\usepackage{amssymb}
\usepackage{amscd}
\usepackage{amsthm}
\usepackage[dvips]{graphicx}
\usepackage{color}

\theoremstyle{definition}
\newtheorem{defn}{Definition}[section]

\newtheorem{rem}[defn]{Remark}
\theoremstyle{plain}
\newtheorem{thm}[defn]{Theorem}
\newtheorem{prop}[defn]{Proposition}
\newtheorem{lem}[defn]{Lemma}

\numberwithin{equation}{section}
\allowdisplaybreaks
\title[Presentation for mapping class group]{An infinite presentation for the mapping class group of a non-orientable surface with boundary}

\author[R.~Kobayashi]{Ryoma Kobayashi}
\address{
(Ryoma Kobayashi)
Department of General Education, Ishikawa National College of Technology, Tsubata, Ishikawa, 929-0392, Japan 
}
\email{kobayashi\_ryoma@ishikawa-nct.ac.jp}
\author[G.~Omori]{Genki Omori}
\address{
(Genki Omori)
Department of Mathematics,
Tokyo Institute of Technology,
Oh-okayama, Meguro, Tokyo 152-8551, Japan
}

\email{omori.g.aa@m.titech.ac.jp}
\date{\today}
\begin{document}
\maketitle
\begin{abstract}
We give an infinite presentation for the mapping class group of a non-orientable surface with boundary components. The presentation is a generalization of the presentation given by the second author~\cite{Omori}. 
\end{abstract}

\section{Introduction}

Let $\Sigma _{g,n}$ be a compact connected oriented surface of genus $g\geq 0$ with $n\geq 0$ boundary components. The {\it mapping class group} $\mathcal{M}(\Sigma _{g,n})$ of $\Sigma _{g,n}$ is the group of isotopy classes of orientation preserving self-diffeomorphisms on $\Sigma _{g,n}$ fixing the boundary pointwise. A finite presentation for $\mathcal{M}(\Sigma _{g,n})$ was given by Hatcher-Thurston~\cite{Hatcher-Thurston}, Harer~\cite{Harer}, Wajnryb~\cite{Wajnryb}, Gervais~\cite{Gervais2} and Labru\`ere-Paris~\cite{Labruere-Paris}. Gervais~\cite{Gervais} obtained an infinite presentation for $\mathcal{M}(\Sigma _{g,n})$ by using Wajnryb's finite presentation for $\mathcal{M}(\Sigma _{g,n})$, and Luo~\cite{Luo} rewrote Gervais' presentation into a simpler infinite presentation (see Theorem~\ref{pres_Gervais}).

Let $N_{g,n}$ be a compact connected non-orientable surface of genus $g\geq 1$ with $n\geq 0$ boundary components. The surface $N_g=N_{g,0}$ is a connected sum of $g$ real projective planes. The mapping class group $\mathcal{M}(N_{g,n})$ of $N_{g,n}$ is the group of isotopy classes of self-diffeomorphisms on $N_{g,n}$ fixing the boundary pointwise. For $g\geq 2$ and $n\in \{ 0,1\}$, a finite presentation for $\mathcal{M}(N_{g,n})$ was given by Lickorish~\cite{Lickorish1}, Birman-Chillingworth~\cite{Birman-Chillingworth}, Stukow~\cite{Stukow1} and Paris-Szepietowski~\cite{Paris-Szepietowski}. Note that $\mathcal{M}(N_1)$ and $\mathcal{M}(N_{1,1})$ are trivial (see \cite[Theorem~3.4]{Epstein}) and $\mathcal{M}(N_2)$ is finite (see \cite[Lemma~5]{Lickorish1}). Stukow~\cite{Stukow3} rewrote Paris-Szepietowski's presentation into a finite presentation with Dehn twists and a ``Y-homeomorphism'' as generators (see Theorem~\ref{thm_Stukow}). The second author~\cite{Omori} gave a simple infinite presentation for $\mathcal{M}(N_{g,n})$ for $g\geq 1$ and $n\in \{ 0,1\}$. The generating set consists of all Dehn twits and all ``crosscap pushing maps'' along simple loops. We review the crosscap pushing map in Section~\ref{Preliminaries}.

In this paper, we give a simple infinite presentation for $\mathcal{M}(N_{g,n})$ for arbitrary $g\geq 0$ and $n\geq 0$ (Theorem~\ref{main-thm}). The presentation is a generalization of the presentation given by the second author~\cite{Omori}. We will prove Theorem~\ref{main-thm} by applying Gervais' argument to a finite presentation for $\mathcal{M}(N_{g,n})$ in Proposition~\ref{thm_finite_presentation}. 

Contents of this paper are as follows. In Section~\ref{Preliminaries}, we prepare some elements of $\mathcal{M}(N_{g,n})$ and some relations among their elements in $\mathcal{M}(N_{g,n})$, and review the infinite presentation for $\mathcal{M}(\Sigma _{g,n})$ (Theorem~\ref{pres_Gervais}) which is an improvement by Luo~\cite{Luo} of Gervais' presentation in \cite{Gervais}. In Section~\ref{section_finitepres}, we review Stukow's finite presentation for $\mathcal{M}(N_{g,n})$ when $n\in \{ 0,1\}$ (Theorem~\ref{thm_Stukow}) and give a finite presentation for $\mathcal{M}(N_{g,n})$ when $n\geq 2$ (Proposition~\ref{thm_finite_presentation}). In the proof of the main theorem in Section~\ref{section-mainthm}, we use their finite presentations for $\mathcal{M}(N_{g,n})$. In Section~\ref{section-mainthm}, we give the main theorem (Theorem~\ref{main-thm}) in this paper and a proof of the main theorem. Finally, in Section~\ref{section_finite_presentation}, we give a proof of Proposition~\ref{thm_finite_presentation}.

\section{Preliminaries}\label{Preliminaries}

\subsection{Relations among Dehn twists and Gervais' presentation}
Let $S$ be either $N_{g,n}$ or $\Sigma _{g,n}$. We denote by $\mathcal{N}_{S}(A)$ a regular neighborhood of a subset $A$ in $S$. 
We assume that every simple closed curve on $S$ is oriented throughout this paper, and for simple closed curves $c_1$, $c_2$ on $S$, $c_1=c_2$ means $c_1$ is isotopic to $c_2$ in consideration of their orientations. Denote by $c^{-1}$ the inverse curve of a simple closed curve $c$ on $S$. Note that $(c^{-1})^{-1}=c$. For a two-sided simple closed curve $c$ on $S$, we can take two orientations $+_c$ and $-_c$ of $\mathcal{N}_{S}(c)$. When $S$ is orientable, we take $+_c$ as the orientation of $\mathcal{N}_{S}(c)$ which is induced by the orientation of $S$. Then for a two-sided simple closed curve $c$ on $S$ and an orientation $\theta \in \{ +_c, -_c\}$ of $\mathcal{N}_{S}(c)$, denote by $t_{c;\theta }$ the right-handed Dehn twist along $c$ on $S$ with respect to $\theta $. Note that $t_{c;+_c}=t_{c^{-1};+_c}=t_{c;-_c}^{-1}$.
For some convenience, we write $t_c=t_{c;+_c}$ for a two-sided simple closed curve $c$, where orientation of $\mathcal{N}_{S}(c)$ is given explicitly (for instance, $S$ is an orientable surface). In particular, for a given explicit two-sided simple closed curve, an arrow on a side of the simple closed curve indicates the direction of the Dehn twist (see Figure~\ref{dehntwist}). For elements $f=[\varphi ]$, $h=[\psi ]\in \mathcal{M}(S)$, we define $fh:=[\varphi \circ \psi ]\in \mathcal{M}(S)$.

\begin{figure}[h]
\includegraphics[scale=0.65]{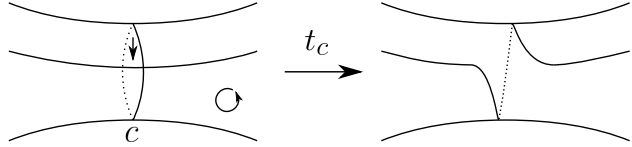}
\caption{The right-handed Dehn twist $t_c=t_{c;\theta }$ along a two-sided simple closed curve $c$ on $S$ with respect to the orientation $\theta \in \{ +_c, -_c\}$ of $\mathcal{N}_{S}(c)$ as in the figure.}\label{dehntwist}
\end{figure}

Recall the following relations in $\mathcal{M}(S)$ among Dehn twists along two-sided simple closed curves on $S$.

\begin{lem}\label{trivial}
Let $c$ be a two-sided simple closed curve $c$ on $S$ and $\theta \in \{ +_c, -_c\}$ an orientation of $\mathcal{N}_{S}(c)$. If $c$
bounds a disk or a M\"{o}bius band in $S$, then we have $t_{c;\theta }=1$ in $\mathcal{M}(S)$.
\end{lem} 

For a two-sided simple closed curve $c$ on $S$ and $f\in \mathcal{M}(S)$, we have a bijection $f_\ast =(f|_{\mathcal{N}_{S}(c)})_\ast :\{ +_c, -_c\}\rightarrow \{ +_{f(c)}, -_{f(c)}\}$.

\begin{lem}[The braid relation (i)]\label{braid1}
For a two-sided simple closed curve $c$ on $S$ and $f\in \mathcal{M}(S)$, we have 
\begin{eqnarray*}
t_{f(c);\theta ^\prime } &=& ft_{c;\theta }f^{-1} \hspace{0.5cm}\text{when }f_\ast (\theta )=\theta ^\prime ,\\
t_{f(c);\theta ^\prime }^{-1} &=& ft_{c;\theta }f^{-1} \hspace{0.5cm}\text{when }f_\ast (\theta )\not =\theta ^\prime .
\end{eqnarray*}
\end{lem}

When $f$ in Lemma~\ref{braid1} is a Dehn twist $t_d$ along a two-sided simple closed curve $d$ and the geometric intersection number $|c\cap d|$ of $c$ and $d$ is $m$, we denote by $T_m$ the braid relation.

Let $c_1$, $c_2$, $\dots $, $c_k$ be two-sided simple closed curves on $S$. The sequence $c_1$, $c_2$, $\dots $, $c_k$ is a {\it k-chain on $S$} if $c_1$, $c_2$, $\dots $, $c_k$ satisfy $|c_i\cap c_{i+1}|=1$ for each $i=1$, $2$, $\dots $, $k-1$ and $|c_i\cap c_j|=0$ for $|j-i|>1$.

\begin{lem}[The $k$-chain relation]\label{chain}
Let $c_1$, $c_2$, $\dots $, $c_k$ be a $k$-chain on $S$ and let $\delta $, $\delta ^\prime $ (resp. $\delta $) be distinct boundary components (resp. the boundary component) of $\mathcal{N}_S(c_1\cup c_2\cup \cdots \cup c_k)$ when $k$ is odd (resp. even). We give an orientation $\theta \in \{ +_\delta ,-_\delta \}$ of $\mathcal{N}_S(\delta )$ and the orientation of $\mathcal{N}_S(c_1\cup c_2\cup \cdots \cup c_k)$ which is induced by $\theta $. Then we have
\begin{eqnarray*}
(t_{c_1;\theta _{c_1}}^{\varepsilon _{c_1}}t_{c_2;\theta _{c_2}}^{\varepsilon _{c_2}}\cdots t_{c_k;\theta _{c_k}}^{\varepsilon _{c_k}})^{k+1} &=& t_{\delta ;\theta }t_{\delta ^\prime ;\theta _{\delta ^\prime }}^{\varepsilon _{\delta ^\prime }} \hspace{0.5cm}\text{when}\ k\text{ is odd},\\
(t_{c_1;\theta _{c_1}}^{\varepsilon _{c_1}}t_{c_2;\theta _{c_2}}^{\varepsilon _{c_2}}\cdots t_{c_k;\theta _{c_k}}^{\varepsilon _{c_k}})^{2k+2} &=& t_{\delta ;\theta } \hspace{0.5cm}\text{when}\ k\text{ is even},
\end{eqnarray*}
where for $c\in \{ c_1, c_2, \dots , c_k, \delta ^\prime \}$, $\varepsilon _c=1$ if $\theta _c$ coincides with the orientation of $\mathcal{N}_S(c_1\cup c_2\cup \cdots \cup c_k)$, and $\varepsilon _c=-1$ otherwise.
\end{lem}

\begin{lem}[The lantern relation]\label{lantern}
Let $\Sigma $ be a subsurface of $S$ which is diffeomorphic to $\Sigma _{0,4}$ and let $\delta _{12}$, $\delta _{23}$, $\delta _{13}$, $\delta _1$, $\delta _2$, $\delta _3$ and $\delta _4$ be simple closed curves on $\Sigma $ as in Figure~\ref{lantern1}. We give an orientation $\theta \in \{ +_{\delta _4},-_{\delta _4}\}$ of $\mathcal{N}_S(\delta _4)$ and the orientation of $\Sigma $ which is induced by $\theta $. Then we have
\[
t_{\delta _{12};\theta _{\delta _{12}}}^{\varepsilon_{\delta _{12}}}t_{\delta _{23};\theta _{\delta _{23}}}^{\varepsilon_{\delta _{23}}}t_{\delta _{13};\theta _{\delta _{13}}}^{\varepsilon_{\delta _{13}}}=t_{\delta _1;\theta _{\delta _1}}^{\varepsilon_{\delta _1}}t_{\delta _2;\theta _{\delta _2}}^{\varepsilon_{\delta _2}}t_{\delta _3;\theta _{\delta _3}}^{\varepsilon_{\delta _3}}t_{\delta _4;\theta },
\]
where for $c\in \{ \delta _{12}, \delta _{23}, \delta _{13}, \delta _1, \delta _2, \delta _3\}$, $\varepsilon _c=1$ if $\theta _c$ coincides with the orientation of $\Sigma $, and $\varepsilon _c=-1$ otherwise.
\end{lem}

\begin{figure}[h]
\includegraphics[scale=0.75]{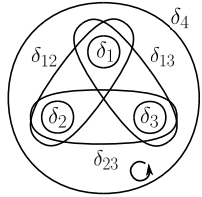}
\caption{The simple closed curves $\delta _{12}$, $\delta _{23}$, $\delta _{13}$, $\delta _1$, $\delta _2$, $\delta _3$ and $\delta _4$ on $\Sigma $.}\label{lantern1}
\end{figure}

Luo's presentation for $\mathcal{M}(\Sigma _{g,n})$, which is an improvement of Gervais' one, is as follows.

\begin{thm}[\cite{Gervais}, \cite{Luo}]\label{pres_Gervais}
For $g\geq 0$ and $n\geq 0$, $\mathcal{M}(\Sigma _{g,n})$ has the following presentation:

generators: $\{ t_c \mid c:\text{ s.c.c. on }\Sigma _{g,n} \}$, where s.c.c. means simple closed curve.

relations:
\begin{enumerate}
 \item[(0$^\prime $)] $t_c=1$ when $c$ bounds a disk in $\Sigma _{g,n}$,
 \item[(I$^\prime $)] All the braid relations $T_0$ and $T_1$,
 \item[(I\hspace{-0.04cm}I)] All the 2-chain relations,
 \item[(I\hspace{-0.04cm}I\hspace{-0.04cm}I)] All the lantern relations.
\end{enumerate} 
\end{thm}

\subsection{Relations among the crosscap pushing maps and Dehn twists}\label{rel_twist_crosscappushing}
Let $\mu $ be a one-sided simple closed curve on $N_{g,n}$ and let $\alpha $ be a simple closed curve on $N_{g,n}$ such that $\mu $ and $\alpha $ intersect transversely at one point. Recall that $\alpha $ is oriented. For these simple closed curves $\mu$ and $\alpha $, we denote by $Y_{\mu , \alpha }$ a self-diffeomorphism on $N_{g,n}$ which is described as the result of pushing the M\"{o}bius band $\mathcal{N}_{N_{g,n}}(\mu )$ once along $\alpha $.  We call $Y_{\mu , \alpha }$ a {\it crosscap pushing map}. In particular, if $\alpha $ is two-sided, we call $Y_{\mu , \alpha }$ a {\it Y-homeomorphism} (or a {\it crosscap slide}), where a {\it crosscap} means a M\"{o}bius band in the interior of a surface. Note that $Y_{\mu ,\alpha }=Y_{\mu ,\alpha ^{-1}}^{-1}=Y_{\mu ^{-1},\alpha }$. The Y-homeomorphism was originally defined by Lickorish~\cite{Lickorish1}. We have the following fundamental relation in $\mathcal{M}(N_{g,n})$ and we also call the relation the {\it braid relation}.

\begin{lem}[The braid relation (ii)]\label{braid2}
Let $\mu $ be a one-sided simple closed curve on $N_{g,n}$ and let $\alpha $ and $\beta $ be simple closed curves on $N_{g,n}$ such that $\mu $ and $\alpha $ intersect transversely at one point. For $f\in \mathcal{M}(N_{g,n})$, suppose that $f(\alpha )=\beta $ or $f(\alpha )=\beta ^{-1}$. Then we have
\begin{eqnarray*}
Y_{f(\mu ),\beta }&=&fY_{\mu ,\alpha }f^{-1} \hspace{0.5cm}\text{when }f(\alpha )=\beta ,\\
Y_{f(\mu ),\beta }^{-1}&=&fY_{\mu ,\alpha }f^{-1} \hspace{0.5cm}\text{when }f(\alpha )=\beta ^{-1}.
\end{eqnarray*}
\end{lem}

We describe crosscap pushing maps from a different point of view. Let $e:D^\prime \hookrightarrow {\rm int}S$ be a smooth embedding of the unit disk $D^\prime \subset \mathbb C$. Put $D:=e(D^\prime )$. Let $S^\prime $ be the surface obtained from $S-{\rm int}D$ by the identification of antipodal points of $\partial D$. We call the manipulation that gives $S^\prime $ from $S$ the {\it blowup of} $S$ {\it on} $D$. Note that the image $M\subset S^\prime$ of $\mathcal{N}_{S-{\rm int}D}(\partial D)\subset S-{\rm int}D$ with respect to the blowup of $S$ on $D$ is a crosscap. Conversely, the {\it blowdown of} $S^\prime$ {\it on }$M$ is the following manipulation that gives $S$ from $S^\prime $. We paste a disk on the boundary obtained by cutting $S$ along the center line $\mu $ of $M$. The blowdown of $S^\prime $ on $M$ is the inverse manipulation of the blowup of $S$ on $D$.

Let $\mu $ be a one-sided simple closed curve on $N_{g,n}$ and let $S$ be the surface which is obtained from $N_{g,n}$ by the blowdown of $N_{g,n}$ on $\mathcal{N}_{N_{g,n}}(\mu )$. Note that $S$ is diffeomorphic to $N_{g-1,n}$ or $\Sigma _{h,n}$ for $g=2h+1$. 
Denote by $x_\mu $ the center point of a disk $D_\mu $ that is pasted on the boundary obtained by cutting $S$ along $\mu $. Let $e:D^\prime \hookrightarrow D_\mu \subset S$ be a smooth embedding of the unit disk $D^\prime \subset \mathbb C$ to $S$ such that $D_\mu =e(D^\prime )$ and $e(0)=x_\mu $. Let $\mathcal{M}(S,x_\mu )$ be the group of isotopy classes of self-diffeomorphisms on $S$ fixing the boundary $\partial S$ and the point $x_\mu $, where isotopies also fix the boundary $\partial S$ and $x_\mu $. Then we have the {\it blowup homomorphism} 
\[
\varphi _\mu :\mathcal{M}(S,x_\mu )\rightarrow \mathcal{M}(N_{g,n})
\]
that is defined as follows. For $h \in \mathcal{M}(S,x_\mu )$, we take a representative $h^\prime $ of $h$ which satisfies either of the following conditions: (a) $h^\prime |_{D_\mu }$ is the identity map on $D_\mu $, (b) $h^\prime (x)=e(\overline{e^{-1}(x)})$ for $x\in D_\mu $, where $\overline{e^{-1}(x)}$ is the complex conjugation of $e^{-1}(x)\in \mathbb C$. Such $h^\prime $ is compatible with the blowup of $S$ on $D_\mu $, thus $\varphi _\mu (h)\in \mathcal{M}(S)$ is induced and well defined (c.f. \cite[Subsection~2.3]{Szepietowski1}). 

The {\it point pushing map} 
\[
j_{x_\mu }:\pi _1(S,x_\mu )\rightarrow \mathcal{M}(S,x_\mu )
\]
is a homomorphism that is defined as follows. For $\gamma \in \pi _1(S,x_\mu )$, $j_{x_\mu }(\gamma )\in \mathcal{M}(S,x_\mu )$ is described as the result of pushing the point $x_\mu $ once along $\gamma $. The point pushing map comes from the Birman exact sequence. Note that for $\gamma _1$, $\gamma _2\in \pi _1(S,x_\mu )$, $\gamma _1\gamma _2$ means $\gamma _1\gamma _2(t)=\gamma _2(2t)$ for $0\leq t\leq \frac{1}{2}$ and $\gamma _1\gamma _2(t)=\gamma _1(2t-1)$ for $\frac{1}{2}\leq t\leq 1$.

Following Szepietowski~\cite{Szepietowski1} we define the composition of the homomorphisms:
\[
\psi _{x_\mu }:=\varphi _\mu \circ j_{x_\mu }:\pi _1(S,x_\mu )\rightarrow \mathcal{M}(N_{g,n}).
\]
For each closed curve $\alpha $ on $N_{g,n}$ which transversely intersects with $\mu $ at one point, we take a loop $\overline{\alpha }$ on $S$ based at $x_\mu $ such that $\overline{\alpha }$ has no self-intersection points on $D_\mu $ and $\alpha $ is the image of $\overline{\alpha }$ with respect to the blowup of $S$ on $D_\mu $. If $\alpha $ is simple, we take $\overline{\alpha }$ as a simple loop. The next two lemmas follow from the description of the point pushing map (see \cite[Lemma~2.2, Lemma~2.3]{Korkmaz2}). 

\begin{lem}\label{pushing1}
For a simple closed curve $\alpha $ on $N_{g,n}$ which transversely intersects with a one-sided simple closed curve $\mu $ on $N_{g,n}$ at one point, we have
\[
\psi _{x_\mu }(\overline{\alpha })=Y_{\mu ,\alpha }.
\]
\end{lem}

\begin{lem}\label{pushing2}
For a one-sided simple closed curve $\alpha $ on $N_{g,n}$ which transversely intersects with a one-sided simple closed curve $\mu $ on $N_{g,n}$ at one point, we take $\mathcal{N}_{S}(\overline{\alpha })$ such that the interior of $\mathcal{N}_{S}(\overline{\alpha })$ contains $D_\mu $ and an orientation $\theta _{\overline{\alpha }}\in \{ +_{\overline{\alpha }},-_{\overline{\alpha }}\}$ of $\mathcal{N}_{S}(\overline{\alpha })$. Denote by $\overline{\delta _1}$ (resp. $\overline{\delta _2}$) the boundary component of $\mathcal{N}_{S}(\overline{\alpha })$ on the right (resp. left) side of $\overline{\alpha }$, and by $\delta _i$ $(i=1,2)$ the two-sided simple closed curve on $N_{g,n}$ which is the image of $\overline{\delta _i}$ with respect to the blowup of $S$ on $D_\mu $. Let $\overline{\theta _i}\in \{ +_{\overline{\delta _i}},-_{\overline{\delta _i}}\}$ $(i=1,2)$ be the orientation of $\mathcal{N}_{S}(\overline{\delta _i})$ which is induced by $\theta _{\overline{\alpha }}$ and $\theta _i\in \{ +_{\delta _i},-_{\delta _i}\}$ $(i=1,2)$ the orientation of $\mathcal{N}_{N_{g,n}}(\delta _i)$ which is induced by $\overline{\theta _i}$ (see Figure~\ref{crosscap_def_twist}). Then we have 
\[
Y_{\mu ,\alpha }=t_{\delta _1;\theta _{\delta _1}}^{\varepsilon _{\delta _1}}t_{\delta _2;\theta _{\delta _2}}^{-\varepsilon _{\delta _2}},
\] 
where $\varepsilon _{\delta _i}=1$ if $\theta _{\delta _i}=\theta _i$, and $\varepsilon _{\delta _i}=-1$ otherwise.
\end{lem}

\begin{figure}[h]
\includegraphics[scale=0.6]{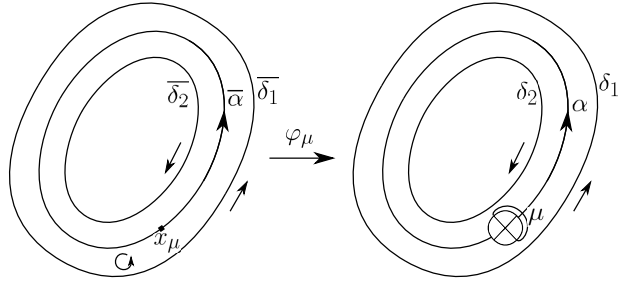}
\caption{Simple closed curves $\overline{\delta _1}$, $\overline{\delta _2}$, $\delta _1$ and $\delta _2$, and orientations $\overline{\theta _1}$, $\overline{\theta _2}$, $\theta _1$ and $\theta _2$ of their regular neighborhoods. The x-mark means that antipodal points of $\partial D_\mu $ are identified.}\label{crosscap_def_twist}
\end{figure}

By the definition of the homomorphism $\psi _{x_\mu }$ and Lemma~\ref{pushing1}, we have the following lemma.

\begin{lem}\label{pushing3}
Let $\alpha $ and $\beta $ be simple closed curves on $N_{g,n}$ which transversely intersect with a one-sided simple closed curve $\mu $ on $N_{g,n}$ at one point each. Suppose that the product $\overline{\alpha }\overline{\beta }$ of $\overline{\alpha }$ and $\overline{\beta }$ in $\pi _1(S,x_\mu )$ is represented by a simple loop on $S$, and $\alpha \beta $ is a simple closed curve on $N_{g,n}$ which is the image of the representative of $\overline{\alpha }\overline{\beta }$ with respect to the blowup of $S$ on $D_\mu $. Then we have
\[
Y_{\mu ,\alpha \beta }=Y_{\mu ,\alpha }Y_{\mu ,\beta }.
\]
\end{lem}

Finally, we recall the following relation between a Dehn twist and a Y-homeomorphism.

\begin{lem}\label{pushing4}
Let $\alpha $ be a two-sided simple closed curve on $N_{g,n}$ which transversely intersect with a one-sided simple closed curve $\mu $ on $N_{g,n}$ at one point and let $\delta $ be the boundary of $\mathcal{N}_{N_{g,n}}(\alpha \cup \mu )$. Since $\overline{\alpha }^2\in \pi _1(S,x_\mu )$ is represented by a two-sided simple loop, we take $\delta _i$ and $\theta _i$ $(i=1,2)$ as in Lemma~\ref{pushing2} when $\overline{\alpha}$ in Lemma~\ref{pushing2} is replaced by $\overline{\alpha}^2$ (see Figure~\ref{yhomeo_square}). Then we have
\begin{eqnarray*}
Y_{\mu ,\alpha }^2&=&t_{\delta ;\theta _{\delta _1}}^{\varepsilon _{\delta _1}} \hspace{0.5cm}\text{when }\delta =\delta _1,\\
Y_{\mu ,\alpha }^2&=&t_{\delta ;\theta _{\delta _2}}^{-\varepsilon _{\delta _2}} \hspace{0.5cm}\text{when }\delta =\delta _2,
\end{eqnarray*}
where $\varepsilon _{\delta _i}=1$ for $i=1,2$ if $\theta _{\delta _i}= \theta _i$, and $\varepsilon _{\delta _i}=-1$ otherwise.
\end{lem}
Lemma~\ref{pushing4} follows from relations in Lemma~\ref{trivial}, Lemma~\ref{pushing2} and Lemma~\ref{pushing3}.

\begin{figure}[h]
\includegraphics[scale=0.6]{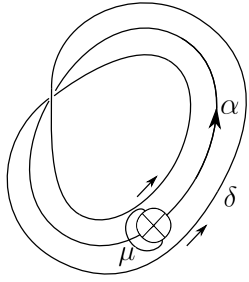}
\caption{The orientation $\theta _1$ of $\mathcal{N}_{N_{g,n}}(\delta )$ when $\delta =\delta _1$ or the orientation $\theta _2$ of $\mathcal{N}_{N_{g,n}}(\delta )$ when $\delta =\delta _2$.}\label{yhomeo_square}
\end{figure}

\section{Finite presentation for $\mathcal{M}(N_{g,n})$}\label{section_finitepres}

In this section, we review Stukow's finite presentation for $\mathcal{M}(N_{g,n})$ when $n\in \{ 0,1\}$ and give a finite presentation for $\mathcal{M}(N_{g,n})$ when $n\geq 2$. We use their finite presentations for $\mathcal{M}(N_{g,n})$ in the proof of the main theorem in Section~\ref{section-mainthm}.

Let $e_i:D^\prime \hookrightarrow {\rm int}\Sigma _{0,1}$ for $i=1$, $2, \dots $, $g+n-1$ be smooth embeddings of the unit disk $D^\prime \subset \mathbb C$ to a disk $\Sigma _{0,1}$ such that $D_i:=e_i(D^\prime )$ and $D_j$ are disjoint for distinct $1\leq i,j\leq g+n-1$. For $n\geq 1$, we take a model of $N_{g,n}$ as the surface obtained from $\Sigma _{0,1}-({\rm int}D_{g+1}\sqcup \cdots \sqcup {\rm int}D_{n-1})$ by the blowups on $D_1,\dots ,D_g$ and we describe the identification of $\partial D_i$ by the x-mark as in Figures~\ref{scc_nonorisurf1}. We denote by $\delta _1, \dots ,\delta _{n-1}$ and $\delta $ boundary components of $N_{g,n}$ as in Figure~\ref{scc_nonorisurf1} which are obtained from $\partial D_{g+1}, \dots ,\partial D_{g+n-1}$ and $\partial \Sigma _{0,1}$, respectively.
Let $\alpha _1, \dots ,\alpha _{g-1}, \beta $ and $\mu _1$ be simple closed curves on $N_{g,n}$ as in Figure~\ref{scc_nonorisurf1} and let $\alpha _{i;j}$ for $1\leq i\leq g-1$ and $1\leq j\leq n-1$, $\rho _{i;j}$ for $1\leq i\leq g$ and $1\leq j\leq n-1$ and $\sigma _{i,j}$, $\bar{\sigma }_{i,j}$ for $1\leq i<j\leq n-1$ be simple closed curves on $N_{g,n}$ as in Figure~\ref{scc_nonorisurf2}. We give orientations of regular neighborhoods of their simple closed curves as in Figure~\ref{scc_nonorisurf1}, \ref{scc_nonorisurf2}. Then we define the mapping classes 
\begin{eqnarray*}
a_i&:=&t_{\alpha _i}   \hspace{1cm} \text{for } 1\leq i\leq g-1, \\
b&:=&t_\beta ,\\
y&:=&Y_{\mu _1,\alpha _1},\\
d_i&:=&t_{\delta _i} \hspace{1cm} \text{for } 1\leq i\leq n-1,\\
a_{i;j}&:=&t_{\alpha _{i;j}}  \hspace{1cm} \text{for } 1\leq i\leq g-1 \text{ and } 1\leq j\leq n-1,\\
r_{i;j}&:=&t_{\rho _{i;j}}  \hspace{1cm} \text{for } 1\leq i\leq g \text{ and } 1\leq j\leq n-1,\\
s_{i,j}&:=&t_{\sigma _{i,j}} \hspace{1cm} \text{for } 1\leq i<j\leq n-1,\\
\bar{s}_{i,j}&:=&t_{\bar{\sigma }_{i,j}} \hspace{1cm} \text{for } 1\leq i<j\leq n-1,\\
\bar{s}_{j,k;i}&:=&\{(a_{1;k}a_1^{-1})^{-1}r_{2;k}\cdots(a_{i-1;k}a_{i-1}^{-1})^{-1}r_{i;k}\}^{-1}\bar{s}_{j,k}\\
&&\{(a_{1;k}a_1^{-1})^{-1}r_{2;k}\cdots(a_{i-1;k}a_{i-1}^{-1})^{-1}r_{i;k}\} \\
&&\text{ for }2\leq i\leq g \text{ and }1\leq j<k\leq n-1.
\end{eqnarray*}
Remark that, for $2\leq i\leq g$ and $1\leq j<k\leq n-1$, $\bar{s}_{j,k;i}$ is the Dehn twist along the simple closed curve $\bar{\sigma }_{j,k;i}$ on $N_{g,n}$ as in Figure~\ref{sigmabar_jki}.

\begin{figure}[h]
\includegraphics[scale=0.7]{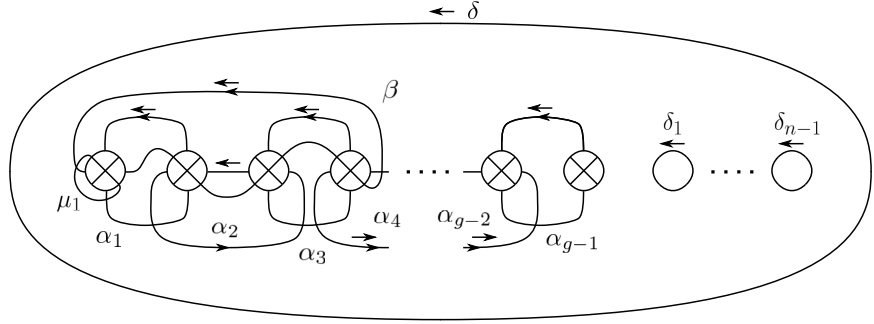}
\caption{A model of $N_{g,n}$ and simple closed curves $\alpha _1$, $\dots $, $\alpha _{g-1}$, $\beta $ and $\mu _1$ on $N_{g,n}$.}\label{scc_nonorisurf1}
\end{figure}

\begin{figure}[h]
\includegraphics[scale=0.65]{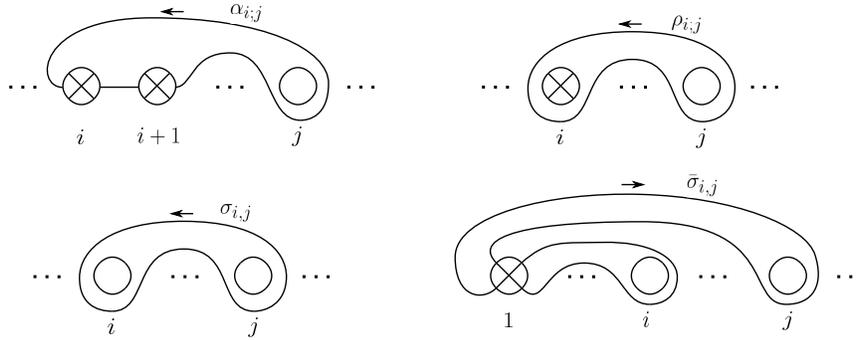}
\caption{The simple closed curves $\alpha _{i;j}$, $\rho _{i;j}$, $\sigma _{i,j}$ and $\bar{\sigma }_{i,j}$ on $N_{g,n}$.}\label{scc_nonorisurf2}
\end{figure}

\begin{figure}[h]
\includegraphics[scale=0.7]{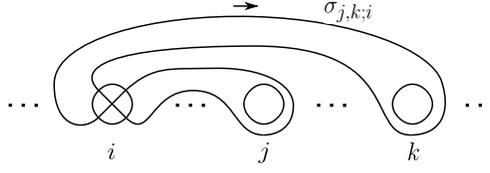}
\caption{The simple closed curve $\bar{\sigma }_{j,k;i}$ on $N_{g,n}$.}\label{sigmabar_jki}
\end{figure}

Epstein~\cite{Epstein} show that $\mathcal{M}(N_{1,1})$ is trivial. We define $[x_1,x_2]=x_1x_2x_1^{-1}x_2^{-1}$. Stukow gave the following finite presentation for $\mathcal{M}(N_{g,1})$ when $g=2$ in \cite{Stukow1}, and when $g\geq 3$ in  \cite{Stukow3} by rewriting the finite presentation in \cite{Paris-Szepietowski}.

\begin{thm}[\cite{Epstein}, \cite{Stukow1}, \cite{Stukow3}]\label{thm_Stukow}
$\mathcal{M}(N_{1,1})$ is the trivial group. $\mathcal{M}(N_{2,1})$ has the presentation
\begin{eqnarray*}
\mathcal{M}(N_{2,1})&=&\bigl< a_1, y \mid ya_1y^{-1}=a_1^{-1}\bigr>.
\end{eqnarray*}

If $g\geq 3$, then $\mathcal{M}(N_{g,1})$ admits a presentation with generators $a_1,\dots , a_{g-1}, y$, and $b$ for $g\geq 4$. The defining relations are
\begin{enumerate}
 \item[(A1)] $[a_i,a_j]=1$\hspace{0.5cm} for $g\geq 4$, $|i-j|>1$,
 \item[(A2)] $a_ia_{i+1}a_i=a_{i+1}a_ia_{i+1}$\hspace{1cm} for $i=1,\dots ,g-2$,
 \item[(A3)] $[a_i,b]=1$\hspace{0.5cm} for $g\geq 4$, $i\not=4$,
 \item[(A4)] $a_4ba_4=ba_4b$\hspace{0.5cm} for $g\geq 5$,
 \item[(A5)] $(a_2a_3a_4b)^{10}=(a_1a_2a_3a_4b)^6$\hspace{0.5cm} for $g\geq 5$,
 \item[(A6)] $(a_2a_3a_4a_5a_6b)^{12}=(a_1a_2a_3a_4a_5a_6b)^9$\hspace{0.5cm} for $g\geq 7$,
 \item[(A9a)] $[b_2,b]=1$\hspace{0.5cm} for $g=6$,
 \item[(A9b)] $[a_{g-5},b_{\frac{g-2}{2}}]=1$\hspace{0.5cm} for $g\geq 8$ even,\\
 where $b_0=a_1$, $b_1=b$ and\\ $b_{i+1}=(b_{i-1}a_{2i}a_{2i+1}a_{2i+2}a_{2i+3}b_i)^5(b_{i-1}a_{2i}a_{2i+1}a_{2i+2}a_{2i+3})^{-6}$ \\
for $1\leq i\leq \frac{g-4}{2}$,
 \item[(B1)] $y(a_2a_3a_1a_2ya_2^{-1} a_1^{-1}a_3^{-1}a_2^{-1}) = (a_2a_3a_1a_2ya_2^{-1}a_1^{-1}a_3^{-1}a_2^{-1})y$ \hspace{0.5cm}for $g\geq 4$,
 \item[(B2)] $y(a_2a_1y^{-1}a_2^{-1}ya_1a_2)y=a_1(a_2a_1y^{-1}a_2^{-1}ya_1a_2)a_1$,
 \item[(B3)] $[a_i,y]=1$\hspace{0.5cm} for $g\geq 4$, $i=3,\dots ,g-1$,
 \item[(B4)] $a_2(ya_2y^{-1}) = (ya_2y^{-1})a_2$,
 \item[(B5)] $ya_1=a_1^{-1}y$,
 \item[(B6)] $byby^{-1} = \{a_1a_2a_3(y^{-1}a_2y)a_3^{-1}a_2^{-1}a_1^{-1} \}\{a_2^{-1}a_3^{-1}(ya_2y^{-1})a_3a_2\}$\hspace{0.5cm} for $g\geq 4$,
 \item[(B7)] $[(a_4a_5a_3a_4a_2a_3a_1a_2ya_2^{-1}a_1^{-1}a_3^{-1}a_2^{-1}a_4^{-1}a_3^{-1}a_5^{-1}a_4^{-1}),b] =1$\hspace{0.5cm} for $g\geq 6$,
 \item[(B8)] $\{(ya_1^{-1}a_2^{-1}a_3^{-1}a_4^{-1})b(a_4a_3a_2a_1y^{-1})\}\{(a_1^{-1}a_2^{-1}a_3^{-1}a_4^{-1})b^{-1}(a_4a_3a_2a_1)\}$
	     $=\{(a_4^{-1}a_3^{-1}a_2^{-1})y(a_2a_3a_4)\}\{a_3^{-1}a_2^{-1}y^{-1}a_2a_3\}\{a_2^{-1}ya_2\}y^{-1}$\hspace{0.5cm} for $g\geq 5$.
\end{enumerate}
\end{thm}

For $n\geq 2$, we have the following finite presentation for $\mathcal{M}(N_{g,n})$ and give the proof in Section~\ref{proof_finite_presentation}. 

\begin{prop}\label{thm_finite_presentation}
For $g\geq 1$ and $n\geq 2$, $\mathcal{M}(N_{g,n})$ has the presentation which obtained from the finite presentation for $\mathcal{M}(N_{g,1})$ in Theorem~\ref{thm_Stukow} by adding generators $d_i$ $(i=1,\dots ,n-1)$, $a_{i;j}$ $(1\leq i\leq g-1, 1\leq j\leq n-1)$, $r_{i,j}$ $(1\leq i\leq g, 1\leq j\leq n-1)$, $s_{i,j}$ $(1\leq i<j\leq n-1)$ and $\bar{s}_{i,j}$ $(1\leq i<j\leq n-1)$ and the following relations for $1\leq{i,m}\leq{g}$, $1\leq{j<k}\leq{n-1}$ and $1\leq{l<t<k}$:

\begin{enumerate}
\item[(D0)] $[d_j,a_i]=[d_j,y]=[d_j,b]=[d_j,d_l]=[d_j,a_{i;l}]=[d_j,r_{i;l}]=[d_j,s_{l,t}]=[d_j,\bar{s}_{l,t}]=1,$
\item[(D1a)] $a_m(a_{i;k}a_i^{-1})a_m^{-1}=
		\left\{
		\begin{array}{ll}
		(a_{i;k}a_i^{-1})(a_{i-1;k}a_{i-1}^{-1}) \hspace{0.5cm}\text{for }m=i-1,\\
		(a_{i+1;k}a_{i+1}^{-1})^{-1}(a_{i;k}a_i^{-1})\hspace{0.5cm}\text{for }m=i+1,\\
		a_{i;k}a_i^{-1}\hspace{0.5cm}\text{for }m\neq{i-1,i+1},
		\end{array}
		\right.$
\item[(D1b)] $y(a_{i;k}a_i^{-1})y^{-1}=
		\left\{
		\begin{array}{ll}
		(a_{1;k}a_1^{-1})^{-1}r_{2;k}r_{1;k}d_{n-1}^{-2}\hspace{0.5cm}\text{for }i=1,\\
		(a_{2;k}a_2^{-1})r_{1;k}d_k^{-1}\hspace{0.5cm}\text{for }i=2,\\
		a_{i;k}a_i^{-1}\hspace{0.5cm}\text{for }i\geq 3,
		\end{array}
		\right.$
\item[(D1c)] $b(a_{i;k}a_i^{-1})b^{-1}=\\
		\left\{
		\begin{array}{ll}
		\{(a_{3;k}a_3^{-1})(a_{1;k}a_1^{-1})\}^{-1}(a_{1;k}a_1^{-1})\{(a_{3;k}a_3^{-1})(a_{1;k}a_1^{-1})\}\hspace{0.5cm}\text{for }i=1,\\
		\{(a_{3;k}a_3^{-1})(a_{1;k}a_1^{-1})\}^{-1}(a_{2;k}a_2^{-1})\{(a_{3;k}a_3^{-1})(a_{1;k}a_1^{-1})\}\hspace{0.5cm}\text{for }i=2,\\
		(a_{1;k}a_1^{-1})^{-1}(a_{3;k}a_3^{-1})(a_{1;k}a_1^{-1})\hspace{0.5cm}\text{for }i=3,\\
		(a_{4;k}a_4^{-1})(a_{3;k}a_3^{-1})(a_{1;k}a_1^{-1})d_k^{-1}\hspace{0.5cm}\text{for }i=4,\\
		a_{i;k}a_i^{-1}\hspace{0.5cm}\text{for }i\geq5,
		\end{array}
		\right.$
\item[(D1d)] $a_{m;l}(a_{i;k}a_i^{-1})a_{m;l}^{-1}=\\
		\left\{
		\begin{array}{ll}
		{[(s_{l,k}d_l^{-1})^{-1},(a_{m;k}a_m^{-1})^{-1}]}^{-1}(a_{i;k}a_i^{-1}){[(s_{l,k}d_l^{-1})^{-1},(a_{m;k}a_m^{-1})^{-1}]} \\ \text{for }m\leq{i-2},\\
		{[(a_{i-1;k}a_{i-1}^{-1})^{-1},(s_{l,k}d_l^{-1})^{-1}]}(a_{i;k}a_i^{-1})(s_{l,k}d_l^{-1})(a_{i-1;k}a_{i-1}^{-1})d_k^{-1} \\ \text{for }m=i-1,\\
		\{(s_{l,k}d_l^{-1})(a_{i;k}a_i^{-1})\}^{-1}(a_{i;k}a_i^{-1})\{(s_{l,k}d_l^{-1})(a_{i;k}a_i^{-1})\} \hspace{0.5cm}\text{for }m=i,\\
		(a_{i+1;k}a_{i+1}^{-1})^{-1}(s_{l,k}d_l^{-1})^{-1}(a_{i;k}a_i^{-1})d_k \hspace{0.5cm}\text{for }m=i+1,\\
		a_{i;k}a_i^{-1}\hspace{0.5cm}\text{for }m\geq{i+2},
		\end{array}
		\right.$
\item[(D1e)] $r_{m;l}(a_{i;k}a_i^{-1})r_{m;l}^{-1}=\\
		\left\{
		\begin{array}{ll}
		{[(s_{l,k}d_l^{-1})^{-1},r_{m;k}^{-1}]}^{-1}(a_{i;k}a_i^{-1}){[(s_{l,k}d_l^{-1})^{-1},r_{m;k}^{-1}]}\hspace{0.5cm}\text{for }m\leq{i-1},\\
		\{ r_{i;k}^{-1}(s_{l,k}d_l^{-1})^{-1}r_{i;k}\} (s_{l,k}d_l^{-1})(a_{i;k}a_i^{-1})\\
		(\bar{s}_{l,k;i}d_l^{-1})^{-1}\{ r_{i;k}^{-1}(s_{l,k}d_l^{-1})^{-1}r_{i;k}\} ^{-1}d_k^{-2}\hspace{0.5cm}\text{for }m=i,\\
		r_{i+1;k}^{-1}(s_{l,k}d_l^{-1})^{-1}r_{i+1;k}(\bar{s}_{l,k;i+1}d_l^{-1})(a_{i;k}a_i^{-1})d_k^2\hspace{0.5cm}\text{for }m=i+1,\\
		a_{i;k}a_i^{-1}\hspace{0.5cm}\text{for }m\geq{i+2},
		\end{array}
		\right.$
\item[(D1f)] $s_{l,t}(a_{i;k}a_i^{-1})s_{l,t}^{-1}=a_{i;k}a_i^{-1}$,
\item[(D1g)] $\bar{s}_{l,t}(a_{i;k}a_i^{-1})\bar{s}_{l,t}^{-1}=\\
		\left\{
		\begin{array}{ll}
		[(\bar{s}_{l,k}d_l^{-1})^{-1},s_{t,k}d_t^{-1}]^{-1}(s_{l,k}d_l^{-1})(a_{1;k}a_1^{-1})r_{1;k}(\bar{s}_{t,k}d_t^{-1})r_{1;k}^{-1}\\
(s_{l,k}d_l^{-1})^{-1}r_{1;k}(\bar{s}_{t,k}d_t^{-1})^{-1}r_{1;k}^{-1}[(\bar{s}_{l,k}d_l^{-1})^{-1},(s_{t,k}d_t^{-1})] \hspace{0.5cm}\text{for }i=1,\\
		\{[\bar{s}_{l,k}d_l^{-1},(s_{t,k}d_t^{-1})^{-1}][s_{l,k}d_l^{-1},r_{1;k}(\bar{s}_{t,k}d_t^{-1})^{-1}r_{1;k}^{-1}]\} (a_{i;k}a_i^{-1})\\
		\{[\bar{s}_{l,k}d_l^{-1},(s_{t,k}d_t^{-1})^{-1}][s_{l,k}d_l^{-1},r_{1;k}(\bar{s}_{t,k}d_t^{-1})^{-1}r_{1;k}^{-1}]\} ^{-1} \hspace{0.5cm}\text{for }i\geq2,
		\end{array}
		\right.$
\item[(D2a)] $a_mr_{i;k}a_m^{-1}=
		\left\{
		\begin{array}{ll}
		r_{i;k}r_{i-1;k}(a_{i-1;k}a_{i-1}^{-1})^{-1}r_{i;k}(a_{i-1;k}a_{i-1}^{-1})d_k\\
\text{for }m=i-1,\\
		(a_{i;k}a_i^{-1})^{-1}r_{i+1;k}^{-1}(a_{i;k}a_i^{-1})\hspace{0.5cm}\text{for }m=i,\\
		r_{i;k}\hspace{0.5cm}\text{for }m\neq{i-1,i},
		\end{array}
		\right.$
\item[(D2b)] $yr_{i;k}y^{-1}=
		\left\{
		\begin{array}{ll}
		\{ (a_{1;k}a_1^{-1})^{-1}r_{2;k}r_{1;k}\} ^{-1}r_{1;k}^{-1}\{ (a_{1;k}a_1^{-1})^{-1}r_{2;k}r_{1;k}\} \\
\text{for }i=1,\\
		(a_{1;k}a_1^{-1})r_{1;k}(a_{1;k}a_1^{-1})^{-1}r_{2;k}r_{1;k}d_{n-1}\hspace{0.5cm}\text{for }i=2,\\
		r_{i;k}\hspace{0.5cm}\text{for }i\geq 3,
		\end{array}
		\right.$
\item[(D2c)] $br_{i;k}b^{-1}=
		\left\{
		\begin{array}{ll}
		(a_{1;k}a_1^{-1})^{-1}(a_{3;k}a_3^{-1})^{-1}(a_{2;k}a_2^{-1})^{-1}\\
		r_{4;k}^{-1}(a_{3;k}a_3^{-1})r_{3;k}^{-1}(a_{2;k}a_2^{-1})r_{2;k}^{-1}(a_{1;k}a_1^{-1})\hspace{0.5cm}\text{for }i=1,\\
		\{ (a_{3;k}a_3^{-1})(a_{1;k}a_1^{-1})\} ^{-1}(a_{1;k}a_1^{-1})(a_{3;k}a_3^{-1})r_{2;k}\\
		r_{1;k}(a_{1;k}a_1^{-1})^{-1}r_{2;k}(a_{2;k}a_2^{-1})^{-1}r_{3;k}(a_{3;k}a_3^{-1})^{-1}\\
		r_{4;k}(a_{2;k}a_2^{-1})\{ (a_{3;k}a_3^{-1})(a_{1;k}a_1^{-1})\} \hspace{0.5cm}\text{for }i=2,\\
		\{ (a_{3;k}a_3^{-1})(a_{1;k}a_1^{-1})\} ^{-1}r_{4;k}^{-1}\\
		(a_{3;k}a_3^{-1})r_{3;k}^{-1}(a_{2;k}a_2^{-1})r_{2;k}^{-1}(a_{1;k}a_1^{-1})r_{1;k}^{-1}\\
		(a_{2;k}a_2^{-1})^{-1}r_{3;k}\\
		(a_{3;k}a_3^{-1})^{-1}(a_{1;k}a_1^{-1})^{-1}\{ (a_{3;k}a_3^{-1})(a_{1;k}a_1^{-1})\} \hspace{0.5cm}\text{for }i=3,\\
		r_{4;k}(a_{2;k}a_2^{-1})\\
		r_{1;k}(a_{1;k}a_1^{-1})^{-1}r_{2;k}(a_{2;k}a_2^{-1})^{-1}r_{3;k}(a_{3;k}a_3^{-1})^{-1}\\
		r_{4;k}(a_{3;k}a_3^{-1})(a_{1;k}a_1^{-1})\hspace{0.5cm}\text{for }i=4,\\
		r_{i;k}\hspace{0.5cm}\text{for }i\geq5,
		\end{array}
		\right.$
\item[(D2d)] $a_{m;l}r_{i;k}a_{m;l}^{-1}=\\
		\left\{
		\begin{array}{ll}
		{[(s_{l,k}d_l^{-1})^{-1},(a_{m;k}a_m^{-1})^{-1}]}^{-1}r_{i;k}{[(s_{l,k}d_l^{-1})^{-1},(a_{m;k}a_m^{-1})^{-1}]}\\
\text{for }m\leq{}i-2,\\
		\{ (s_{l,k}d_l^{-1})(a_{i-1;k}a_{i-1}^{-1})\} ^{-1}(a_{i-1;k}a_{i-1}^{-1})(s_{l,k}d_l^{-1})r_{i;k}\\
		r_{i-1;k}(a_{i-1;k}a_{i-1}^{-1})^{-1}r_{i;k}(\bar{s}_{l,k;i}d_l^{-1})\{ (s_{l,k}d_l^{-1})(a_{i-1;k}a_{i-1}^{-1})\} \\
\text{for }m=i-1,\\
		(a_{i;k}a_i^{-1})^{-1}(s_{l,k}d_l^{-1})^{-1}r_{i+1;k}^{-1}(a_{i;k}a_i^{-1})(\bar{s}_{l,k;i}d_l^{-1})^{-1}\hspace{0.5cm}\text{for }m=i,\\
		r_{i;k}\hspace{0.5cm}\text{for }m\geq{i+1},
		\end{array}
		\right.$
\item[(D2e)] $r_{m;l}r_{i;k}r_{m;l}^{-1}=
		\left\{
		\begin{array}{ll}
		{[(s_{l,k}d_l^{-1})^{-1},r_{m;k}^{-1}]}^{-1}r_{i;k}{[(s_{l,k}d_l^{-1})^{-1},r_{m;k}^{-1}]}\\
\text{for }m\leq{i-1},\\
		\{(s_{l,k}d_l^{-1})r_{i;k}\}^{-1}r_{i;k}\{(s_{l,k}d_l^{-1})r_{i;k}\}\hspace{0.5cm}\text{for }m=i,\\
		r_{i;k}\hspace{0.5cm}\text{for }m\geq{i+1},
		\end{array}
		\right.$
\item[(D2f)]	$s_{l,t}r_{i;k}s_{l,t}^{-1}=r_{i;k}$,
\item[(D2g)] $\bar{s}_{l,t}r_{i;k}\bar{s}_{l,t}^{-1}=\\
		\left\{
		\begin{array}{ll}
		[s_{t,k}d_t^{-1},(\bar{s}_{l,k}d_l^{-1})^{-1}][r_{1;k}(\bar{s}_{t,k}d_t^{-1})r_{1;k}^{-1},(s_{l,k}d_l^{-1})^{-1}]r_{1;k} \hspace{0.5cm}\text{for }i=1,\\
		\{ [\bar{s}_{l,k}d_l^{-1},(s_{t,k}d_t^{-1})^{-1}][s_{l,k}d_l^{-1},r_{1;k}(\bar{s}_{t,k}d_t^{-1})^{-1}r_{1;k}^{-1}]\} ^{-1}r_{1;k}\\
		\{ [\bar{s}_{l,k}d_l^{-1},(s_{t,k}d_t^{-1})^{-1}][s_{l,k}d_l^{-1},r_{1;k}(\bar{s}_{t,k}d_t^{-1})^{-1}r_{1;k}^{-1}]\}\hspace{0.5cm}\text{for }i\geq2,
		\end{array}
		\right.$
\item[(D3a)] $a_m(s_{j,k}d_j^{-1})a_m^{-1}=s_{j,k}d_j^{-1}$,
\item[(D3b)] $y(s_{j,k}d_j^{-1})y^{-1}=s_{j,k}d_j^{-1}$,
\item[(D3c)] $b(s_{j,k}d_j^{-1})b^{-1}=s_{j,k}d_j^{-1}$,
\item[(D3d)] $a_{m;l}(s_{j,k}d_j^{-1})a_{m;l}^{-1}=\\
		\left\{
		\begin{array}{ll}
		[(s_{l,k}d_l^{-1})^{-1},(a_{m;k}a_m^{-1})^{-1}]^{-1}(s_{j,k}d_j^{-1})[(s_{l,k}d_l^{-1})^{-1},(a_{m;k}a_m^{-1})^{-1}]\\
 \text{for }l>j,\\
		(a_{m;k}a_m^{-1})^{-1}(s_{j,k}d_j^{-1})(a_{m;k}a_m^{-1})\hspace{0.5cm}\text{for }l=j,\\
		s_{j,k}d_j^{-1}\hspace{0.5cm}\text{for }l=j,
		\end{array}
		\right.$
\item[(D3e)] $r_{m;l}(s_{j,k}d_j^{-1})r_{m;l}^{-1}=\\
		\left\{
		\begin{array}{ll}
		{[(s_{l,k}d_l^{-1})^{-1},r_{m;k}^{-1}]}^{-1}(s_{j,k}d_j^{-1}){[(s_{l,k}d_l^{-1})^{-1},r_{m;k}^{-1}]}\hspace{0.5cm}\text{for }l>j,\\
		r_{m;k}^{-1}(s_{j,k}d_j^{-1})r_{m;k}\hspace{0.5cm}\text{for }l=j,\\
		s_{j,k}d_j^{-1}\hspace{0.5cm}\text{for }l<j,
		\end{array}
		\right.$
\item[(D3f)] $s_{l,t}(s_{j,k}d_j^{-1})s_{l,t}^{-1}=\\
		\left\{
		\begin{array}{ll}
		\{(s_{t,k}d_t^{-1})(s_{j,k}d_j^{-1})\}^{-1}(s_{j,k}d_j^{-1})\{(s_{t,k}d_t^{-1})(s_{j,k}d_j^{-1})\}\hspace{0.5cm}\text{for }l=j,\\
		{[(s_{t,k}d_t^{-1})^{-1},(s_{l,k}d_l^{-1})^{-1}]}^{-1}(s_{j,k}d_j^{-1}){[(s_{t,k}d_t^{-1})^{-1},(s_{l,k}d_l^{-1})^{-1}]}\\
\text{for }l<j<t,\\
		(s_{l,k}d_l^{-1})^{-1}(s_{j,k}d_j^{-1})(s_{l,k}d_l^{-1})\hspace{0.5cm}\text{for }t=j,\\
		s_{j,k}d_j^{-1}\hspace{0.5cm}\text{for the other cases},
		\end{array}
		\right.$
\item[(D3g)] $\bar{s}_{l,t}(s_{j,k}d_j^{-1})\bar{s}_{l,t}^{-1}=\\
		\left\{
		\begin{array}{ll}
		\{[\bar{s}_{l,k}d_l^{-1},(s_{t,k}d_t^{-1})^{-1}][s_{l,k}d_l^{-1},r_{1;k}(\bar{s}_{t,k}d_t^{-1})^{-1}r_{1;k}^{-1}]\}^{-1}(s_{j,k}d_j^{-1})\\
		\{[\bar{s}_{l,k}d_l^{-1},(s_{t,k}d_t^{-1})^{-1}][s_{l,k}d_l^{-1},r_{1;k}(\bar{s}_{t,k}d_t^{-1})^{-1}r_{1;k}^{-1}]\}\hspace{0.5cm}\text{for }l>j,\\
		\{[\bar{s}_{j,k}d_j^{-1},(s_{t,k}d_t^{-1})^{-1}](s_{j,k}d_j^{-1})r_{1;k}(\bar{s}_{t,k}d_t^{-1})^{-1}r_{1;k}^{-1}\}^{-1}(s_{j,k}d_j^{-1})\\
		\{[\bar{s}_{j,k}d_j^{-1},(s_{t,k}d_t^{-1})^{-1}](s_{j,k}d_j^{-1})r_{1;k}(\bar{s}_{t,k}d_t^{-1})^{-1}r_{1;k}^{-1}\}\hspace{0.5cm}\text{for }l=j,\\
		{[(\bar{s}_{l,k}d_l^{-1})^{-1},s_{t,k}d_t^{-1}]}^{-1}(s_{j,k}d_j^{-1}){[(\bar{s}_{l,k}d_l^{-1})^{-1},s_{t,k}d_t^{-1}]}\hspace{0.5cm}\text{for }l<j<t,\\
		\{ (\bar{s}_{l,k}d_l^{-1})(s_{j,k}d_j^{-1})^{-1}\} ^{-1}(s_{j,k}d_j^{-1})\{ (\bar{s}_{l,k}d_l^{-1})(s_{j,k}d_j^{-1})^{-1}\} \hspace{0.5cm}\text{for }t=j,\\
		s_{j,k}d_j^{-1}\hspace{0.5cm}\text{for }t<j,
		\end{array}
		\right.$
\item[(D4a)] $a_m(\bar{s}_{j,k}d_j^{-1})a_m^{-1}=\\
		\left\{
		\begin{array}{ll}
		\{r_{1;k}^{-1}r_{2;k}^{-1}(a_{1;k}a_1^{-1})\}^{-1}(\bar{s}_{j,k}d_j^{-1})\{r_{1;k}^{-1}r_{2;k}^{-1}(a_{1;k}a_1^{-1})\}\hspace{0.5cm}\text{for }m=1,\\
		\bar{s}_{j,k}d_j^{-1}\hspace{0.5cm}\text{for }m\geq2,
		\end{array}
		\right.$
\item[(D4b)] $y(\bar{s}_{j,k}d_j^{-1})y^{-1}=\\
\{ r_{1;k}^{-1}(a_{1;k}a_1^{-1})^{-2}r_{2;k}r_{1;k}\} ^{-1}(\bar{s}_{j,k}d_j^{-1})\{ r_{1;k}^{-1}(a_{1;k}a_1^{-1})^{-2}r_{2;k}r_{1;k}\},$
\item[(D4c)] $b(\bar{s}_{j,k}d_j^{-1})b^{-1}=\\
\{r_{1;k}^{-1}(a_{2;k}a_2^{-1})^{-1}r_{4;k}^{-1}(a_{3;k}a_3^{-1})r_{3;k}^{-1}(a_{2;k}a_2^{-1})r_{2;k}^{-1}(a_{1;k}a_1^{-1})\}^{-1}(\bar{s}_{j,k}d_j^{-1})\\
		\{r_{1;k}^{-1}(a_{2;k}a_2^{-1})^{-1}r_{4;k}^{-1}(a_{3;k}a_3^{-1})r_{3;k}^{-1}(a_{2;k}a_2^{-1})r_{2;k}^{-1}(a_{1;k}a_1^{-1})\},$
\item[(D4d)] $a_{m;l}(\bar{s}_{j,k}d_j^{-1})a_{m;l}^{-1}=\\
		\left\{
		\begin{array}{ll}
		\{r_{1;k}^{-1}r_{2;k}^{-1}(a_{1;k}a_1^{-1})(\bar{s}_{l,k}d_l^{-1})^{-1}\}^{-1}(\bar{s}_{j,k}d_j^{-1})
		\{r_{1;k}^{-1}r_{2;k}^{-1}(a_{1;k}a_1^{-1})(\bar{s}_{l,k}d_l^{-1})^{-1}\} \\
\text{for }m=1,l<j,\\
		\{(\bar{s}_{l,k}d_l^{-1})^{-1}r_{1;k}^{-1}r_{2;k}^{-1}(a_{1;k}a_1^{-1})\}^{-1}(\bar{s}_{j,k}d_j^{-1})
		\{(\bar{s}_{l,k}d_l^{-1})^{-1}r_{1;k}^{-1}r_{2;k}^{-1}(a_{1;k}a_1^{-1})\}\\
\text{for }m=1,l>j,\\
		\{(a_{m-1;k}a_{m-1}^{-1})r_{m-1;k}\cdots(a_{2;k}a_2^{-1})r_{2;k}^{-1}(a_{1;k}a_1^{-1})\}^{-1}
		(\bar{s}_{j,k;m+1}d_j^{-1})\\
		\{(a_{m-1;k}a_{m-1}^{-1})r_{m-1;k}\cdots(a_{2;k}a_2^{-1})r_{2;k}^{-1}(a_{1;k}a_1^{-1})\}\hspace{0.5cm}\text{for }m\geq2,l=j,\\
		\{(a_{m-1;k}a_{m-1}^{-1})r_{m-1;k}\cdots(a_{2;k}a_2^{-1})r_{2;k}^{-1}(a_{1;k}a_1^{-1})\}^{-1}\\
		\{(\bar{s}_{l,k;m}d_l^{-1})^{-1}r_{m;k}^{-1}(\bar{s}_{l,k;m+1}d_l^{-1})\}^{-1}
		(\bar{s}_{j,k;m}d_j^{-1})\\
		\{(\bar{s}_{l,k;m}d_l^{-1})^{-1}r_{m;k}^{-1}(\bar{s}_{l,k;m+1}d_l^{-1})\}\\
		\{(a_{m-1;k}a_{m-1}^{-1})r_{m-1;k}\cdots(a_{2;k}a_2^{-1})r_{2;k}^{-1}(a_{1;k}a_1^{-1})\}\hspace{0.5cm}\text{for }m\geq2,l>j,\\
		\bar{s}_{j,k}d_j^{-1}\hspace{0.5cm}\text{for the other cases},
		\end{array}
		\right.$
\item[(D4e)] $r_{m;l}(\bar{s}_{j,k}d_j^{-1})r_{m;l}^{-1}=\\
		\left\{
		\begin{array}{ll}
		\{r_{1;k}^{-1}(\bar{s}_{l,k}d_l^{-1})^{-1}(s_{l,k}d_l^{-1})r_{1;k}\}^{-1}(\bar{s}_{j,k}d_j^{-1})\\
		\{r_{1;k}^{-1}(\bar{s}_{l,k}d_l^{-1})^{-1}(s_{l,k}d_l^{-1})r_{1;k}\}\hspace{0.5cm}\text{for }m=1,l<j,\\
		\{(s_{j,k}d_j^{-1})r_{1;k}\}^{-1}(\bar{s}_{j,k}d_j^{-1})\{(s_{j,k}d_j^{-1})r_{1;k}\}\hspace{0.5cm}\text{for }m=1,l=j,\\
		\{(\bar{s}_{l,k}d_l^{-1})^{-1}r_{1;k}^{-1}(s_{l,k}d_l^{-1})r_{1;k}\}^{-1}(\bar{s}_{j,k}d_j)^{-1}\\
		\{(\bar{s}_{l,k}d_l^{-1})^{-1}r_{1;k}^{-1}(s_{l,k}d_l^{-1})r_{1;k}\} \hspace{0.5cm}\text{for }m=1,l>j,\\
		\{(a_{m-1;k}a_{m-1}^{-1})r_{m-1;k}\cdots(a_{2;k}a_2^{-1})r_{2;k}^{-1}(a_{1;k}a_1^{-1})\}^{-1}\\
		(\bar{s}_{j,k;m}d_j^{-1})\\
		\{(a_{m-1;k}a_{m-1}^{-1})r_{m-1;k}\cdots(a_{2;k}a_2^{-1})r_{2;k}^{-1}(a_{1;k}a_1^{-1})\}\hspace{0.5cm}\text{for }m\geq2,l=j,\\
		\{(a_{m-1;k}a_{m-1}^{-1})r_{m-1;k}\cdots(a_{2;k}a_2^{-1})r_{2;k}^{-1}(a_{1;k}a_1^{-1})\}^{-1}\\
		\{(\bar{s}_{l,k;m}d_l^{-1})^{-1}r_{m;k}^{-1}(\bar{s}_{l,k;m}d_l^{-1})\}^{-1}\\
		(\bar{s}_{j,k;m}d_j^{-1})\\
		\{(\bar{s}_{l,k;m}d_l^{-1})^{-1}r_{m;k}^{-1}(\bar{s}_{l,k;m}d_l^{-1})\}\\
		\{(a_{m-1;k}a_{m-1}^{-1})r_{m-1;k}\cdots(a_{2;k}a_2^{-1})r_{2;k}^{-1}(a_{1;k}a_1^{-1})\}\hspace{0.5cm}\text{for }m\geq2,l>j,\\
		\bar{s}_{j,k}d_j^{-1}\hspace{0.5cm}\text{for }m\geq2,l<j,
		\end{array}
		\right.$
\item[(D4f)] $s_{l,t}(\bar{s}_{j,k}d_j^{-1})s_{l,t}^{-1}=\\
		\left\{
		\begin{array}{ll}
		(\bar{s}_{l,k}d_l^{-1})^{-1}(\bar{s}_{j,k}d_j^{-1})(\bar{s}_{l,k}d_l^{-1})\hspace{0.5cm}\text{for }t=j,\\
		{[(\bar{s}_{t,k}d_t^{-1})^{-1},(\bar{s}_{l,k}d_l^{-1})^{-1}]}^{-1}(\bar{s}_{j,k}d_j^{-1})\\
		{[(\bar{s}_{t,k}d_t^{-1})^{-1},(\bar{s}_{l,k}d_l^{-1})^{-1}]}\hspace{0.5cm}\text{for }l<j<t,\\
		\{(\bar{s}_{t,k}d_t^{-1})(\bar{s}_{j,k}d_j^{-1})\}^{-1}(\bar{s}_{j,k}d_j^{-1})\{(\bar{s}_{t,k}d_t^{-1})(\bar{s}_{j,k}d_j^{-1})\}\hspace{0.5cm}\text{for }l=j,\\
		\bar{s}_{j,k}d_j^{-1}\hspace{0.5cm}\text{for the other cases},
		\end{array}
		\right.$
\item[(D4g)] $\bar{s}_{l,t}(\bar{s}_{j,k}d_j^{-1})\bar{s}_{l,t}^{-1}=\\
		\left\{
		\begin{array}{ll}
		{[(\bar{s}_{t,k}d_t^{-1}),r_{1;k}^{-1}(s_{l,k}d_l^{-1})^{-1}r_{1;k}]}^{-1}(\bar{s}_{j,k}d_j^{-1})\\
		{[(\bar{s}_{t,k}d_t^{-1}),r_{1;k}^{-1}(s_{l,k}d_l^{-1})^{-1}r_{1;k}]} \hspace{0.5cm}\text{for }t<j,\\
		\{r_{1;k}^{-1}(s_{l,k}d_l^{-1})^{-1}r_{1;k}\}(\bar{s}_{j,k}d_j^{-1})\{r_{1;k}^{-1}(s_{l,k}d_l^{-1})^{-1}r_{1;k}\}^{-1}\hspace{0.5cm}\text{for }t=j,\\
		(s_{t,k}d_t^{-1})(\bar{s}_{j,k}d_j^{-1})(s_{t,k}d_t^{-1})^{-1}\hspace{0.5cm}\text{for }l=j,\\
		{[(\bar{s}_{l,k}d_l^{-1})^{-1},(s_{t,k}d_t^{-1})]}^{-1}(\bar{s}_{j,k}d_j^{-1}){[(\bar{s}_{l,k}d_l^{-1})^{-1},(s_{t,k}d_t^{-1})]}\hspace{0.5cm}\text{for }l>j,\\
		\bar{s}_{j,k}d_j^{-1}\hspace{0.5cm}\text{for }l<j<t.
		\end{array}
		\right.$
\end{enumerate}
\end{prop}

\section{Infinite presentation for $\mathcal{M}(N_{g,n})$}\label{section-mainthm}

The main theorem in this paper is as follows:

\begin{thm}\label{main-thm}
For $g\geq 1$ and $n\geq 0$, $\mathcal{M}(N_{g,n})$ has the following presentation:

generators: $\{ t_{c;+_c}, t_{c;-_c} \mid c:\hspace{0.05cm} \text{two-sided s.c.c. on }N_{g,n} \}$\\
\hspace{2.0cm}$\cup \{ Y_{\mu ,\alpha } \mid \mu :\hspace{0.05cm} \text{one-sided s.c.c. on } N_{g,n},\ \alpha :\hspace{0.05cm} \text{s.c.c. on } N_{g,n},\ |\mu \cap \alpha |=1\}$.\\
Denote the generating set by $X$. 

relations:
\begin{enumerate}
 \item[(0)] 
\begin{enumerate}
\item[(i)] $t_{c;\theta _c}=1$ when $\theta _c\in \{ +_c,-_c\}$ and $c$ bounds a disk or a M\"{o}bius band in $N_{g,n}$,
\item[(ii)] $t_{c;+_c}=t_{c^{-1};+_c}=t_{c;-_c}^{-1}$,
\item[(iii)] $Y_{\mu ,\alpha }=Y_{\mu ,\alpha ^{-1}}^{-1}=Y_{\mu ^{-1},\alpha }$,
\end{enumerate}
 \item[(I)] All the braid relations\\
 \[
\left\{ \begin{array}{lll}
(i) &t_{f(c);\theta ^\prime }=ft_{c;\theta }f^{-1}& \text{when }f_\ast (\theta )=\theta ^\prime \text{ and }f\in X,\\
&t_{f(c);\theta ^\prime }^{-1}=ft_{c;\theta }f^{-1}& \text{when }f_\ast (\theta )\not =\theta ^\prime \text{ and }f\in X,  \\
(ii) &Y_{f(\mu ),\beta }=fY_{\mu ,\alpha }f^{-1}& \text{when }f(\alpha )=\beta \text{ and }f\in X,\\
&Y_{f(\mu ),\beta }^{-1}=fY_{\mu ,\alpha }f^{-1}& \text{when }f(\alpha )=\beta ^{-1} \text{ and }f\in X,
 \end{array} \right.
 \]
 \item[(I\hspace{-0.04cm}I)] All the 2-chain relations,
 \item[(I\hspace{-0.04cm}I\hspace{-0.04cm}I)] All the lantern relations,
 \item[(I\hspace{-0.04cm}V)] All the relations in Lemma~\ref{pushing3},
 i.e. $Y_{\mu ,\alpha \beta }=Y_{\mu ,\alpha }Y_{\mu ,\beta }$,
 \item[(V)] All the relations in Lemma~\ref{pushing2},
 i.e. $Y_{\mu ,\alpha }=t_{\delta _1;\theta _{\delta _1}}^{\varepsilon _{\delta _1}}t_{\delta _2;\theta _{\delta _2}}^{-\varepsilon _{\delta _2}}$.
\end{enumerate} 
\end{thm}

The second author~\cite{Omori} proved Theorem~\ref{main-thm} when $g\geq 1$ and $n\in \{ 0,1\}$. Since we do not distinguish $t_{c;+_c}$, $t_{c^{-1};+_c}$ and $t_{c;-_c}^{-1}$, and also do not distinguish $Y_{\mu ,\alpha }$, $Y_{\mu ,\alpha ^{-1}}^{-1}$ and $Y_{\mu ^{-1},\alpha }$ in \cite{Omori}, the presentation in Theorem~3.1 of \cite{Omori} is different from the presentation in Theorem~\ref{main-thm}. However, these presentation are equivalent by Relation~(0)(ii) and (0)(iii). In fact, we can apply the proof of Theorem~3.1 in \cite{Omori} to the presentation in Theorem~\ref{main-thm}. In (I) and (I\hspace{-0.04cm}V) one can substitute the right hand side of (V) for each generator $Y_{\mu ,\alpha }$ with one-sided $\alpha $. Then one can remove the generators $Y_{\mu ,\alpha }$ with one-sided $\alpha $ and relations~(V) from the presentation.

We denote by $G$ the group which has the presentation in Theorem~\ref{main-thm}. Let $\iota :\Sigma _{h,m}\hookrightarrow N_{g,n}$ be a smooth embedding and let $G^\prime $ be the group whose presentation has all Dehn twists along simple closed curves on $\Sigma _{h,m}$ as generators and Relations~(0$^\prime $), (I$^\prime $), (I\hspace{-0.04cm}I) and (I\hspace{-0.04cm}I\hspace{-0.04cm}I) in Theorem~\ref{pres_Gervais}. By Theorem~\ref{pres_Gervais}, $\mathcal{M}(\Sigma _{h,m})$ is isomorphic to $G^\prime $, and we have the homomorphism $G^\prime \rightarrow G$ defined by the correspondence of $t_{c;+_c}$ to $t_{\iota (c);\iota _\ast(+_c)}$.
 Then we remark the following.  

\begin{rem}\label{orientable}
The composition $\iota _\ast :\mathcal{M}(\Sigma _{h,m})\rightarrow G$ of the isomorphism $\mathcal{M}(\Sigma _{h,m})\rightarrow G^\prime $ and the homomorphism $G^\prime \rightarrow G$ is a homomorphism. In particular, if a product $t_{c_1;\theta _{c_1}}^{\varepsilon _1}t_{c_2;\theta _{c_2}}^{\varepsilon _2}\cdots t_{c_k;\theta _{c_k}}^{\varepsilon _k}$ of Dehn twists along simple closed curves $c_1$, $c_2$, $\dots $, $c_k$ on a connected compact orientable subsurface of $N_{g,n}$ is equal to the identity map in the mapping class group of the subsurface, then $t_{c_1;\theta _{c_1}}^{\varepsilon _1}t_{c_2;\theta _{c_2}}^{\varepsilon _2}\cdots t_{c_k;\theta _{c_k}}^{\varepsilon _k}$ is equal to $1$ in $G$. That means such a relation $t_{c_1;\theta _{c_1}}^{\varepsilon _1}t_{c_2;\theta _{c_2}}^{\varepsilon _2}\cdots t_{c_k;\theta _{c_k}}^{\varepsilon _k}=1$ is obtained from Relations~(0), (I), (I\hspace{-0.04cm}I) and (I\hspace{-0.04cm}I\hspace{-0.04cm}I).
\end{rem}

Set $X^\pm :=X\cup \{ x^{-1} \mid x\in X \}$. By Relation~(I), we have the following lemma.

\begin{lem}\label{braid3}
For $f\in G$, suppose that $f=f_1f_2\dots f_k$, where $f_1$, $f_2$, $\dots $, $f_k\in X^\pm $. Then we have
 \[
\left\{ \begin{array}{lll}
(i) &t_{f(c);\theta ^\prime }=ft_{c;\theta }f^{-1}& \text{when }f_\ast (\theta )=\theta ^\prime ,\\
&t_{f(c);\theta ^\prime }^{-1}=ft_{c;\theta }f^{-1}& \text{when }f_\ast (\theta )\not =\theta ^\prime ,  \\
(ii) &Y_{f(\mu ),\beta }=fY_{\mu ,\alpha }f^{-1}& \text{when }f(\alpha )=\beta,\\
&Y_{f(\mu ),\beta }^{-1}=fY_{\mu ,\alpha }f^{-1}& \text{when }f(\alpha )=\beta ^{-1}.
 \end{array} \right.
 \]
\end{lem}

The next lemma follows from an argument of the combinatorial group theory (for instance, see \cite[Lemma~4.2.1, p42]{Johnson1}).

\begin{lem}\label{combinatorial}
For groups $\Gamma $, $\Gamma ^\prime $ and $F$, a surjective homomorphism $\pi :F\rightarrow \Gamma $ and a homomorphism $\nu :F\rightarrow \Gamma ^\prime $, we define a map $\nu ^\prime :\Gamma \rightarrow \Gamma ^\prime $ by $\nu ^\prime (x):=\nu (\widetilde{x})$ for $x\in \Gamma $, where $\widetilde{x}\in F$ is a lift of $x$ with respect to $\pi $ (see the diagram below).

Then if $\ker \pi \subset \ker \nu $, $\nu ^\prime $ is well-defined and a homomorphism.
\end{lem}

\[
\xymatrix{
 F \ar@{->>}[d]_\pi \ar[dr]^{\nu } &  \\
\Gamma  \ar@{-->}[r]_{\nu ^\prime}  & \Gamma ^\prime  \\
}
\]

We start the proof of Theorem~\ref{main-thm}.
When $n\in \{ 0,1\}$, we proved Theorem~\ref{main-thm} in \cite{Omori}. Assume $g\geq 1$ and $n\geq 2$. 
Then we obtain Theorem~\ref{main-thm} if $\mathcal{M}(N_{g,n})$ is isomorphic to $G$. Let $\varphi :G\rightarrow \mathcal{M}(N_{g,n})$ be the surjective homomorphism defined by $\varphi (t_{c;+_c}):=t_{c;+_c}$, $\varphi (t_{c;-_c}):=t_{c;-_c}$ and $\varphi (Y_{\mu ,\alpha }):=Y_{\mu ,\alpha }$. 

Denote by $X_0\subset \mathcal{M}(N_{g,n})$ the generating set of the finite presentation for $\mathcal{M}(N_{g,n})$ in Proposition~\ref{thm_finite_presentation}. Let $F(X_0)$ be the free group which is freely generated by $X_0$ and let $\pi :F(X_0)\rightarrow \mathcal{M}(N_{g,n})$ be the natural projection. We define the homomorphism $\nu :F(X_0)\rightarrow G$ by $\nu (a_i):=a_i$, $\nu (b):=b$, $\nu (y):=y$, $\nu (a_{i;j}):=a_{i;j}$, $\nu (r_{i;j}):=r_{i;j}$, $\nu (s_{i,j}):=s_{i,j}$ and $\nu (\bar{s}_{i,j}):=\bar{s}_{i,j}$, and a map $\psi =\nu ^\prime :\mathcal{M}(N_{g,n})\rightarrow G$ by $\psi (a_i^{\pm 1}):=a_i^{\pm 1}$, $\psi (b^{\pm 1}):=b^{\pm 1}$, $\psi (y^{\pm 1}):=y^{\pm 1}$, $\psi (a_{i;j}^{\pm 1}):=a_{i;j}^{\pm 1}$, $\psi (r_{i;j}^{\pm 1}):=r_{i;j}^{\pm 1}$, $\psi (s_{i,j}^{\pm 1}):=s_{i,j}^{\pm 1}$, $\psi (\bar{s}_{i,j}^{\pm 1}):=\bar{s}_{i,j}^{\pm 1}$ and $\psi (f):=\nu (\widetilde{f})$ for the other $f\in \mathcal{M}(N_{g,n})$, where $\widetilde{f}\in F(X_0)$ is a lift of $f$ with respect to $\pi $ (see the diagram below).

\[
\xymatrix{
 F(X_0) \ar@{->>}[d]_\pi \ar[dr]^{\nu } &  \\
\mathcal{M}(N_{g,n}) \ar@{-->}[r]_{\psi }  & G  \\
}
\]

If $\psi $ is a homomorphism, $\varphi \circ \psi ={\rm id}_{\mathcal{M}(N_{g,n})}$ by the definition of $\varphi $ and $\psi $. Thus it is sufficient to show that $\psi $ is a homomorphism and surjective for proving that $\psi $ is isomorphism.

\subsection{Proof that $\psi $ is a homomorphism}

By Lemma~\ref{combinatorial}, if the relations of the presentation in Proposition~\ref{thm_finite_presentation} are obtained from Relations~(0), (I), (I\hspace{-0.04cm}I), (I\hspace{-0.04cm}I\hspace{-0.04cm}I), (I\hspace{-0.04cm}V) and (V), then $\psi $ is a homomorphism. 

Let $N$ be the subsurface of $N_{g,n}$ as in Figure~\ref{nonorisurf_emb1_boundary}. $N$ is diffeomorphic to $N_{g,1}$ and includes simple closed curves $\alpha _1, \dots , \alpha _{g-1}$, $\mu $ and $\beta $. We regard $\mathcal{M}(N)$ as a subgroup of $\mathcal{M}(N_{g,n})$. Relations~(A1), $\dots $, (A9b) and (B1), $\dots $, (B8) of the presentation for $\mathcal{M}(N_{g,n})$ in Proposition~\ref{thm_finite_presentation} are relations of $\mathcal{M}(N)\cong \mathcal{M}(N_{g,1})$. By Theorem~3.1 in \cite{Omori}, Relations~(A1), $\dots $, (A9b) and (B1), $\dots $, (B8) are obtained from Relations~(0), (I), (I\hspace{-0.04cm}I), (I\hspace{-0.04cm}I\hspace{-0.04cm}I), (I\hspace{-0.04cm}V) and (V).

By arguments in the last part of Section~\ref{proof_finite_presentation} and Section~\ref{section_epsilon_r}, Relations~(D0)-(D4g) are obtained from Relations~(I) and (I\hspace{-0.04cm}I\hspace{-0.04cm}I). 
We have proved that $\psi $ is a homomorphism.
 
\begin{figure}[h]
\includegraphics[scale=0.65]{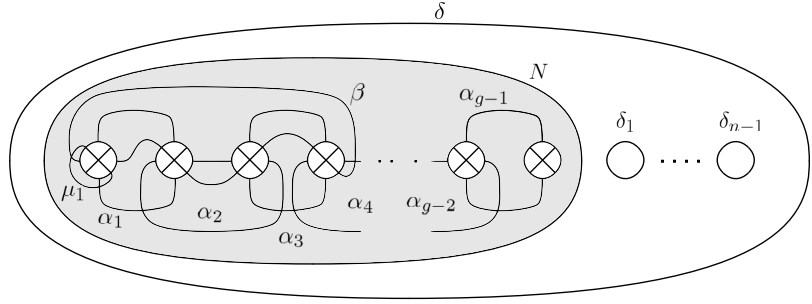}
\caption{The subsurface $N$ of $N_{g,n}$ which is diffeomorphic to $N_{g,1}$.}\label{nonorisurf_emb1_boundary}
\end{figure}

\subsection{Surjectivity of $\psi $}

For some convenience, we write $t_{c;+_c}=t_c$ in this subsection. We show that there exist lifts of $t_c$'s and $Y_{\mu ,\alpha }$'s with respect to $\psi $ for cases below, to prove the surjectivity of $\psi $. 
\begin{itemize}
\item[(1)] $t_c$; $c$ is non-separating and $N_{g,n}-c$ is non-orientable,
\item[(2)] $t_c$; $c$ is non-separating and $N_{g,n}-c$ is orientable,
\item[(3)] $t_c$; $c$ is separating,
\item[(4)] $Y_{\mu ,\alpha }$; $\alpha $ is two-sided and $N_{g,n}-\alpha $ is non-orientable,
\item[(5)] $Y_{\mu ,\alpha }$; $\alpha $ is two-sided and $N_{g,n}-\alpha $ is orientable,
\item[(6)] $Y_{\mu ,\alpha }$; $\alpha $ is one-sided.
\end{itemize}

Set $X_0^\pm :=X_0\cup \{ x^{-1} \mid x\in X_0 \}$. For a simple closed curve $c$ on $N_{g,n}$, we denote by $(N_{g,n})_c$ the surface obtained from $N_{g,n}$ by cutting $N_{g,n}$ along $c$ and denote by $\Sigma $ the component of $(N_{g,n})_c$ which does not include $\delta $. 


For generators of type~(1), (2), (4), (5), (6), by similar arguments in Section~3.2 of \cite{Omori}, there exist their lifts with respect to $\psi $. We note that we use the existence for lifts of generators of type~(3) for the proof of the existence for lifts of generators of type~(6).  

\textbf{Case~(3) where $\Sigma $ is diffeomorphic to $\Sigma _{0,m+1}$ for $m\geq 0$.} 
We proceed by induction on $m\geq 0$. When $m=0$, $t_c$ is trivial by Relation~(0). When $m=1$, $c=\delta _i$ for some $1\leq i\leq b-1$. Hence $d_i$ is the lift of $t_c$. 

When $m=2$, there exists a product $f=f_1f_2\cdots f_k\in \mathcal{M}(N_{g,n})$ of $f_1$, $f_2$, $\cdots $, $f_k \in X_0^\pm $ which satisfy either $c=f(\sigma _{i,j})$ or $c=f(\bar{\sigma }_{i,j})$ for some $1\leq i<j\leq b-1$. 
Thus, if $c=f(\sigma _{i,j})$, we have
\begin{eqnarray*}
\psi (fs_{i,j}f^{-1})
&=& f_1 f_2\cdots f_ks_{i,j}f_k^{-1}\cdots f_2^{-1}f_1^{-1}\\
&\stackrel{\text{Lem.~\ref{braid3}}}{=}& t_{f(\sigma _{i,j})}^\varepsilon \\
&=& t_c^\varepsilon ,
\end{eqnarray*}
where $\varepsilon $ is $1$ or $-1$. Thus $fs_{i,j}^\varepsilon f^{-1}\in \mathcal{M}(N_{g,n})$ is a lift of $t_c\in G$ with respect to $\psi $ for some $\varepsilon \in \{ -1,1\}$. By a similar argument, when $c=f(\bar{\sigma }_{i,j})$, $f\bar{s}_{i,j}^\varepsilon f^{-1}\in \mathcal{M}(N_{g,n})$ is also a lift of $t_c\in G$ with respect to $\psi $ for some $\varepsilon \in \{ -1,1\}$.

For $m\geq 3$, there exists a simple closed curve $c^\prime $ on $\Sigma $ such that $c^\prime $ separates $\Sigma $ into $\Sigma ^\prime $ and $\Sigma ^{\prime \prime }$ which are diffeomorphic to $\Sigma _{0,4}$ and $\Sigma _{0,m-1}$, respectively, and $c\subset \Sigma ^\prime $. By using a lantern relation on $\Sigma ^\prime $, there exist simple closed curves $c_1=c^\prime ,c_2,\dots ,c_6$ on $\Sigma ^\prime $ such that $t_c=t_{c_1}^{\varepsilon _1}t_{c_2}^{\varepsilon _2}\cdots t_{c_6}^{\varepsilon _6}\in G$ for some $\varepsilon _1, \varepsilon _2,\dots ,\varepsilon _6\in \{ -1,1\}$. Since each $c_i$ $(i=1,2,\dots ,6)$ bounds a subsurface of $N_{g,n}$ which does not include $c$ and is diffeomorphic to $\Sigma _{0,m_i+1}$ for some $m_i<m$, by the inductive assumption, there exist lifts $h_1,\dots ,h_6\in \mathcal{M}(N_{g,n})$ of $t_{c_1},\dots ,t_{c_6}\in G$ with respect to $\psi $, respectively. Thus $h_1^{\varepsilon _1}h_2^{\varepsilon _2}\cdots h_6^{\varepsilon _6}\in \mathcal{M}(N_{g,n})$ is a lift of $t_c$ with respect to $\psi $.

\textbf{Case~(3) where $\Sigma $ is diffeomorphic to $\Sigma _{h,m+1}$ for $h\geq 1$ $m\geq 0$.}  
In this case, there exists a simple closed curve $c^\prime $ on $\Sigma $ such that $c^\prime $ separates $\Sigma $ into $\Sigma ^\prime $ and $\Sigma ^{\prime \prime }$ which are diffeomorphic to $\Sigma _{h,2}$ and $\Sigma _{0,m+1}$, respectively. Then there exists a $2h+1$-chain $c_1$, $c_2$, $\dots $, $c_{2h+1}$ on $\Sigma ^\prime $ such that $\mathcal{N}_{N_{g,n}}(c_1\cup c_2\cup \cdots \cup c_{2h+1})=\Sigma ^\prime $. By the chain relation, $(t_{c_1}^{\varepsilon _1}t_{c_2}^{\varepsilon _2}\cdots t_{c_{2h+1}}^{\varepsilon _{2h+1}})^{2h+2}=t_ct_{c^\prime }^{\varepsilon ^\prime }$ for some $\varepsilon _1$, $\varepsilon _2$, $\dots $, $\varepsilon _{2h+1}, \varepsilon ^\prime \in \{ -1,1\}$. Since $t_{c_1}$, $t_{c_2}$, $\dots $, $t_{c_{2h+1}}$ are Dehn twists of type~(1) and $c^\prime $ bounds $\Sigma ^{\prime \prime }$, $t_{c_1}$, $t_{c_2}$, $\dots $, $t_{c_{2h+1}}, t_{c^\prime }\in G$ have lifts $h_1$, $h_2$, $\dots $, $h_{2h+1},h^\prime \in \mathcal{M}(N_{g,n})$ with respect to $\psi $, respectively. Thus we have
\begin{eqnarray*}
\psi ((h_1^{\varepsilon _1}h_2^{\varepsilon _2}\dots h_{2h+1}^{\varepsilon _{2h+1}})^{2h+2}(h^\prime )^{-\varepsilon ^\prime })&=& (t_{c_1}^{\varepsilon _1}t_{c_2}^{\varepsilon _2}\cdots t_{c_{2h+1}}^{\varepsilon _{2h+1}})^{2h+2}t_{c^\prime }^{-\varepsilon ^\prime }\\
&\stackrel{(0),(\text{I}),(\text{I\hspace{-0.04cm}I}),(\text{I\hspace{-0.04cm}I\hspace{-0.04cm}I})}{=}& t_c.
\end{eqnarray*}
We remark that the last relation is obtained from Theorem~\ref{pres_Gervais} and Remark~\ref{orientable}. Thus $((h_1^{\varepsilon _1}h_2^{\varepsilon _2}\dots h_{2h+1}^{\varepsilon _{2h+1}})^{2h+2}(h^\prime )^{-\varepsilon ^\prime }\in \mathcal{M}(N_{g,n})$ is a lift of $t_c\in G$ with respect to $\psi $.

\textbf{Case~(3) where $\Sigma $ is diffeomorphic to $N_{h,m+1}$ for $h\geq 1$ $m\geq 0$.}
We proceed by induction on $m\geq 0$. When $m=0$, by similar arguments in Section~3.2 in \cite{Omori}, there exists their a lift of $t_c\in G$ with respect to $\psi $.

When $m=1$, we proceed by induction on $h\geq 1$. When $h=1$, there exists a product $f=f_1f_2\cdots f_k\in \mathcal{M}(N_{g,n})$ of $f_1$, $f_2$, $\cdots $, $f_k \in X_0^\pm $ such that $c=f(\rho _{1;j})$ for some $1\leq j\leq n-1$. By a similar argument in the case where $\Sigma $ is diffeomorphic to $\Sigma _{0,m+1}$, we can obtain a lift of $t_c$ with respect to $\psi $. Suppose $h\geq 2$. Then there exist simple closed curves $c_1$ and $c_2$ on $\Sigma $ such that $c_1\sqcup c_2$ separates $\Sigma $ into $\Sigma ^\prime $, $\Sigma ^{\prime \prime }$ and $\Sigma ^{\prime \prime \prime }$ which are diffeomorphic to $\Sigma _{0,4}$, $N_{1,1}$ and $N_{h-1,1}$, respectively, and $c\subset \Sigma ^\prime $. By using a lantern relation on $\Sigma ^\prime $, there exist simple closed curves $c_3,\dots ,c_6$ on $\Sigma ^\prime $ such that $t_c=t_{c_1}^{\varepsilon _1}t_{c_2}^{\varepsilon _2}t_{c_3}^{\varepsilon _3}\cdots t_{c_6}^{\varepsilon _6}\in G$ for some $\varepsilon _1, \varepsilon _2,\varepsilon _3,\dots ,\varepsilon _6\in \{ -1,1\}$. Since each $c_i$ $(i=1,\dots ,6)$ is a boundary component of a subsurface of $\Sigma $ which is diffeomorphic to an orientable surface, $N_{h_i,1}$ for some $h_i\leq h$ or $N_{h_i,2}$ for some $h_i<h$, by the inductive assumption, there exist lifts $h_1,\dots ,h_6\in \mathcal{M}(N_{g,n})$ of $t_{c_1},\dots ,t_{c_6}\in G$ with respect to $\psi $, respectively. Thus $h_1^{\varepsilon _1}h_2^{\varepsilon _2}\cdots h_6^{\varepsilon _6}\in \mathcal{M}(N_{g,n})$ is a lift of $t_c$ with respect to $\psi $.

Suppose $m\geq 2$. Then there exist simple closed curves $c_1$ and $c_2$ on $\Sigma $ such that $c_1\sqcup c_2$ separates $\Sigma $ into $\Sigma ^\prime $, $\Sigma ^{\prime \prime }$ and $\Sigma ^{\prime \prime \prime }$ which are diffeomorphic to $\Sigma _{0,4}$, $\Sigma _{0,m}$ and $N_{h,1}$, respectively, and $c\subset \Sigma ^\prime $.
By using a lantern relation on $\Sigma ^\prime $, there exist simple closed curves $c_3,\dots ,c_6$ on $\Sigma ^\prime $ such that $t_c=t_{c_1}^{\varepsilon _1}t_{c_2}^{\varepsilon _2}t_{c_3}^{\varepsilon _3}\cdots t_{c_6}^{\varepsilon _6}\in G$ for some $\varepsilon _1, \varepsilon _2,\varepsilon _3,\dots ,\varepsilon _6\in \{ -1,1\}$. Since each $c_i$ $(i=1,\dots ,6)$ is a boundary component of a subsurface of $\Sigma $ which is diffeomorphic to an orientable surface or $N_{h,m_i+1}$ for some $m_i<m$, by the inductive assumption, there exist lifts $h_1,\dots ,h_6\in \mathcal{M}(N_{g,n})$ of $t_{c_1},\dots ,t_{c_6}\in G$ with respect to $\psi $, respectively. Thus $h_1^{\varepsilon _1}h_2^{\varepsilon _2}\cdots h_6^{\varepsilon _6}\in \mathcal{M}(N_{g,n})$ is a lift of $t_c$ with respect to $\psi $.


We have completed the proof of Theorem~\ref{main-thm}.

\section{Proof of Proposition~\ref{thm_finite_presentation} and preliminaries for the proof}\label{section_finite_presentation}

In this section, we give a proof of Proposition~\ref{thm_finite_presentation} which is used in the proof of Theorem~\ref{main-thm}. The proof is given in Section~\ref{proof_finite_presentation} and \ref{section_epsilon_r}. For giving the proof, we prepare Section~\ref{subsection_presentation_groupextention}, \ref{subsection_extended_lantern} and \ref{subsection_gen_pi1plus}.

\subsection{Group presentations and short exact sequence}\label{subsection_presentation_groupextention}

Let $G$ be a group and let $H=\bigl< X\mid R\bigr>$, $Q=\bigl< Y\mid S\bigr>$ be presented groups which have the exact sequence 
\[
1\longrightarrow H\stackrel{\iota }{\longrightarrow }G\stackrel{\nu }{\longrightarrow }Q\longrightarrow 1.
\]
We take a lift $\tilde{y}\in G$ of $y\in Q$ with respect to $\nu $ for each $y\in Q$. Then we put $\widetilde{X}:=\{ \iota (x) \mid x\in X\} \subset G$ and $\widetilde{Y}:=\{ \tilde{y} \mid y\in Y\} \subset G$. Denote by $\tilde{r}$ the word in $\widetilde{X}$ which obtained from $r\in R$ by replacing each $x\in X$ by $\iota (x)$ and denote by $\tilde{s}$ the word in $\widetilde{Y}$ which obtained from $s\in S$ by replacing each $y\in Y$ by $\tilde{y}$. We note that $\tilde{r}=1$ in $G$. For each $s\in S$, since $\tilde{s}\in G$ is an element of $\ker \nu $, there exists a word $v_s$ in $\widetilde{X}$ such that $\tilde{s}=v_s$ in $G$. Since $\iota (H)$ is a normal subgroup of $G$, for each $x\in X$ and $y\in Y$, $\tilde{y}\iota (x)\tilde{y}^{-1}$ is an element of $\iota (H)$. Hence there exists a word $w_{x,y}$ in $\widetilde{X}$ such that $\tilde{y}\iota (x)\tilde{y}^{-1}=w_{x,y}$ in $G$. The next lemma follows from an argument of the combinatorial group theory (for instance, see \cite[Proposition~10.2.1, p139]{Johnson1}).
\begin{lem}\label{presentation_exact}
In this situation above, the group $G$ has the following presentation:\\
generators: $\{ \iota (x), \tilde{y} \mid x\in X, y\in Y\}$.\\
relations:
\begin{enumerate}
 \item[(A)] $\tilde{r}=1$\hspace{0.5cm} for $r\in R$,
 \item[(B)] $\tilde{s}=v_s$\hspace{0.5cm} for $s\in S$,
 \item[(C)] $\tilde{y}\iota (x)\tilde{y}^{-1}=w_{x,y}$\hspace{0.5cm} for $x\in X$, $y\in Y$.
\end{enumerate} 
\end{lem}

\subsection{Extended lantern relations}\label{subsection_extended_lantern}
Let $S$ be a connected compact surface and let $D$ be a disk on ${\rm int}S$ with the center point $x_0$. Then we have the point pushing map (defined in Section~\ref{rel_twist_crosscappushing}) $j_{x_0}:\pi _1(S,x_0)\rightarrow \mathcal{M}(S,x_0)$. For a two-sided simple loop $\gamma $ on $S$ based at $x_0$, we take an orientation $\theta _{\gamma }\in \{ +_\gamma ,-_\gamma \}$ of $\mathcal{N}_{S}(\gamma )$. Denote by $c_1$ (resp. $c_2$) the boundary component of $\mathcal{N}_{S}(\gamma )$ on the right (resp. left) side of $\gamma $, and by $\theta _i\in \{ +_{c_i}, -_{c_i}\}$ $(i=1,2)$ the orientation of $\mathcal{N}_{S}(c_i)$ which is induced by $\theta _\gamma $. We regard $\gamma $ as an element of $\pi _1(S,x_0)$. Then we have a well-known relation
\[
j_{x_0}(\gamma )=t_{c_1;\theta _1}t_{c_2;\theta _2}^{-1}.
\] 

Let $\mathcal{L}^+=\mathcal{L}^+(S,x_0)$ be the subset of $\pi _1(S,x_0)$ which consists of elements represented by two-sided simple loops. Then we define a map
\[
\Delta =\Delta _{x_0}: \mathcal{L}^+\rightarrow \mathcal{M}(S-{\rm int}D)
\]
as follows. For any two-sided simple loop $\gamma $ on $S$ based at $x_0$, we take $\mathcal{N}_{S}(\gamma )$ whose interior contains $D$. Then we take $c_1$, $c_2$, $\theta _1$ and $\theta _2$ as above.
Define the inclusion $\iota :S-{\rm int}D\rightarrow S$ and $\tilde{c}_i:=\iota ^{-1}(c_i)$ for $i=1,2$. Then we define 
\[
\Delta (\gamma ):=t_{\tilde{c}_1;\theta _{\tilde{c}_1}}t_{\tilde{c}_2;\theta _{\tilde{c}_2}}^{-1}\in \mathcal{M}(S-{\rm int}D),
\]
where $\theta _{\tilde{c}_i}:=(\iota |_{S-{\rm int}D})^{-1}_\ast (\theta _{c_i})\in \{ +_{\tilde{c}_i},-_{\tilde{c}_i}\}$ for $i=1,2$.
 
The next two lemmas are obtained from an argument in Section~3 of \cite{Johnson1}.
\begin{lem}\label{L+}
Let $\alpha $ and $\beta $ be two-sided simple loops on $S$ based at $x_0$ such that $\alpha $ tangentially intersects with $\beta $ at $x_0$ only and the composition $\alpha \beta \in \pi _1(S,x_0)$ is represented by a simple loop. We take the orientation of $\mathcal{N}_{S}(\alpha \cup \beta )$ which is induced by the orientation of $\mathcal{N}_{S-{\rm int}D}(\partial D)$. Then we have 
\[
\Delta (\alpha )\Delta (\beta )=\Delta (\alpha \beta )t_{\partial D}^\varepsilon ,
\]   
where $\varepsilon =1$ if $\alpha $ and $\beta $ are counterclockwise around $x_0$ as on the left-hand side of Figure~\ref{extended_lantern1} and $\varepsilon =-1$ if $\alpha $ and $\beta $ are clockwise around $x_0$ as on the right-hand side of Figure~\ref{extended_lantern1}.
\end{lem}
The relations in Lemma~\ref{L+} are original lantern relations. We call the relations in Lemma~\ref{L+} Relations~(L+) when $\varepsilon =1$ and Relations~(L-) when $\varepsilon =-1$ (see Figure~\ref{extended_lantern1}).

\begin{figure}[h]
\includegraphics[scale=0.8]{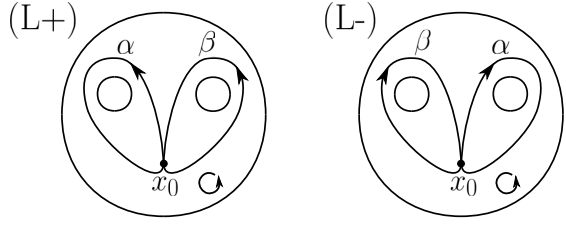}
\caption{Oriented subsurfaces $\mathcal{N}_{S}(\alpha \cup \beta )$.}\label{extended_lantern1}
\end{figure}

\begin{lem}\label{L0}
Let $\alpha $ and $\beta $ be two-sided simple loops on $S$ based at $x_0$ such that $\alpha $ transversely intersects with $\beta $ at $x_0$ only. Then we have 
\[
\Delta (\alpha )\Delta (\beta )=\Delta (\alpha \beta ).
\]   
\end{lem} 
We call the relations in Lemma~\ref{L0} Relations~(L0). We have the following lemma. 
\begin{lem}\label{L0_braid}
Relations~(L0) are obtained from the braid relations~(i).
\end{lem}
\begin{proof}
Let $\alpha $ and $\beta $ be two-sided simple loops on $S$ based at $x_0$ such that $\alpha $ transversely intersects with $\beta $ at $x_0$ only. We take a representative of $\alpha \beta \in \pi _1(S,x_0)$ by a simple loop $\gamma $ and also take the orientations of $\mathcal{N}_{S}(\alpha \cup \beta )\subset S$ and $\mathcal{N}_{S}(\alpha \cup \beta )-{\rm int}D\subset S-{\rm int}D$ which is induced by the orientation of $\mathcal{N}_{S-{\rm int}D}(\partial D)$. Define boundary components $a_1\sqcup a_2=\partial \mathcal{N}_{S}(\alpha )$, $b_1\sqcup b_2=\partial \mathcal{N}_{S}(\beta )$ and $c_1\sqcup c_2=\partial \mathcal{N}_{S}(\gamma )$ such that $a_1$, $b_1$ and $c_1$ are on the right-hand side of $\alpha $, $\beta $ and $\gamma $, respectively. We consider the case where the algebraic intersection number, with respect to the orientation of $\mathcal{N}_{S}(\alpha \cup \beta )$, of $\alpha $ and $\beta $ is $1$ and orientations of $\mathcal{N}_{S-{\rm int}D}(\tilde{a}_i)$, $\mathcal{N}_{S-{\rm int}D}(\tilde{b}_i)$ and $\mathcal{N}_{S-{\rm int}D}(\tilde{c}_i)$ are compatible with the orientation of $\mathcal{N}_{S}(\alpha \cup \beta )$. Figure~\ref{l0_braid1} expresses this situation. Then we have $\Delta (\alpha )=t_{\tilde{a}_1}t_{\tilde{a}_2}^{-1}$, $\Delta (\beta )=t_{\tilde{b}_1}t_{\tilde{b}_2}^{-1}$ and $\Delta (\gamma )=t_{\tilde{c}_1}t_{\tilde{c}_2}^{-1}$. For the other cases, we can prove this lemma by an argument similar to the following argument.

Since $t_{\tilde{a}_2}^{-1}(\tilde{b}_i)=\tilde{c}_i$ for $i=1,2$, we have 
\begin{eqnarray*}
\Delta (\alpha \beta )
&=& t_{\tilde{c}_1}t_{\tilde{c}_2}^{-1}\\
&\stackrel{(\text{I})}{=}& t_{\tilde{a}_2}^{-1}(t_{\tilde{b}_1}t_{\tilde{b}_2}^{-1})t_{\tilde{a}_2}\\
&=& t_{\tilde{a}_2}^{-1}t_{\tilde{b}_1}t_{\tilde{b}_2}^{-1}t_{\tilde{a}_2}\cdot (t_{\tilde{b}_1}t_{\tilde{b}_2}^{-1})^{-1}t_{\tilde{b}_1}t_{\tilde{b}_2}^{-1}\\
&=& t_{\tilde{a}_2}^{-1}\underline{(t_{\tilde{b}_1}t_{\tilde{b}_2}^{-1})t_{\tilde{a}_2}(t_{\tilde{b}_1}t_{\tilde{b}_2}^{-1})^{-1}}\cdot t_{\tilde{b}_1}t_{\tilde{b}_2}^{-1}\\
&\stackrel{(\text{I})}{=}& t_{\tilde{a}_2}^{-1}t_{\tilde{a}_1}\cdot t_{\tilde{b}_1}t_{\tilde{b}_2}^{-1}\\
&=& \Delta (\alpha )\Delta (\beta ).
\end{eqnarray*}
We have the lemma.
\end{proof}  
We remark that Relations~(L+), (L-) and (L0) are obtained from Relations~(I) and (I\hspace{-0.04cm}I\hspace{-0.04cm}I).
 
\begin{figure}[h]
\includegraphics[scale=0.9]{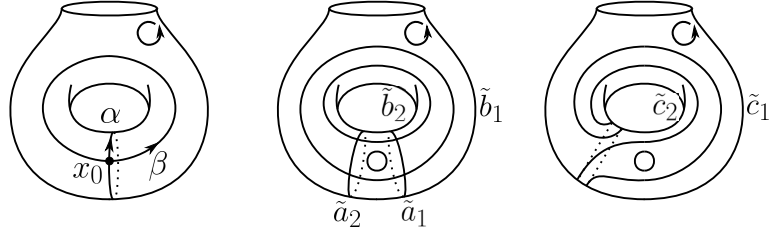}
\caption{Two-sided simple loops $\alpha $ and $\beta $ on $\mathcal{N}_{S}(\alpha \cup \beta )$ such that their algebraic intersection number is $1$ (on the left side), and simple closed curves $\tilde{a}_i$, $\tilde{b}_i$ and $\tilde{c}_i$ on $\mathcal{N}_{S}(\alpha \cup \beta )-{\rm int}D$ for $i=1,2$ (on the center and the right side). }\label{l0_braid1}
\end{figure}

\subsection{Generators for the subgroup of the fundamental group generated by two-sided loops}\label{subsection_gen_pi1plus}
Recall that we take a model of $N_{g,n}$ as in Figure~\ref{scc_nonorisurf1} for $n\geq 1$. Assume $n\geq 2$. We regard $N_{g,n-1}$ as the surface obtained by regluing $D_{g+n-1}$ and $N_{g,n}$. Put the center point $x_0$ of $D_{g+n-1}$. Let $\pi _1(N_{g,n-1})^+$ be the subgroup of $\pi _1(N_{g,n-1},x_0)$ which consists of elements that is represented by loops such that the pushing maps along their loops preserve a local orientation of $x_0$. Let $x_1,\dots ,x_g, y_1,\dots ,y_{n-2}$ be loops on $N_{g,n-1}$ based at $x_0$ as in Figure~\ref{gen_pi1_nonori_boundary1} and we regard $N_{g,n-1}$ as the surface in Figure~\ref{gen_pi1_nonori_boundary1} for some conveniences. Note that $x_i^2$ for $1\leq i\leq g$, $x_{i+1}x_i$ for $1\leq i\leq g-1$ and $x_1^{-1}y_ix_1$ for $1\leq i\leq n-2$ are elements of $\pi _1(N_{g,n-1})^+$. $x_i^2$, $x_{i+1}x_i$ and $x_1^{-1}y_ix_1$ are represented by loops as in Figure~\ref{gen_pi1plus_nonori_boundary1}. Since $\pi _1(N_{g,n-1},x_0)$ is the free group which is freely generated by $x_1,\dots ,x_g, y_1,\dots ,y_{n-2}$, $\pi _1(N_{g,n-1})^+$ is also isomorphic to a free group. We have the following lemma.
\begin{lem}\label{gen_pi+}
For $g\geq 1$ and $n\geq 2$, $\pi _1(N_{g,n-1})^+$ is the free group which is freely generated by $x_1^2, \dots, x_g^2, x_2x_1, \dots , x_gx_{g-1}, y_1, \dots ,y_{n-2}, x_1^{-1}y_1x_1, \dots ,x_1^{-1}y_{n-2}x_1$.
\end{lem}
\begin{proof}
We use the Reidemeister-Schreier method (for instance see \cite{Johnson2}) for $\pi _1(N_{g,n-1})^+\subset \pi _1(N_{g,n-1},x_0)$ to obtain the generators for $\pi _1(N_{g,n-1})^+$. Since $\pi _1(N_{g,n-1})^+$ is an index $2$ subgroup of $\pi _1(N_{g,n-1},x_0)$ and the non-trivial element of the quotient group $\pi _1(N_{g,n-1},x_0)/\pi _1(N_{g,n-1})^+$ is represented by $x_1$, the set $U:=\{ 1, x_1\} \subset \pi _1(N_{g,n-1})^+$ is a Schreier transversal for $\pi _1(N_{g,n-1})^+$ in $\pi _1(N_{g,n-1},x_0)$. Set $X:=\{ x_1,\dots ,x_g, y_1,\dots ,y_{n-2}\}$. For any word $w$ in $X$, denote by $\overline{w}$ the element of $U$ whose equivalence class in $\pi _1(N_{g,n-1},x_0)/\pi _1(N_{g,n-1})^+$ is the same as that of $w$. Then $\pi _1(N_{g,n-1})^+$ is the free group which is freely generated by 
\begin{eqnarray*}
B&=&\{ \overline{xu}^{-1}xu \mid x\in X, u\in U, xu\not\in U \} \\
&=&\{ x_ix_1,\ x_1^{-1}x_j,\ y_k,\ x_1^{-1}y_kx_1\mid i=1,\dots ,g,\ j=2,\dots ,g,\ k=1,\dots ,n-2 \}.
\end{eqnarray*}
Put $z_1:=x_1^2$, $z_i:=(x_ix_1)(x_1^{-1}x_i)$ for $i=2,\dots ,g$, $w_1:=x_2x_1$ and $w_i:=(x_{i+1}x_1)(x_1^{-1}x_i)$ for $i=2,\dots ,g-1$ as words in $B$. By using the Tietze transformations (for instance see \cite[Proposition~4.4.5, p46]{Johnson2}) and relations $(x_1^{-1}x_i)=(x_ix_1)^{-1}z_i$ and $(x_{i+1}x_1)=w_i(x_1^{-1}x_i)^{-1}$ for $i\geq 2$, we have isomorphisms
\begin{eqnarray*}
&&\bigl< B \mid \bigr> \\
&\cong &\bigl< B\cup \{ z_i, w_j \mid i=1,\dots g,\ j=1,\dots ,g-1\} \mid \\
& &z_1=x_1^2,\ z_i=(x_ix_1)(x_1^{-1}x_i),\ w_1=x_2x_1,\ w_j=(x_{j+1}x_1)(x_1^{-1}x_j)\bigr> \\
&\cong &\bigl< B\cup \{ z_i, w_j \mid i=1,\dots g,\ j=1,\dots ,g-1\} \mid \\
& &z_1=x_1^2,\ w_1=x_2x_1,\ x_ix_1=w_{i-1}z_{i-1}^{-1}\cdots w_2z_2^{-1}w_1,\\
&& x_1^{-1}x_{j+1}=w_1^{-1}z_2w_2^{-1}\cdots z_jw_j^{-1}z_{j+1}\bigr> \\
&\cong &\bigl< \{  z_i,\ w_j,\ y_k,\ x_1^{-1}y_kx_1 \mid i=1,\dots g,\ j=1,\dots ,g-1,\ k=1,\dots ,n-2\} \mid \bigr> .
\end{eqnarray*}
Note that $z_i=x_i^2$ and $w_i=x_{i+1}x_i$ as elements of $\pi _1(N_{g,n-1})^+$. We get this lemma.
\end{proof}

\begin{figure}[h]
\includegraphics[scale=0.7]{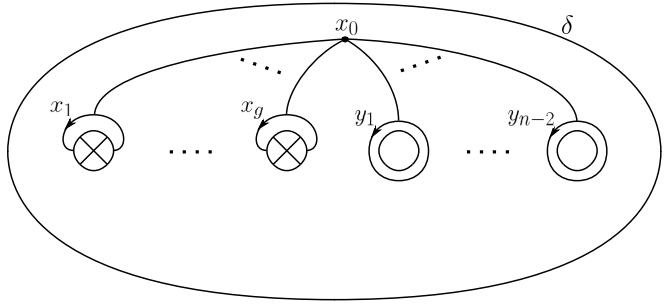}
\caption{Loops $x_1,\dots ,x_g, y_1,\dots ,y_{n-2}$ on $N_{g,n-1}$ based at $x_0$.}\label{gen_pi1_nonori_boundary1}
\end{figure}
\begin{figure}[h]
\includegraphics[scale=0.7]{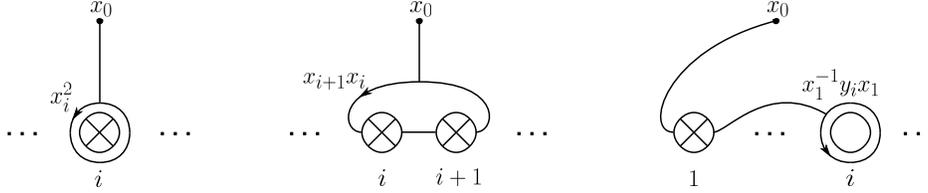}
\caption{Loops $x_i^2,\ x_{i+1}x_i,\  x_1^{-1}y_ix_1$ on $N_{g,n-1}$ based at $x_0$.}\label{gen_pi1plus_nonori_boundary1}
\end{figure}

\subsection{Proof of Proposition~\ref{thm_finite_presentation}}\label{proof_finite_presentation}

Let $\mathcal{M}^+(N_{g,n-1},x_0)$ be the subgroup of $\mathcal{M}(N_{g,n-1},x_0)$ whose elements preserve a local orientation of $x_0$. For $n\geq 2$, the forgetful homomorphism $\mathcal{M}(N_{g,n-1},x_0)\rightarrow \mathcal{M}(N_{g,n-1})$ induces the following exact sequence
\begin{eqnarray}\label{exact1}
\hspace{1.0cm} 1\longrightarrow \pi _1(N_{g,n-1})^+ \stackrel{j_{x_0}|_{\pi _1(N_{g,n-1})^+}}{\longrightarrow }\mathcal{M}^+(N_{g,n-1},x_0)\longrightarrow \mathcal{M}(N_{g,n-1})\longrightarrow 1.
\end{eqnarray}
Since $\pi _1(N_{g,n-1},x_0)$ is isomorphic to a free group, the center of $\pi _1(N_{g,n-1},x_0)$ is trivial. Thus the homomorphism $j_{x_0}$ is injective by \cite[Corollary~1.2]{Birman}.

The natural inclusion $\iota :N_{g,n}\hookrightarrow N_{g,n-1}$ induces the surjective homomorphism $\iota _\ast :\mathcal{M}(N_{g,n})\rightarrow \mathcal{M}^+(N_{g,n-1},x_0)$. By Theorem~3.6 in \cite{Stukow2}, we have the exact sequence 
\begin{eqnarray}\label{exact2}
1\longrightarrow \mathbb Z[d_{n-1}] \longrightarrow \mathcal{M}(N_{g,n})\stackrel{\iota _\ast }{\longrightarrow }\mathcal{M}^+(N_{g,n-1},x_0)\longrightarrow 1.
\end{eqnarray}

We remark that for a relation $v_1^{\varepsilon _1}\cdots v_k^{\varepsilon _k}=w_1^{\delta _1}\cdots w_l^{\delta _l}$ of a group $G$, we call $v_1^{\varepsilon _1}\cdots v_k^{\varepsilon _k}(w_1^{\delta _1}\cdots w_l^{\delta _l})^{-1}$ the relator of $G$ (obtained from the relation).

\begin{proof}[Proof of Proposition~\ref{thm_finite_presentation}]
We proceed by induction on $n\geq 1$. The finite presentation for $\mathcal{M}(N_{g,1})$ is given by Theorem~\ref{thm_Stukow}. 

Assume $n\geq 2$. For some conveniences, we define $\iota _\ast (a_i)=a_i\in \mathcal{M}^+(N_{g,n-1},x_0)$, $\iota _\ast (y)=y\in \mathcal{M}^+(N_{g,n-1},x_0)$, $\iota _\ast (b)=b\in \mathcal{M}^+(N_{g,n-1},x_0)$, etc., and we take lifts $a_i$, $y$, $b$, $\ldots \in \mathcal{M}^+(N_{g,n-1},x_0)$ of $a_i$, $y$, $b$, $\ldots \in \mathcal{M}(N_{g,n-1})$, respectively. By Lemma~\ref{gen_pi+}, the inductive assumption and applying Lemma~\ref{presentation_exact} to the exact sequence~\ref{exact1}, we obtain the following presentation for $\mathcal{M}^+(N_{g,n-1},x_0)$:\\
generators: 
\begin{enumerate}
\item[(1)] $a_i$ $(1\leq i\leq g-1)$, $y$, $b$, $d_i$ $(1\leq i\leq n-2)$, $a_{i;j}$ $(1\leq i\leq g-1, 1\leq j\leq n-2)$, $r_{i;j}$ $(1\leq i\leq g,1\leq j\leq n-2)$, $s_{i,j}$  $(1\leq i<j\leq n-2)$, $\bar{s}_{i,j}$ $(1\leq i<j\leq n-2)$,\\
\item[(2)] $j_{x_0}(x_i^2)=r_{i;n-1}$ $(1\leq i\leq g)$, $j_{x_0}(x_{i+1}x_i)=a_{i;n-1}a_i^{-1}$ $(1\leq i\leq g-1)$, $j_{x_0}(y_i)=s_{i,n-1}d_i^{-1}$ $(1\leq i\leq n-2)$ and $j_{x_0}(x_1^{-1}y_ix_1)=\bar{s}_{i,n-1}d_i^{-1}$ $(1\leq i\leq n-2)$.\\
\end{enumerate}
Denote by $X_1$ the set of generators in (1) and by $X_2$ the set of generators in (2).\\
relations: 
\begin{enumerate}
\item[(1)] For any relator $s$ of the presentation for $\mathcal{M}(N_{g,n-1})$ and the lift $\tilde{s}$ of $s$ to $\mathcal{M}^+(N_{g,n-1},x_0)$, 
\[
\tilde{s}=v_s
\] 
for some product $v_s$ of elements in $X_2$.
\item[(2)] For $x\in X_2$ and $f \in X_1$, 
\[
fxf^{-1}=w_{x,f}
\]
for some product $w_{x,f}$ of elements in $X_2$.
\end{enumerate}
Note that the generators in (1) come from the generators for $\mathcal{M}(N_{g,n-1})$ and the generators in (2) come from the generators for $\pi _1(N_{g,n-1})^+$ in Lemma~\ref{gen_pi+}. 

We calculate each $v_s$ and $w_{x,f}$ in the relation~(1) and (2) above. We take the subsurface $N^\prime $ of $N_{g,n}$ which is diffeomorphic to $N_{g,n-1}$ as in Figure~\ref{nonorisurf_emb2_boundary}. Since every simple closed curve used in a generator of the presentation for $\mathcal{M}(N_{g,n-1})$ is isotopic to a simple closed curve on ${\rm int}N^\prime $, we regard generators of the presentation for $\mathcal{M}(N_{g,n-1})$ as elements of $\mathcal{M}(N^\prime )$. In particular, the inclusion $\iota ^\prime :N^\prime \hookrightarrow N_{g,n}$ induces the injective homomorphism $\iota ^\prime _\ast :\mathcal{M}(N^\prime )\rightarrow \mathcal{M}(N_{g,n})$. By using the composition $\iota _\ast \circ \iota ^\prime _\ast :\mathcal{M}(N^\prime )\rightarrow \mathcal{M}^+(N_{g,n-1},x_0)$, we can take $v_s=1$ for each relator $s$ of the presentation for $\mathcal{M}(N_{g,n-1})$. 

Recall that $\bar{s}_{j,k;i}$ is the Dehn twist along the simple closed curve $\iota (\bar{\sigma }_{j,k;i})$ and $\bar{\sigma }_{j,k;i}$ is defined in Figure~\ref{sigmabar_jki}. For the relation~(2) above, we can take $w_{x,f}$ as follows. Assume $1\leq{i,m}\leq{g}$, $1\leq j\leq n-2$, $1\leq{l<t}\leq n-2$ and $[x_1,x_2]=x_1x_2x_1^{-1}x_2^{-1}$.
\begin{enumerate}
\item[(D1a)$^\prime $] $a_m(a_{i;n-1}a_i^{-1})a_m^{-1}=
		\left\{
		\begin{array}{ll}
		(a_{i;n-1}a_i^{-1})(a_{i-1;n-1}a_{i-1}^{-1}) \hspace{0.5cm}\text{for }m=i-1,\\
		(a_{i+1;n-1}a_{i+1}^{-1})^{-1}(a_{i;n-1}a_i^{-1})\hspace{0.5cm}\text{for }m=i+1,\\
		a_{i;n-1}a_i^{-1}\hspace{0.5cm}\text{for }m\neq{i-1,i+1},
		\end{array}
		\right.$
\item[(D1b)$^\prime $] $y(a_{i;n-1}a_i^{-1})y^{-1}=
		\left\{
		\begin{array}{ll}
		(a_{1;n-1}a_1^{-1})^{-1}r_{2;n-1}r_{1;n-1}\hspace{0.5cm}\text{for }i=1,\\
		(a_{2;n-1}a_2^{-1})r_{1;n-1}\hspace{0.5cm}\text{for }i=2,\\
		a_{i;n-1}a_i^{-1}\hspace{0.5cm}\text{for }i\geq 3,
		\end{array}
		\right.$
\item[(D1c)$^\prime $] $b(a_{i;n-1}a_i^{-1})b^{-1}=\\
		\left\{
		\begin{array}{ll}
		\{(a_{3;n-1}a_3^{-1})(a_{1;n-1}a_1^{-1})\}^{-1}(a_{1;n-1}a_1^{-1})\{(a_{3;n-1}a_3^{-1})(a_{1;n-1}a_1^{-1})\}\\
\text{for }i=1,\\
		\{(a_{3;n-1}a_3^{-1})(a_{1;n-1}a_1^{-1})\}^{-1}(a_{2;n-1}a_2^{-1})\{(a_{3;n-1}a_3^{-1})(a_{1;n-1}a_1^{-1})\}\\
\text{for }i=2,\\
		(a_{1;n-1}a_1^{-1})^{-1}(a_{3;n-1}a_3^{-1})(a_{1;n-1}a_1^{-1})\hspace{0.5cm}\text{for }i=3,\\
		(a_{4;n-1}a_4^{-1})(a_{3;n-1}a_3^{-1})(a_{1;n-1}a_1^{-1})\hspace{0.5cm}\text{for }i=4,\\
		a_{i;n-1}a_i^{-1}\hspace{0.5cm}\text{for }i\geq5,
		\end{array}
		\right.$
\item[(D1d)$^\prime $] $a_{m;l}(a_{i;n-1}a_i^{-1})a_{m;l}^{-1}=\\
		\left\{
		\begin{array}{ll}
		{[(s_{l,n-1}d_l^{-1})^{-1},(a_{m;n-1}a_m^{-1})^{-1}]}^{-1}(a_{i;n-1}a_i^{-1})\\
{[(s_{l,n-1}d_l^{-1})^{-1},(a_{m;n-1}a_m^{-1})^{-1}]}  \hspace{0.5cm}\text{for }m\leq{i-2},\\
		{[(a_{i-1;n-1}a_{i-1}^{-1})^{-1},(s_{l,n-1}d_l^{-1})^{-1}]}(a_{i;n-1}a_i^{-1})(s_{l,n-1}d_l^{-1})(a_{i-1;n-1}a_{i-1}^{-1}) \\ \text{for }m=i-1,\\
		\{(s_{l,n-1}d_l^{-1})(a_{i;n-1}a_i^{-1})\}^{-1}(a_{i;n-1}a_i^{-1})\{(s_{l,n-1}d_l^{-1})(a_{i;n-1}a_i^{-1})\} \\ \text{for }m=i,\\
		(a_{i+1;n-1}a_{i+1}^{-1})^{-1}(s_{l,n-1}d_l^{-1})^{-1}(a_{i;n-1}a_i^{-1}) \hspace{0.5cm}\text{for }m=i+1,\\
		a_{i;n-1}a_i^{-1}\hspace{0.5cm}\text{for }m\geq{i+2},
		\end{array}
		\right.$
\item[(D1e)$^\prime $] $r_{m;l}(a_{i;n-1}a_i^{-1})r_{m;l}^{-1}=\\
		\left\{
		\begin{array}{ll}
		{[(s_{l,n-1}d_l^{-1})^{-1},r_{m;n-1}^{-1}]}^{-1}(a_{i;n-1}a_i^{-1}){[(s_{l,n-1}d_l^{-1})^{-1},r_{m;n-1}^{-1}]}\\ \text{for }m\leq{i-1},\\
		\{ r_{i;n-1}^{-1}(s_{l,n-1}d_l^{-1})^{-1}r_{i;n-1}\} (s_{l,n-1}d_l^{-1})(a_{i;n-1}a_i^{-1})\\
		(\bar{s}_{l,n-1;i}d_l^{-1})^{-1}\{ r_{i;n-1}^{-1}(s_{l,n-1}d_l^{-1})^{-1}r_{i;n-1}\} ^{-1}\hspace{0.5cm}\text{for }m=i,\\
		r_{i+1;n-1}^{-1}(s_{l,n-1}d_l^{-1})^{-1}r_{i+1;n-1}(\bar{s}_{l,n-1;i+1}d_l^{-1})(a_{i;n-1}a_i^{-1})\\ \text{for }m=i+1,\\
		a_{i;n-1}a_i^{-1}\hspace{0.5cm}\text{for }m\geq{i+2},
		\end{array}
		\right.$
\item[(D1f)$^\prime $] $s_{l,t}(a_{i;n-1}a_i^{-1})s_{l,t}^{-1}=a_{i;n-1}a_i^{-1}$,
\item[(D1g)$^\prime $] $\bar{s}_{l,t}(a_{i;n-1}a_i^{-1})\bar{s}_{l,t}^{-1}=\\
		\left\{
		\begin{array}{ll}
		[(\bar{s}_{l,k}d_l^{-1})^{-1},s_{t,k}d_t^{-1}]^{-1}(s_{l,k}d_l^{-1})(a_{1;k}a_1^{-1})r_{1;k}(\bar{s}_{t,k}d_t^{-1})r_{1;k}^{-1}\\
(s_{l,k}d_l^{-1})^{-1}r_{1;k}(\bar{s}_{t,k}d_t^{-1})^{-1}r_{1;k}^{-1}[(\bar{s}_{l,k}d_l^{-1})^{-1},(s_{t,k}d_t^{-1})] \hspace{0.5cm}\text{for }i=1,\\
		\{[\bar{s}_{l,n-1}d_l^{-1},(s_{t,n-1}d_t^{-1})^{-1}][s_{l,n-1}d_l^{-1},r_{1;n-1}(\bar{s}_{t,n-1}d_t^{-1})^{-1}r_{1;n-1}^{-1}]\} \\(a_{i;n-1}a_i^{-1})\\
		\{[\bar{s}_{l,n-1}d_l^{-1},(s_{t,n-1}d_t^{-1})^{-1}][s_{l,n-1}d_l^{-1},r_{1;n-1}(\bar{s}_{t,n-1}d_t^{-1})^{-1}r_{1;n-1}^{-1}]\} ^{-1} \\ \text{for }i\geq2,
		\end{array}
		\right.$
\item[(D1h)$^\prime $] $d_l(a_{i;n-1}a_i^{-1})d_l^{-1}=a_{i;n-1}a_i^{-1}$,
\item[(D2a)$^\prime $] $a_mr_{i;n-1}a_m^{-1}=
		\left\{
		\begin{array}{ll}
		r_{i;n-1}r_{i-1;n-1}(a_{i-1;n-1}a_{i-1}^{-1})^{-1}r_{i;n-1}(a_{i-1;n-1}a_{i-1}^{-1})\\
\text{for }m=i-1,\\
		(a_{i;n-1}a_i^{-1})^{-1}r_{i+1;n-1}^{-1}(a_{i;n-1}a_i^{-1})\hspace{0.5cm}\text{for }m=i,\\
		r_{i;n-1}\hspace{0.5cm}\text{for }m\neq{i-1,i},
		\end{array}
		\right.$
\item[(D2b)$^\prime $] $yr_{i;n-1}y^{-1}=
		\left\{
		\begin{array}{ll}
		\{ (a_{1;n-1}a_1^{-1})^{-1}r_{2;n-1}r_{1;n-1}\} ^{-1}r_{1;n-1}^{-1}\\ \{ (a_{1;n-1}a_1^{-1})^{-1}r_{2;n-1}r_{1;n-1}\} \hspace{0.5cm} \text{for }i=1,\\
		(a_{1;n-1}a_1^{-1})r_{1;n-1}(a_{1;n-1}a_1^{-1})^{-1}r_{2;n-1}r_{1;n-1}\hspace{0.5cm}\text{for }i=2,\\
		r_{i;n-1}\hspace{0.5cm}\text{for }i\geq 3,
		\end{array}
		\right.$
\item[(D2c)$^\prime $] $br_{i;n-1}b^{-1}=
		\left\{
		\begin{array}{ll}
		(a_{1;n-1}a_1^{-1})^{-1}(a_{3;n-1}a_3^{-1})^{-1}(a_{2;n-1}a_2^{-1})^{-1}\\
		r_{4;n-1}^{-1}(a_{3;n-1}a_3^{-1})r_{3;n-1}^{-1}(a_{2;n-1}a_2^{-1})r_{2;n-1}^{-1}(a_{1;n-1}a_1^{-1})\\ \text{for }i=1,\\
		\{ (a_{3;n-1}a_3^{-1})(a_{1;n-1}a_1^{-1})\} ^{-1}(a_{1;n-1}a_1^{-1})(a_{3;n-1}a_3^{-1})r_{2;n-1}\\
		r_{1;n-1}(a_{1;n-1}a_1^{-1})^{-1}r_{2;n-1}(a_{2;n-1}a_2^{-1})^{-1}r_{3;n-1}\\
(a_{3;n-1}a_3^{-1})^{-1}r_{4;n-1}(a_{2;n-1}a_2^{-1})\{ (a_{3;n-1}a_3^{-1})(a_{1;n-1}a_1^{-1})\} \\ \text{for }i=2,\\
		\{ (a_{3;n-1}a_3^{-1})(a_{1;n-1}a_1^{-1})\} ^{-1}r_{4;n-1}^{-1}\\
		(a_{3;n-1}a_3^{-1})r_{3;n-1}^{-1}(a_{2;n-1}a_2^{-1})r_{2;n-1}^{-1}(a_{1;n-1}a_1^{-1})r_{1;n-1}^{-1}\\
		(a_{2;n-1}a_2^{-1})^{-1}r_{3;n-1}(a_{3;n-1}a_3^{-1})^{-1}(a_{1;n-1}a_1^{-1})^{-1}\\
\{ (a_{3;n-1}a_3^{-1})(a_{1;n-1}a_1^{-1})\} \hspace{0.5cm}\text{for }i=3,\\
		r_{4;n-1}(a_{2;n-1}a_2^{-1})r_{1;n-1}(a_{1;n-1}a_1^{-1})^{-1}r_{2;n-1}(a_{2;n-1}a_2^{-1})^{-1}\\
r_{3;n-1}(a_{3;n-1}a_3^{-1})^{-1}r_{4;n-1}(a_{3;n-1}a_3^{-1})(a_{1;n-1}a_1^{-1})\\
\text{for }i=4,\\
		r_{i;n-1}\hspace{0.5cm}\text{for }i\geq5,
		\end{array}
		\right.$
\item[(D2d)$^\prime $] $a_{m;l}r_{i;n-1}a_{m;l}^{-1}=\\
		\left\{
		\begin{array}{ll}
		{[(s_{l,n-1}d_l^{-1})^{-1},(a_{m;n-1}a_m^{-1})^{-1}]}^{-1}r_{i;n-1}{[(s_{l,n-1}d_l^{-1})^{-1},(a_{m;n-1}a_m^{-1})^{-1}]}\\
\text{for }m\leq{}i-2,\\
		\{ (s_{l,n-1}d_l^{-1})(a_{i-1;n-1}a_{i-1}^{-1})\} ^{-1}(a_{i-1;n-1}a_{i-1}^{-1})(s_{l,n-1}d_l^{-1})r_{i;n-1}\\
		r_{i-1;n-1}(a_{i-1;n-1}a_{i-1}^{-1})^{-1}r_{i;n-1}(\bar{s}_{l,n-1;i}d_l^{-1})\{ (s_{l,n-1}d_l^{-1})(a_{i-1;n-1}a_{i-1}^{-1})\} \\
\text{for }m=i-1,\\
		(a_{i;n-1}a_i^{-1})^{-1}(s_{l,n-1}d_l^{-1})^{-1}r_{i+1;n-1}^{-1}(a_{i;n-1}a_i^{-1})(\bar{s}_{l,n-1;i}d_l^{-1})^{-1}\\
\text{for }m=i,\\
		r_{i;n-1}\hspace{0.5cm}\text{for }m\geq{i+1},
		\end{array}
		\right.$
\item[(D2e)$^\prime $] $r_{m;l}r_{i;n-1}r_{m;l}^{-1}=
		\left\{
		\begin{array}{ll}
		{[(s_{l,n-1}d_l^{-1})^{-1},r_{m;n-1}^{-1}]}^{-1}r_{i;n-1}{[(s_{l,n-1}d_l^{-1})^{-1},r_{m;n-1}^{-1}]}\\
\text{for }m\leq{i-1},\\
		\{(s_{l,n-1}d_l^{-1})r_{i;n-1}\}^{-1}r_{i;n-1}\{(s_{l,n-1}d_l^{-1})r_{i;n-1}\} \\
\text{for }m=i,\\
		r_{i;n-1}\hspace{0.5cm}\text{for }m\geq{i+1},
		\end{array}
		\right.$
\item[(D2f)$^\prime $]	$s_{l,t}r_{i;n-1}s_{l,t}^{-1}=r_{i;n-1}$,
\item[(D2g)$^\prime $] $\bar{s}_{l,t}r_{i;n-1}\bar{s}_{l,t}^{-1}=\\
		\left\{
		\begin{array}{ll}
		[s_{t,n-1}d_t^{-1},(\bar{s}_{l,n-1}d_l^{-1})^{-1}][r_{1;n-1}(\bar{s}_{t,n-1}d_t^{-1})r_{1;n-1}^{-1},(s_{l,n-1}d_l^{-1})^{-1}]r_{1;n-1} \\
\text{for }i=1,\\
		\{ [\bar{s}_{l,n-1}d_l^{-1},(s_{t,n-1}d_t^{-1})^{-1}][s_{l,n-1}d_l^{-1},r_{1;n-1}(\bar{s}_{t,n-1}d_t^{-1})^{-1}r_{1;n-1}^{-1}]\} ^{-1}\\r_{1;n-1}
		\{ [\bar{s}_{l,n-1}d_l^{-1},(s_{t,n-1}d_t^{-1})^{-1}][s_{l,n-1}d_l^{-1},r_{1;n-1}(\bar{s}_{t,n-1}d_t^{-1})^{-1}r_{1;n-1}^{-1}]\} \\
\text{for }i\geq2,
		\end{array}
		\right.$
\item[(D2h)$^\prime $] $d_lr_{i;n-1}d_l^{-1}=r_{i;n-1}$,
\item[(D3a)$^\prime $] $a_m(s_{j,n-1}d_j^{-1})a_m^{-1}=s_{j,n-1}d_j^{-1}$,
\item[(D3b)$^\prime $] $y(s_{j,n-1}d_j^{-1})y^{-1}=s_{j,n-1}d_j^{-1}$,
\item[(D3c)$^\prime $] $b(s_{j,n-1}d_j^{-1})b^{-1}=s_{j,n-1}d_j^{-1}$,
\item[(D3d)$^\prime $] $a_{m;l}(s_{j,n-1}d_j^{-1})a_{m;l}^{-1}=\\
		\left\{
		\begin{array}{ll}
		{[(s_{l,n-1}d_l^{-1})^{-1},(a_{m;n-1}a_m^{-1})^{-1}]}^{-1}(s_{j,n-1}d_j^{-1})\\ 
{[(s_{l,n-1}d_l^{-1})^{-1},(a_{m;n-1}a_m^{-1})^{-1}]} \hspace{0.5cm} \text{for }l>j,\\
		(a_{m;n-1}a_m^{-1})^{-1}(s_{j,n-1}d_j^{-1})(a_{m;n-1}a_m^{-1})\hspace{0.5cm}\text{for }l=j,\\
		s_{j,n-1}d_j^{-1}\hspace{0.5cm}\text{for }l=j,
		\end{array}
		\right.$
\item[(D3e)$^\prime $] $r_{m;l}(s_{j,n-1}d_j^{-1})r_{m;l}^{-1}=\\
		\left\{
		\begin{array}{ll}
		{[(s_{l,n-1}d_l^{-1})^{-1},r_{m;n-1}^{-1}]}^{-1}(s_{j,n-1}d_j^{-1}){[(s_{l,n-1}d_l^{-1})^{-1},r_{m;n-1}^{-1}]} \\
\text{for }l>j,\\
		r_{m;n-1}^{-1}(s_{j,n-1}d_j^{-1})r_{m;n-1}\hspace{0.5cm}\text{for }l=j,\\
		s_{j,n-1}d_j^{-1}\hspace{0.5cm}\text{for }l<j,
		\end{array}
		\right.$
\item[(D3f)$^\prime $] $s_{l,t}(s_{j,n-1}d_j^{-1})s_{l,t}^{-1}=\\
		\left\{
		\begin{array}{ll}
		\{(s_{t,n-1}d_t^{-1})(s_{j,n-1}d_j^{-1})\}^{-1}(s_{j,n-1}d_j^{-1})\{(s_{t,n-1}d_t^{-1})(s_{j,n-1}d_j^{-1})\} \\
\text{for }l=j,\\
		{[(s_{t,n-1}d_t^{-1})^{-1},(s_{l,n-1}d_l^{-1})^{-1}]}^{-1}(s_{j,n-1}d_j^{-1})\\
{[(s_{t,n-1}d_t^{-1})^{-1},(s_{l,n-1}d_l^{-1})^{-1}]} \hspace{0.5cm}\text{for }l<j<t,\\
		(s_{l,n-1}d_l^{-1})^{-1}(s_{j,n-1}d_j^{-1})(s_{l,n-1}d_l^{-1})\hspace{0.5cm}\text{for }t=j,\\
		s_{j,n-1}d_j^{-1}\hspace{0.5cm}\text{for the other cases},
		\end{array}
		\right.$
\item[(D3g)$^\prime $] $\bar{s}_{l,t}(s_{j,n-1}d_j^{-1})\bar{s}_{l,t}^{-1}=\\
		\left\{
		\begin{array}{ll}
		\{[\bar{s}_{l,n-1}d_l^{-1},(s_{t,n-1}d_t^{-1})^{-1}][s_{l,n-1}d_l^{-1},r_{1;n-1}(\bar{s}_{t,n-1}d_t^{-1})^{-1}r_{1;n-1}^{-1}]\}^{-1}\\
(s_{j,n-1}d_j^{-1})\\
\{[\bar{s}_{l,n-1}d_l^{-1},(s_{t,n-1}d_t^{-1})^{-1}][s_{l,n-1}d_l^{-1},r_{1;n-1}(\bar{s}_{t,n-1}d_t^{-1})^{-1}r_{1;n-1}^{-1}]\} \\
\text{for }l>j,\\
		\{[\bar{s}_{j,n-1}d_j^{-1},(s_{t,n-1}d_t^{-1})^{-1}](s_{j,n-1}d_j^{-1})r_{1;n-1}(\bar{s}_{t,n-1}d_t^{-1})^{-1}r_{1;n-1}^{-1}\}^{-1}\\
(s_{j,n-1}d_j^{-1})\\
\{[\bar{s}_{j,n-1}d_j^{-1},(s_{t,n-1}d_t^{-1})^{-1}](s_{j,n-1}d_j^{-1})r_{1;n-1}(\bar{s}_{t,n-1}d_t^{-1})^{-1}r_{1;n-1}^{-1}\} \\
\text{for }l=j,\\
		{[(\bar{s}_{l,n-1}d_l^{-1})^{-1},s_{t,n-1}d_t^{-1}]}^{-1}(s_{j,n-1}d_j^{-1}){[(\bar{s}_{l,n-1}d_l^{-1})^{-1},s_{t,n-1}d_t^{-1}]} \\
\text{for }l<j<t,\\
		\{ (\bar{s}_{l,n-1}d_l^{-1})(s_{j,n-1}d_j^{-1})^{-1}\} ^{-1}(s_{j,n-1}d_j^{-1})\{ (\bar{s}_{l,n-1}d_l^{-1})(s_{j,n-1}d_j^{-1})^{-1}\} \\
\text{for }t=j,\\
		s_{j,n-1}d_j^{-1}\hspace{0.5cm}\text{for }t<j,
		\end{array}
		\right.$
\item[(D3h)$^\prime $] $d_l(s_{j,n-1}d_j^{-1})d_l^{-1}=s_{j,n-1}d_j^{-1}$,
\item[(D4a)$^\prime $] $a_m(\bar{s}_{j,n-1}d_j^{-1})a_m^{-1}=\\
		\left\{
		\begin{array}{ll}
		\{r_{1;n-1}^{-1}r_{2;n-1}^{-1}(a_{1;n-1}a_1^{-1})\}^{-1}(\bar{s}_{j,n-1}d_j^{-1})\{r_{1;n-1}^{-1}r_{2;n-1}^{-1}(a_{1;n-1}a_1^{-1})\} \\
\text{for }m=1,\\
		\bar{s}_{j,n-1}d_j^{-1}\hspace{0.5cm}\text{for }m\geq2,
		\end{array}
		\right.$
\item[(D4b)$^\prime $] $y(\bar{s}_{j,n-1}d_j^{-1})y^{-1}=\{ r_{1;n-1}^{-1}(a_{1;n-1}a_1^{-1})^{-2}r_{2;n-1}r_{1;n-1}\} ^{-1}(\bar{s}_{j,n-1}d_j^{-1})\\ 
\{ r_{1;n-1}^{-1}(a_{1;n-1}a_1^{-1})^{-2}r_{2;n-1}r_{1;n-1}\},$
\item[(D4c)$^\prime $] $b(\bar{s}_{j,n-1}d_j^{-1})b^{-1}=\{r_{1;n-1}^{-1}(a_{2;n-1}a_2^{-1})^{-1}r_{4;n-1}^{-1}(a_{3;n-1}a_3^{-1})r_{3;n-1}^{-1}\\
(a_{2;n-1}a_2^{-1})r_{2;n-1}^{-1}(a_{1;n-1}a_1^{-1})\}^{-1}(\bar{s}_{j,n-1}d_j^{-1})
		\{r_{1;n-1}^{-1}(a_{2;n-1}a_2^{-1})^{-1}\\
r_{4;n-1}^{-1}(a_{3;n-1}a_3^{-1})r_{3;n-1}^{-1}(a_{2;n-1}a_2^{-1})r_{2;n-1}^{-1}(a_{1;n-1}a_1^{-1})\},$
\item[(D4d)$^\prime $] $a_{m;l}(\bar{s}_{j,n-1}d_j^{-1})a_{m;l}^{-1}=\\
		\left\{
		\begin{array}{ll}
		\{r_{1;n-1}^{-1}r_{2;n-1}^{-1}(a_{1;n-1}a_1^{-1})(\bar{s}_{l,n-1}d_l^{-1})^{-1}\}^{-1}(\bar{s}_{j,n-1}d_j^{-1})\\
		\{r_{1;n-1}^{-1}r_{2;n-1}^{-1}(a_{1;n-1}a_1^{-1})(\bar{s}_{l,n-1}d_l^{-1})^{-1}\} \hspace{0.5cm} \text{for }m=1,l<j,\\
		\{(\bar{s}_{l,n-1}d_l^{-1})^{-1}r_{1;n-1}^{-1}r_{2;n-1}^{-1}(a_{1;n-1}a_1^{-1})\}^{-1}(\bar{s}_{j,n-1}d_j^{-1})\\
		\{(\bar{s}_{l,n-1}d_l^{-1})^{-1}r_{1;n-1}^{-1}r_{2;n-1}^{-1}(a_{1;n-1}a_1^{-1})\} \hspace{0.5cm} \text{for }m=1,l>j,\\
		\{(a_{m-1;n-1}a_{m-1}^{-1})r_{m-1;n-1}\cdots(a_{2;n-1}a_2^{-1})r_{2;n-1}^{-1}(a_{1;n-1}a_1^{-1})\}^{-1}\\
		(\bar{s}_{j,n-1;m+1}d_j^{-1})\\
		\{(a_{m-1;n-1}a_{m-1}^{-1})r_{m-1;n-1}\cdots(a_{2;n-1}a_2^{-1})r_{2;n-1}^{-1}(a_{1;n-1}a_1^{-1})\} \\
\text{for }m\geq2,l=j,\\
		\{(a_{m-1;n-1}a_{m-1}^{-1})r_{m-1;n-1}\cdots(a_{2;n-1}a_2^{-1})r_{2;n-1}^{-1}(a_{1;n-1}a_1^{-1})\}^{-1}\\
		\{(\bar{s}_{l,n-1;m}d_l^{-1})^{-1}r_{m;n-1}^{-1}(\bar{s}_{l,n-1;m+1}d_l^{-1})\}^{-1}
		(\bar{s}_{j,n-1;m}d_j^{-1})\\
		\{(\bar{s}_{l,n-1;m}d_l^{-1})^{-1}r_{m;n-1}^{-1}(\bar{s}_{l,n-1;m+1}d_l^{-1})\}\\
		\{(a_{m-1;n-1}a_{m-1}^{-1})r_{m-1;n-1}\cdots(a_{2;n-1}a_2^{-1})r_{2;n-1}^{-1}(a_{1;n-1}a_1^{-1})\} \\
\text{for }m\geq2,l>j,\\
		\bar{s}_{j,n-1}d_j^{-1}\hspace{0.5cm}\text{for the other cases},
		\end{array}
		\right.$
\item[(D4e)$^\prime $] $r_{m;l}(\bar{s}_{j,n-1}d_j^{-1})r_{m;l}^{-1}=\\
		\left\{
		\begin{array}{ll}
		\{r_{1;n-1}^{-1}(\bar{s}_{l,n-1}d_l^{-1})^{-1}(s_{l,n-1}d_l^{-1})r_{1;n-1}\}^{-1}(\bar{s}_{j,n-1}d_j^{-1})\\
		\{r_{1;n-1}^{-1}(\bar{s}_{l,n-1}d_l^{-1})^{-1}(s_{l,n-1}d_l^{-1})r_{1;n-1}\}\hspace{0.5cm}\text{for }m=1,l<j,\\
		\{(s_{j,n-1}d_j^{-1})r_{1;n-1}\}^{-1}(\bar{s}_{j,n-1}d_j^{-1})\{(s_{j,n-1}d_j^{-1})r_{1;n-1}\} \\
\text{for }m=1,l=j,\\
		\{(\bar{s}_{l,n-1}d_l^{-1})^{-1}r_{1;n-1}^{-1}(s_{l,n-1}d_l^{-1})r_{1;n-1}\}^{-1}(\bar{s}_{j,n-1}d_j)^{-1}\\
		\{(\bar{s}_{l,n-1}d_l^{-1})^{-1}r_{1;n-1}^{-1}(s_{l,n-1}d_l^{-1})r_{1;n-1}\} \hspace{0.5cm}\text{for }m=1,l>j,\\
		\{(a_{m-1;n-1}a_{m-1}^{-1})r_{m-1;n-1}\cdots(a_{2;n-1}a_2^{-1})r_{2;n-1}^{-1}(a_{1;n-1}a_1^{-1})\}^{-1}\\
		(\bar{s}_{j,n-1;m}d_j^{-1})\\
		\{(a_{m-1;n-1}a_{m-1}^{-1})r_{m-1;n-1}\cdots(a_{2;n-1}a_2^{-1})r_{2;n-1}^{-1}(a_{1;n-1}a_1^{-1})\} \\
\text{for }m\geq2,l=j,\\
		\{(a_{m-1;n-1}a_{m-1}^{-1})r_{m-1;n-1}\cdots(a_{2;n-1}a_2^{-1})r_{2;n-1}^{-1}(a_{1;n-1}a_1^{-1})\}^{-1}\\
		\{(\bar{s}_{l,n-1;m}d_l^{-1})^{-1}r_{m;n-1}^{-1}(\bar{s}_{l,n-1;m}d_l^{-1})\}^{-1}\\
		(\bar{s}_{j,n-1;m}d_j^{-1})\\
		\{(\bar{s}_{l,n-1;m}d_l^{-1})^{-1}r_{m;n-1}^{-1}(\bar{s}_{l,n-1;m}d_l^{-1})\}\\
		\{(a_{m-1;n-1}a_{m-1}^{-1})r_{m-1;n-1}\cdots(a_{2;n-1}a_2^{-1})r_{2;n-1}^{-1}(a_{1;n-1}a_1^{-1})\} \\
\text{for }m\geq2,l>j,\\
		\bar{s}_{j,n-1}d_j^{-1}\hspace{0.5cm}\text{for }m\geq2,l<j,
		\end{array}
		\right.$
\item[(D4f)$^\prime $] $s_{l,t}(\bar{s}_{j,n-1}d_j^{-1})s_{l,t}^{-1}=\\
		\left\{
		\begin{array}{ll}
		(\bar{s}_{l,n-1}d_l^{-1})^{-1}(\bar{s}_{j,n-1}d_j^{-1})(\bar{s}_{l,n-1}d_l^{-1})\hspace{0.5cm}\text{for }t=j,\\
		{[(\bar{s}_{t,n-1}d_t^{-1})^{-1},(\bar{s}_{l,n-1}d_l^{-1})^{-1}]}^{-1}(\bar{s}_{j,n-1}d_j^{-1})\\
		{[(\bar{s}_{t,n-1}d_t^{-1})^{-1},(\bar{s}_{l,n-1}d_l^{-1})^{-1}]}\hspace{0.5cm}\text{for }l<j<t,\\
		\{(\bar{s}_{t,n-1}d_t^{-1})(\bar{s}_{j,n-1}d_j^{-1})\}^{-1}(\bar{s}_{j,n-1}d_j^{-1})\{(\bar{s}_{t,n-1}d_t^{-1})(\bar{s}_{j,n-1}d_j^{-1})\} \\
\text{for }l=j,\\
		\bar{s}_{j,n-1}d_j^{-1}\hspace{0.5cm}\text{for the other cases},
		\end{array}
		\right.$
\item[(D4g)$^\prime $] $\bar{s}_{l,t}(\bar{s}_{j,n-1}d_j^{-1})\bar{s}_{l,t}^{-1}=\\
		\left\{
		\begin{array}{ll}
		{[(\bar{s}_{t,n-1}d_t^{-1}),r_{1;n-1}^{-1}(s_{l,n-1}d_l^{-1})^{-1}r_{1;n-1}]}^{-1}(\bar{s}_{j,n-1}d_j^{-1})\\
		{[(\bar{s}_{t,n-1}d_t^{-1}),r_{1;n-1}^{-1}(s_{l,n-1}d_l^{-1})^{-1}r_{1;n-1}]} \hspace{0.5cm}\text{for }t<j,\\
		\{r_{1;n-1}^{-1}(s_{l,n-1}d_l^{-1})^{-1}r_{1;n-1}\}(\bar{s}_{j,n-1}d_j^{-1})\{r_{1;n-1}^{-1}(s_{l,n-1}d_l^{-1})^{-1}r_{1;n-1}\}^{-1}\\
\text{for }t=j,\\
		(s_{t,n-1}d_t^{-1})(\bar{s}_{j,n-1}d_j^{-1})(s_{t,n-1}d_t^{-1})^{-1}\hspace{0.5cm}\text{for }l=j,\\
		{[(\bar{s}_{l,n-1}d_l^{-1})^{-1},(s_{t,n-1}d_t^{-1})]}^{-1}(\bar{s}_{j,n-1}d_j^{-1}){[(\bar{s}_{l,n-1}d_l^{-1})^{-1},(s_{t,n-1}d_t^{-1})]} \\
\text{for }l>j,\\
		\bar{s}_{j,n-1}d_j^{-1}\hspace{0.5cm}\text{for }l<j<t,
		\end{array}
		\right.$
\item[(D4h)$^\prime $] $d_l(\bar{s}_{j,n-1}d_j^{-1})d_l^{-1}=\bar{s}_{j,n-1}d_j^{-1}$.
\end{enumerate}
Since $j_{x_0}(\tilde{f}(x))=\tilde{f}j_{x_0}(x)\tilde{f}^{-1}$ for $x\in \pi _1(N_{g,n-1}, x_0)$ and $f\in \mathcal{M}(N_{g,n-1})$, it is useful for obtaining $w_{x,f}$ to find a product $w_{x,f}^\prime $ of generators for $\pi _1(N_{g,n-1})^+$ in Lemma~\ref{gen_pi+} such that $\tilde{f}(x)=w_{x,f}^\prime $.
For example, in Relation~(D1e)$^\prime $ for $m=i$, $r_{i;l}(x_{i+1}x_i)$ is represented by the loop as on the right-hand side of Figure~\ref{rel_d1e2prime}. Thus we have 
\begin{eqnarray*}
r_{i;l}(x_{i+1}x_i)&=&\{ x_i^{-2}y_l^{-1}x_i^2\} y_l(x_{i+1}x_i)\bar{y}_{i;l}^{-1}\{ x_i^{-2}y_l^{-1}x_i^2\} ^{-1} \in \pi _1(N_{g,n-1})^+
\end{eqnarray*}
and also have
\begin{eqnarray*}
r_{i;l}(a_{i;n-1}a_i^{-1})r_{i;l}^{-1}&=&\{ r_{i;n-1}^{-1}(s_{l,n-1}d_l^{-1})^{-1}r_{i;n-1}\} (s_{l,n-1}d_l^{-1})(a_{i;n-1}a_i^{-1})\\
&&(\bar{s}_{l,n-1;i}d_l^{-1})^{-1}\{ r_{i;n-1}^{-1}(s_{l,n-1}d_l^{-1})^{-1}r_{i;n-1}\} ^{-1}.
\end{eqnarray*}
Therefor $\mathcal{M}^+(N_{g,n-1},x_0)$ has the presentation which is obtained from the finite presentation for $\mathcal{M}(N_{g,n-1})$ by adding generators $a_{i;n-1}$ for $1\leq i\leq g-1$, $r_{i;n-1}$ for $1\leq i\leq g$, $s_{i,n-1}$ and $\bar{s}_{i,n-1}$ for $1\leq i\leq n-2$, and Relations~(D1a)$^\prime $-(D4h)$^\prime $.

Next, we apply Lemma~\ref{presentation_exact} to the exact sequence~\ref{exact2}. Let $X_2^\prime $ be the set of added generators $a_{i;n-1}$, $r_{i;n-1}$, $s_{i;n-1}$ and $\bar{s}_{i;n-1}$ to the presentation for $\mathcal{M}^+(N_{g,n-1},x_0)$. Note that $d_{n-1}$ commutes with every element of $\mathcal{M}(N_{g,n})$. $\mathcal{M}(N_{g,n})$ admits the presentation which has the generating set $X_1\cup X_2^\prime $ and relations as follows: 
\begin{enumerate}
\item[(1)] For each relator $r$ of the finite presentation for $\mathcal{M}^+(N_{g,n-1},x_0)$ and the lift $\tilde{r}$ of $r$ with respect to $\iota _\ast $, there exists $\varepsilon _r\in \mathbb Z$ such that 
\[
\tilde{r}=d_{n-1}^{\varepsilon _r}.
\] 
\item[(2)] For each generator $x$ of the finite presentation for $\mathcal{M}^+(N_{g,n-1},x_0)$ and the lift $\tilde{x}$ of $x$ with respect to $\iota _\ast $, 
\[
[d_{n-1}, \tilde{x}]=1.
\]
\end{enumerate} 
Thus we obtain Relations~(D0) from the relations in (2) above. Note that Relations~(D0) are obtained from Relation~(I).

We compute $\varepsilon _r$ for each $r$. We have the homomorphism $\iota ^\prime _\ast :\mathcal{M}(N^\prime )\hookrightarrow \mathcal{M}(N_{g,n})$ induced by the inclusion $\iota ^\prime :N^\prime \hookrightarrow N_{g,n}$. Thus for each relator $r$ of the finite presentation for $\mathcal{M}^+(N_{g,n-1},x_0)$ which obtained from a relator of the finite presentation for $\mathcal{M}(N_{g,n-1})$, we have $\varepsilon _r=0$. Hence we obtain Relations~(A1)-(B8) and (D0)-(D4g) for $1\leq k\leq n-2$ in $\mathcal{M}(N_{g,n})$. By the inductive argument on $n$ and the argument below, we can show that Relations~(A1)-(B8) and (D0)-(D4g) for $1\leq k\leq n-2$ are obtained from Relations~(I) and (I\hspace{-0.04cm}I\hspace{-0.04cm}I). 

We remark that the lift of Relations~(D1h)$^\prime $, (D2h)$^\prime $, (D3h)$^\prime $ and (D4h)$^\prime $ are obtained from Relations~(D0). Hence Relations ~(D1h), (D2h), (D3h) and (D4h) are obtained from Relation~(I).

Relations~(D1a)$^\prime $ for $m\not=i-1,i+1$, (D1b)$^\prime $ for $i\geq 3$, (D1c)$^\prime $ for $i\not=4$, (D1d)$^\prime $ for $m\leq i-2$, $m=i$ and $m\geq i+2$, (D1e)$^\prime $ for $m\leq i-1$ and $m\geq i-2$, (D1f)$^\prime $, (D1g)$^\prime $ for $i\geq 2$, (D1h)$^\prime $, (D2a)$^\prime $ for $m\not=i-1$, (D2b)$^\prime $ for $i\not=2$, (D2c)$^\prime $ for $i\geq 5$, (D2d)$^\prime $ for $m\leq i-2$ and $m\geq i+1$, (D2e)$^\prime $, (D2f)$^\prime $, (D2g)$^\prime $ for $i\geq 2$, (D2h)$^\prime $ and (D3a)$^\prime $-(D4h)$^\prime $ are obtained from the braid relations. Hence for their relators $r$'s, we have $\varepsilon _r=0$. Thus we obtain Relations~(D1a) for $m\not=i-1,i+1$, (D1b) for $i\geq 3$, (D1c) for $i\not=4$, (D1d) for $m\leq i-2$, $m=i$ and $m\geq i+2$, (D1e) for $m\leq i-1$ and $m\geq i-2$, (D1f), (D1g) for $i\geq 2$, (D2a) for $m\not=i-1$, (D2b) for $i\not=2$, (D2c) for $i\geq 5$, (D2d) for $m\leq i-2$ and $m\geq i+1$, (D2e), (D2f), (D2g) for $i\geq 2$ and (D3a)-(D4g) when $k=n-1$ by Relation~(I). 

For the other relators, i.e. Relations~(D1a)$^\prime $ for $m=i-1,i+1$, (D1b)$^\prime $ for $i=1, 2$, (D1c)$^\prime $ for $i=4$, (D1d)$^\prime $ for $m=i-1, i+1$, (D1e)$^\prime $ for $m=i, i+1$, (D1g)$^\prime $ for $i=1$, (D2a)$^\prime $ for $m=i-1$, (D2b)$^\prime $ for $i=2$, (D2c)$^\prime $ for $i=1,2,3,4$, (D2d)$^\prime $ for $m=i-1,i$ and (D2g)$^\prime $ for $i=1$ when $k=n-1$, we compute their $\varepsilon _r$ in Section~\ref{section_epsilon_r}. By the computations in Section~\ref{section_epsilon_r}, we can show that the remaining relations of presentation in Proposition~\ref{thm_finite_presentation} are obtained from Relations~(I) and ($\text{I\hspace{-0.04cm}I\hspace{-0.04cm}I}$) and we have completed the proof of Proposition~\ref{thm_finite_presentation} except for computations of their $\varepsilon _r$'s.

\end{proof}

\begin{figure}[h]
\includegraphics[scale=0.7]{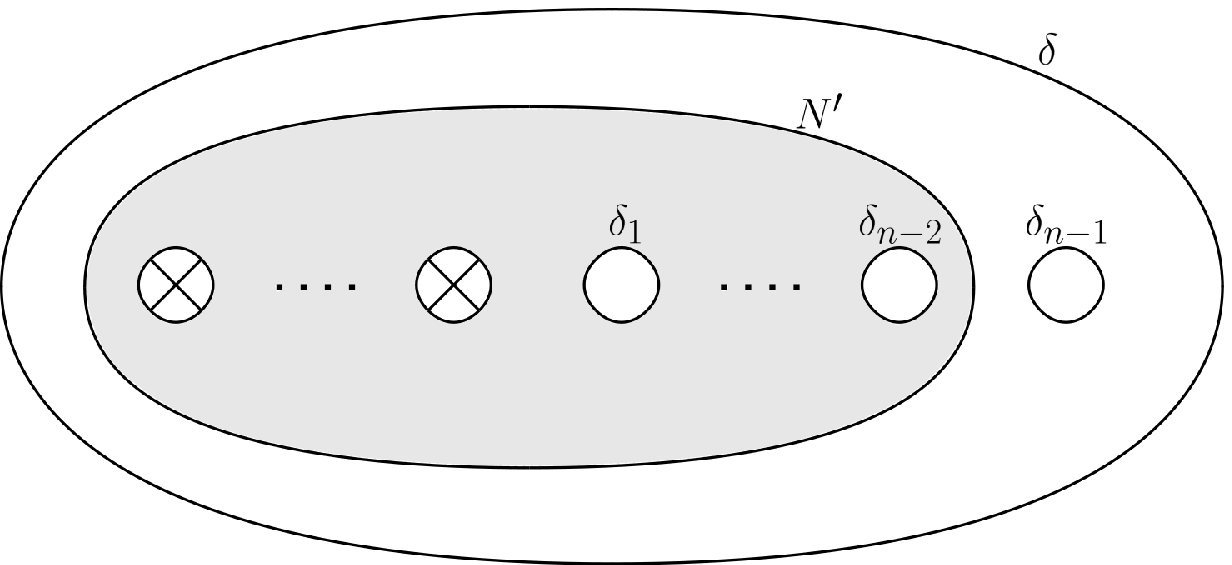}
\caption{The subsurface $N^\prime $ of $N_{g,n}$ which is diffeomorphic to $N_{g,n-1}$.}\label{nonorisurf_emb2_boundary}
\end{figure}
\begin{figure}[h]
\includegraphics[scale=0.7]{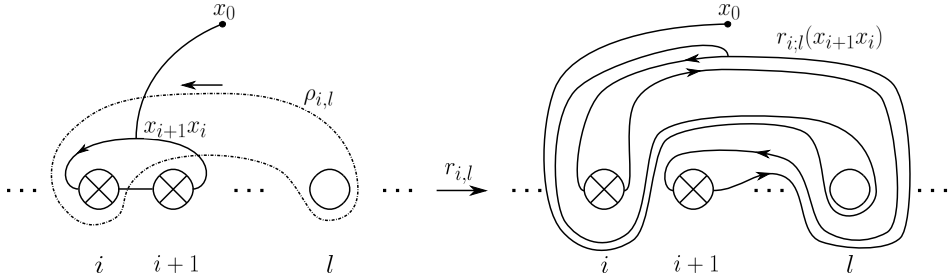}
\caption{Loop $r_{i;l}(x_{i+1}x_i)$ on $N_{g,n-1}$ based at $x_0$ for $1\leq i\leq n-2$.}\label{rel_d1e2prime}
\end{figure}

\subsection{Computing $\varepsilon _r$}\label{section_epsilon_r}

In this section, we compute $\varepsilon _r$ for Relators~(D1a)$^\prime $ for $m=i-1,i+1$, (D1b)$^\prime $ for $i=1, 2$, (D1c)$^\prime $ for $i=4$, (D1d)$^\prime $ for $m=i-1, i+1$, (D1e)$^\prime $ for $m=i, i+1$, (D1g)$^\prime $ for $i=1$, (D2a)$^\prime $ for $m=i-1$, (D2b)$^\prime $ for $i=2$, (D2c)$^\prime $ for $i=1,2,3,4$, (D2d)$^\prime $ for $m=i-1,i$ and (D2g)$^\prime $ for $i=1$ when $k=n-1$. We define 
\begin{eqnarray*}
\bar{y}_i&:=&x_1^{-1}y_ix_1 \hspace{0.3cm}\text{ for } i=1,\dots ,n-2,\\ 
\bar{y}_{i;j}&:=&\{ (x_2x_1)^{-1}x_2^2\cdots (x_ix_{i-1})^{-1}x_i^2\} ^{-1}\bar{y}_i\{ (x_2x_1)^{-1}x_2^2\cdots (x_ix_{i-1})^{-1}x_i^2\} \\
&&\text{for } 1\leq i\leq n-2,\ 2\leq j\leq g. 
\end{eqnarray*}
Remark that $\bar{y}_{i;j}=x_j^{-1}y_ix_j$. By using Relations~(L+), (L-) and (L0) (see Lemma~\ref{L+} and \ref{L0}), when $k=n-1$, we can compute $\varepsilon _r$ as follows: $\varepsilon _r=2$ for Relation~(D1e)$^\prime $ when $m=i+1$, $\varepsilon _r=1$ for Relation~(D1d)$^\prime $ when $m=i+1$, (D1h)$^\prime $ when $m=i-1$ and (D2b)$^\prime $ when $i=2$, $\varepsilon _r=-1$ for Relation~(D1b)$^\prime $ when $i=2$, (D1c)$^\prime $ when $i=4$ and (D1d)$^\prime $, $\varepsilon _r=-2$ for Relation~(D1b)$^\prime $ when $i=1$, (D1e)$^\prime $ when $m=i$ and $\varepsilon _r=0$ for the other relations. In particular, Relations~(D1a) when $m=i-1,i+1$ and (D1b) when $i=2$ are obtained from a single braid relation and one of Relations~(L+), (L-) and (L0), and for Relations~(D1b) when $i=1$, (D1c) when $i=4$, (D1d) when $m=i+1$, (D2a) when $m=i-1$, (D2b) when $i=2$, (D2d) when $m=i$,  we can compute $\varepsilon _r$ easily. As examples, we compute $\varepsilon _r$ for Relation~(D1e)$^\prime $ when $m=i$ and Relation~(D2c)$^\prime $ when $i=2$ by using figures. For the other cases, we give computations of $\varepsilon _r$ by only deformations of the expressions.

For Relation~(D1e)$^\prime $ when $m=i$ and $k=n-1$, we have the following relation in $\mathcal{M}(N_{g,n})$ by Figure~\ref{rel_d1e2prime_epsilon1}:
\begin{eqnarray*}
\underline{\Delta (y_l)\Delta ((x_{i+1}x_i))}\Delta (\bar{y}_{i;l})^{-1}
&\stackrel{(\text{L+})}{=}& \underline{\Delta ((y_l(x_{i+1}x_i))\Delta (\bar{y}_{i;l}^{-1})}d_{n-1} \\
&\stackrel{(\text{L+})}{=}& \Delta (y_l(x_{i+1}x_i)\bar{y}_{i;l}^{-1})d_{n-1}^2 .
\end{eqnarray*}
Note that $\Delta (r_{i;l}(x_ix_{i+1}))=r_{i;l}(a_{i;n-1}a_i^{-1})r_{i;l}^{-1}$ by the braid relation. Hence we have 
\begin{eqnarray*}
&&\{ r_{i;n-1}^{-1}(s_{l,n-1}d_l^{-1})^{-1}r_{i;n-1}\} (s_{l,n-1}d_l^{-1})(a_{i;n-1}a_i^{-1})(\bar{s}_{l,n-1;i}d_l^{-1})^{-1}\\
&&\{ r_{i;n-1}^{-1}(s_{l,n-1}d_l^{-1})^{-1}r_{i;n-1}\} ^{-1}\\
&=&\{ r_{i;n-1}^{-1}(s_{l,n-1}d_l^{-1})^{-1}r_{i;n-1}\} \Delta (y_l)\Delta ((x_{i+1}x_i))\Delta (\bar{y}_{i;l})^{-1}\\
&&\{ r_{i;n-1}^{-1}(s_{l,n-1}d_l^{-1})^{-1}r_{i;n-1}\} ^{-1}\\
&=&\{ r_{i;n-1}^{-1}(s_{l,n-1}d_l^{-1})^{-1}r_{i;n-1}\} \Delta (y_l(x_{i+1}x_i)\bar{y}_{i;l}^{-1})d_{n-1}^2\\
&&\{ r_{i;n-1}^{-1}(s_{l,n-1}d_l^{-1})^{-1}r_{i;n-1}\} ^{-1}\\
&\stackrel{(\text{I})}{=}& \Delta (\{ r_{i;n-1}^{-1}(s_{l,n-1}d_l^{-1})^{-1}r_{i;n-1}\} (y_l(x_{i+1}x_i)\bar{y}_{i;l}^{-1}))d_{n-1}^2\\
&=& \Delta (\{ x_i^{-2}y_l^{-1}x_i^2\} y_l(x_{i+1}x_i)\bar{y}_{i;l}^{-1}\{ x_i^{-2}y_l^{-1}x_i^2\} ^{-1})d_{n-1}^2\\
&=& \Delta (r_{i;l}(x_ix_{i+1}))d_{n-1}^2\\
&\stackrel{(\text{I})}{=}& r_{i;l}(a_{i;n-1}a_i^{-1})r_{i;l}^{-1}d_{n-1}^2.
\end{eqnarray*}
Thus $\varepsilon _r=-2$ for Relation~(D1e)$^\prime $ when $m=i$ and $k=n-1$, and we obtain Relation~(D1e) for $m=i$ and $k=n-1$ by Relations~(I) and (L+). Hence Relation~(D1e) for $m=i$ and $k=n-1$ is obtained from Relations~(I) and ($\text{I\hspace{-0.04cm}I\hspace{-0.04cm}I}$).

\begin{figure}[h]
\includegraphics[scale=0.7]{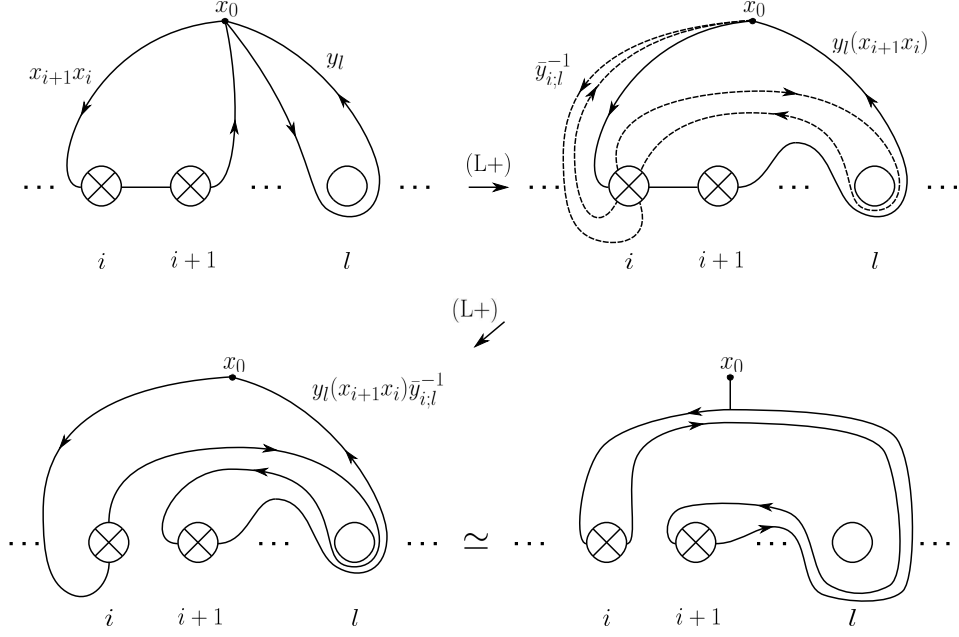}
\caption{On the upper side, we explain that the relation $\Delta (y_l)\Delta ((x_{i+1}x_i))=\Delta ((y_l(x_{i+1}x_i))d_{n-1}$ is obtained from Relation~(L+). As a similarity, the arrow from the upper right side to lower left side explain that the relation $\Delta ((y_l(x_{i+1}x_i))\Delta (\bar{y}_{i;l}^{-1})=\Delta (y_l(x_{i+1}x_i)\bar{y}_{i;l}^{-1})d_{n-1}$ is obtained from Relation~(L+). ``$\simeq $'' means a deformation of the loop by a homotopy fixing $x_0$.}\label{rel_d1e2prime_epsilon1}
\end{figure}

For Relation~(D2c) when $i=2$ and $k=n-1$, we note that $b(x_2^2)$ is represented by a loop as on the lower right side of Figure~\ref{rel_d2c2prime_epsilon5} and $\Delta (b(x_2^2))=br_{2;n-1}b^{-1}$ by the braid relation. By Figure~\ref{rel_d2c2prime_epsilon1}, we have 
\begin{eqnarray*}
&&(a_{1;n-1}a_1^{-1})(a_{3;n-1}a_3^{-1})r_{2;n-1}r_{1;n-1}(a_{1;n-1}a_1^{-1})^{-1}\\
&=&(a_{1;n-1}a_1^{-1}) \Delta (x_4x_3)\underline{\Delta (x_2^2)\Delta (x_1^2)}(a_{1;n-1}a_1^{-1})^{-1}\\
&\stackrel{(\text{L+})}{=}&(a_{1;n-1}a_1^{-1}) \underline{\Delta (x_4x_3)\Delta (x_2^2x_1^2)}(a_{1;n-1}a_1^{-1})^{-1}d_{n-1}\\
&\stackrel{(\text{L+})}{=}&(a_{1;n-1}a_1^{-1}) \Delta ((x_4x_3)x_2^2x_1^2)(a_{1;n-1}a_1^{-1})^{-1}d_{n-1}^2\\
&\stackrel{(\text{I})}{=}& \Delta ((a_{1;n-1}a_1^{-1})((x_4x_3)x_2^2x_1^2))d_{n-1}^2
\end{eqnarray*}
in $\mathcal{M}(N_{g,n})$. By Figure~\ref{rel_d2c2prime_epsilon2} and Figure~\ref{rel_d2c2prime_epsilon3}, we have the following relation in $\mathcal{M}(N_{g,n})$.
\begin{eqnarray*}
&&(a_{2;n-1}a_2^{-1})^{-1}r_{3;n-1}(a_{3;n-1}a_3^{-1})^{-1}r_{4;n-1}(a_{2;n-1}a_2^{-1})\\
&=&(a_{2;n-1}a_2^{-1})^{-1}r_{3;n-1}r_{4;n-1}r_{4;n-1}^{-1}(a_{3;n-1}a_3^{-1})^{-1}r_{4;n-1}(a_{2;n-1}a_2^{-1})\\
&=&(a_{2;n-1}a_2^{-1})^{-1}\underline{\Delta (x_3^2)\Delta (x_4^2)} \underline{r_{4;n-1}^{-1}\Delta (x_4x_3)^{-1}r_{4;n-1}}(a_{2;n-1}a_2^{-1})\\
&\stackrel{(\text{L+}),(\text{I})}{=}&(a_{2;n-1}a_2^{-1})^{-1}\underline{\Delta (x_3^2x_4^2) \Delta (r_{4;n-1}^{-1}((x_4x_3)^{-1}))}(a_{2;n-1}a_2^{-1})d_{n-1}\\
&\stackrel{(\text{L-})}{=}&(a_{2;n-1}a_2^{-1})^{-1}\Delta (x_3^2x_4^2r_{4;n-1}^{-1}((x_4x_3)^{-1}))(a_{2;n-1}a_2^{-1})\\
&\stackrel{(\text{I})}{=}&\Delta ((a_{2;n-1}a_2^{-1})^{-1}(x_3^2x_4^2r_{4;n-1}^{-1}((x_4x_3)^{-1}))). 
\end{eqnarray*}
Let $\zeta _1$ and $\zeta _2$ be simple closed curves on $N_{g,n}$ as in Figure~\ref{scc_zeta1zeta2_nonori1}. Since $\iota _\ast (t_{\zeta _1})\in \mathcal{M}^+(N_{g,n-1},x_0)$ fixes $x_2^2$ and $\Delta ((a_{1;n-1}a_1^{-1})((x_4x_3)x_1^2x_1^2))=t_{\zeta _1}t_{\zeta _2}^{-1}$, we have $\Delta ((a_{1;n-1}a_1^{-1})((x_4x_3)x_1^2x_1^2))(x_2^2)=t_{\zeta _2}^{-1}(x_2^2)$. We remark that the loop as on the upper right side of Figure~\ref{rel_d2c2prime_epsilon4} is homotopic to the loop as on the lower right side of Figure~\ref{rel_d2c2prime_epsilon4} by a homotopy fixing $x_0$ as in Figure~\ref{rel_d2c2prime_epsilon4}. By the relations above and Figure~\ref{rel_d2c2prime_epsilon5}, we have
\begin{eqnarray*}
&&\{ (a_{3;n-1}a_3^{-1})(a_{1;n-1}a_1^{-1})\} ^{-1}\underline{(a_{1;n-1}a_1^{-1})(a_{3;n-1}a_3^{-1})r_{2;n-1}r_{1;n-1}}\\
&&\underline{(a_{1;n-1}a_1^{-1})^{-1}}r_{2;k}(a_{2;n-1}a_2^{-1})^{-1}r_{3;n-1}(a_{3;n-1}a_3^{-1})^{-1}r_{4;n-1}\\
&&(a_{2;n-1}a_2^{-1})\{ (a_{3;n-1}a_3^{-1})(a_{1;n-1}a_1^{-1})\}\\
&\stackrel{(\text{L+}),(\text{I})}{=}&\{ (a_{3;n-1}a_3^{-1})(a_{1;n-1}a_1^{-1})\} ^{-1}\Delta ((a_{1;n-1}a_1^{-1})((x_4x_3)x_2^2x_1^2))r_{2;k}\\
&&\underline{(a_{2;n-1}a_2^{-1})^{-1}r_{3;n-1}(a_{3;n-1}a_3^{-1})^{-1}r_{4;n-1}(a_{2;n-1}a_2^{-1})}\\
&&\{ (a_{3;n-1}a_3^{-1})(a_{1;n-1}a_1^{-1})\} d_{n-1}^2\\
&\stackrel{(\text{L+}),(\text{L-}),(\text{I})}{=}&\{ (a_{3;n-1}a_3^{-1})(a_{1;n-1}a_1^{-1})\} ^{-1}\Delta ((a_{1;n-1}a_1^{-1})((x_4x_3)x_2^2x_1^2))r_{2;k}\\
&&\Delta ((a_{2;n-1}a_2^{-1})^{-1}(x_3^2x_4^2r_{4;n-1}^{-1}((x_4x_3)^{-1})))\{ (a_{3;n-1}a_3^{-1})(a_{1;n-1}a_1^{-1})\} \\
&&d_{n-1}^2\\
&=&\{ (a_{3;n-1}a_3^{-1})(a_{1;n-1}a_1^{-1})\} ^{-1}\underline{\Delta ((a_{1;n-1}a_1^{-1})((x_4x_3)x_2^2x_1^2))\Delta (x_2^2)}\\
&&\underline{\Delta ((a_{1;n-1}a_1^{-1})((x_4x_3)x_2^2x_1^2))^{-1}}\Delta ((a_{1;n-1}a_1^{-1})((x_4x_3)x_2^2x_1^2))\\
&&\Delta ((a_{2;n-1}a_2^{-1})^{-1}(x_3^2x_4^2r_{4;n-1}^{-1}((x_4x_3)^{-1})))\{ (a_{3;n-1}a_3^{-1})(a_{1;n-1}a_1^{-1})\} \\
&&d_{n-1}^2\\
&\stackrel{(\text{I})}{=}&\{ (a_{3;n-1}a_3^{-1})(a_{1;n-1}a_1^{-1})\} ^{-1}\Delta (t_{\zeta _2}^{-1}(x_2^2))\underline{\Delta ((a_{1;n-1}a_1^{-1})((x_4x_3)x_2^2x_1^2))}\\
&&\underline{\Delta ((a_{2;n-1}a_2^{-1})^{-1}(x_3^2x_4^2r_{4;n-1}^{-1}((x_4x_3)^{-1})))}\{ (a_{3;n-1}a_3^{-1})(a_{1;n-1}a_1^{-1})\} \\
&&d_{n-1}^2\\
&\stackrel{(\text{L-})}{=}&\{ (a_{3;n-1}a_3^{-1})(a_{1;n-1}a_1^{-1})\} ^{-1}\underline{\Delta (t_{\zeta _2}^{-1}(x_2^2))}\\
&&\underline{\Delta ((a_{1;n-1}a_1^{-1})((x_4x_3)x_2^2x_1^2)\cdot (a_{2;n-1}a_2^{-1})^{-1}(x_3^2x_4^2r_{4;n-1}^{-1}((x_4x_3)^{-1})))}\\
&&\{ (a_{3;n-1}a_3^{-1})(a_{1;n-1}a_1^{-1})\} d_{n-1}\\
&\stackrel{(\text{L-})}{=}&\{ (a_{3;n-1}a_3^{-1})(a_{1;n-1}a_1^{-1})\} ^{-1}\Delta (t_{\zeta _2}^{-1}(x_2^2)\cdot \\
&&(a_{1;n-1}a_1^{-1})((x_4x_3)x_2^2x_1^2)\cdot (a_{2;n-1}a_2^{-1})^{-1}(x_3^2x_4^2r_{4;n-1}^{-1}((x_4x_3)^{-1})))\\
&&\{ (a_{3;n-1}a_3^{-1})(a_{1;n-1}a_1^{-1})\} \\
&\stackrel{(\text{I})}{=}&\Delta (\{ (a_{3;n-1}a_3^{-1})(a_{1;n-1}a_1^{-1})\} ^{-1}(t_{\zeta _2}^{-1}(x_2^2)\cdot \\
&&(a_{1;n-1}a_1^{-1})((x_4x_3)x_2^2x_1^2)\cdot (a_{2;n-1}a_2^{-1})^{-1}(x_3^2x_4^2r_{4;n-1}^{-1}((x_4x_3)^{-1}))))\\
&=&\Delta (b(x_2^2))\\
&\stackrel{(\text{I})}{=}&br_{2;n-1}b^{-1}.
\end{eqnarray*}
Thus $\varepsilon _r=0$ for Relation~(D2c)$\prime $ when $i=2$ and $k=n-1$, and we obtain Relation~(D2c) for $i=2$ and $k=n-1$ by Relations~(I) and (I\hspace{-0.04cm}I\hspace{-0.04cm}I).

\begin{figure}[h]
\includegraphics[scale=0.7]{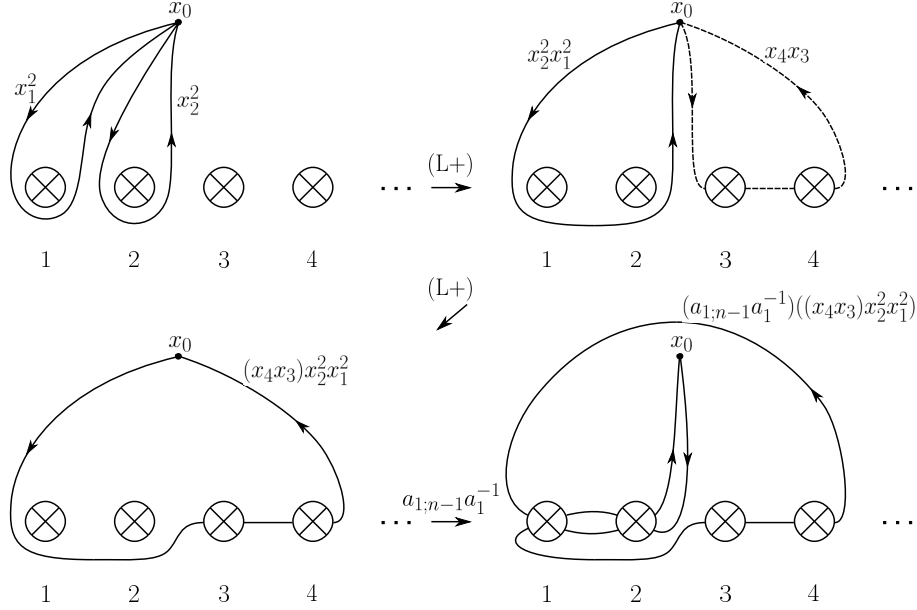}
\caption{Relations $\Delta (x_2^2)\Delta (x_1^2)=\Delta (x_2^2x_1^2)d_{n-1}$ and $\Delta (x_4x_3)\Delta (x_1^2x_2^2)=\Delta ((x_4x_3)x_2^2x_1^2)d_{n-1}$ and loop $(a_{1;n-1}a_1^{-1})((x_4x_3)x_2^2x_1^2)$.}\label{rel_d2c2prime_epsilon1}
\end{figure}

\begin{figure}[h]
\includegraphics[scale=0.7]{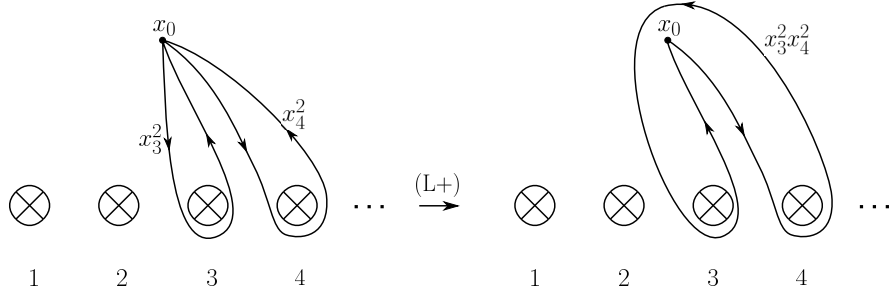}
\caption{Relation $\Delta (x_3^2)\Delta (x_4^2)=\Delta (x_3^2x_4^2)d_{n-1}$.}\label{rel_d2c2prime_epsilon2}
\end{figure}

\begin{figure}[h]
\includegraphics[scale=0.7]{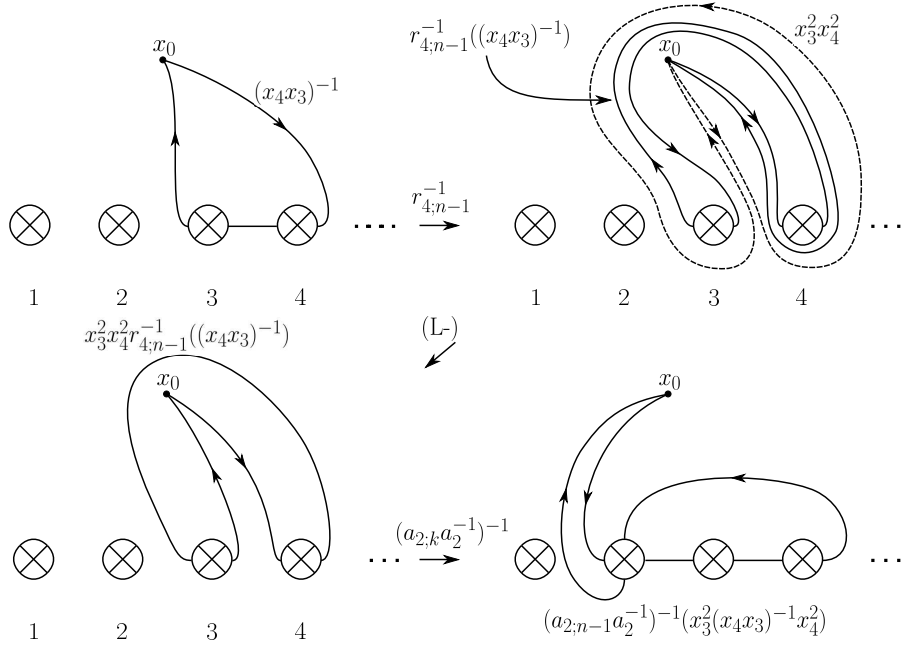}
\caption{Relation $\Delta (x_3^2x_4^2)\Delta (r_{4;n-1}^{-1}((x_4x_3)^{-1}))=\Delta (x_3^2x_4^2r_{4;n-1}^{-1}((x_4x_3)^{-1}))d_{n-1}^{-1}$ and loops $r_{4;n-1}^{-1}((x_4x_3)^{-1})$ and $(a_{2;n-1}a_2^{-1})^{-1}(x_3^2x_4^2r_{4;n-1}^{-1}((x_4x_3)^{-1}))$.}\label{rel_d2c2prime_epsilon3}
\end{figure}

\begin{figure}[h]
\includegraphics[scale=0.6]{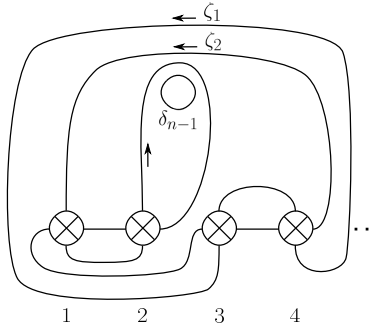}
\caption{Simple closed curves $\zeta _1$ and $\zeta _2$ on $N_{g,n}$.}\label{scc_zeta1zeta2_nonori1}
\end{figure}

\begin{figure}[h]
\includegraphics[scale=0.7]{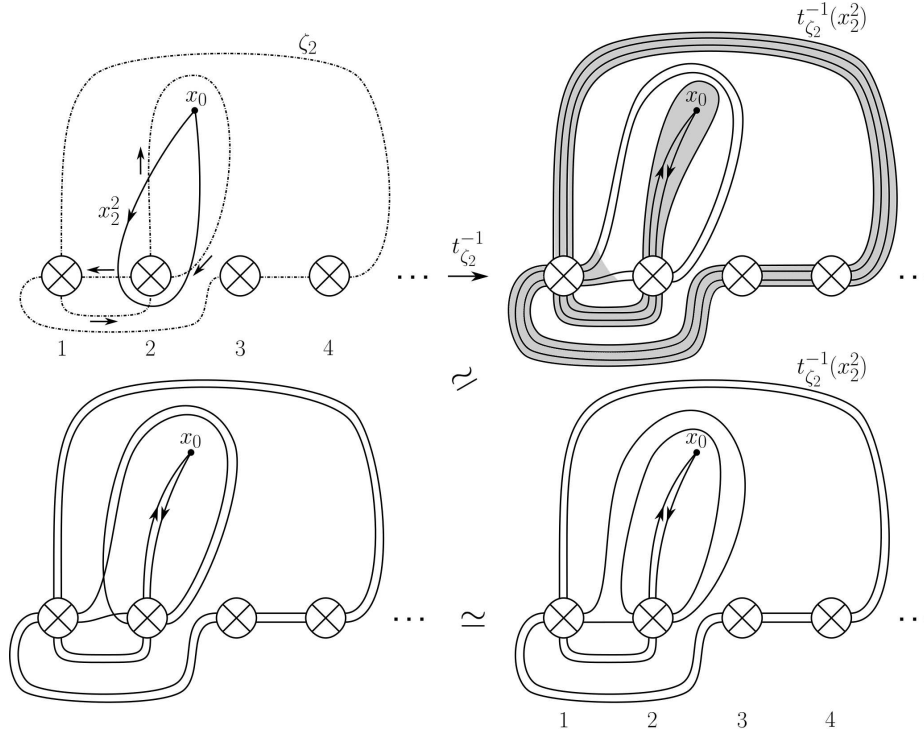}
\caption{Loop $t_{\zeta _2}^{-1}(x_2^2)$.}\label{rel_d2c2prime_epsilon4}
\end{figure}

\begin{figure}[h]
\includegraphics[scale=0.68]{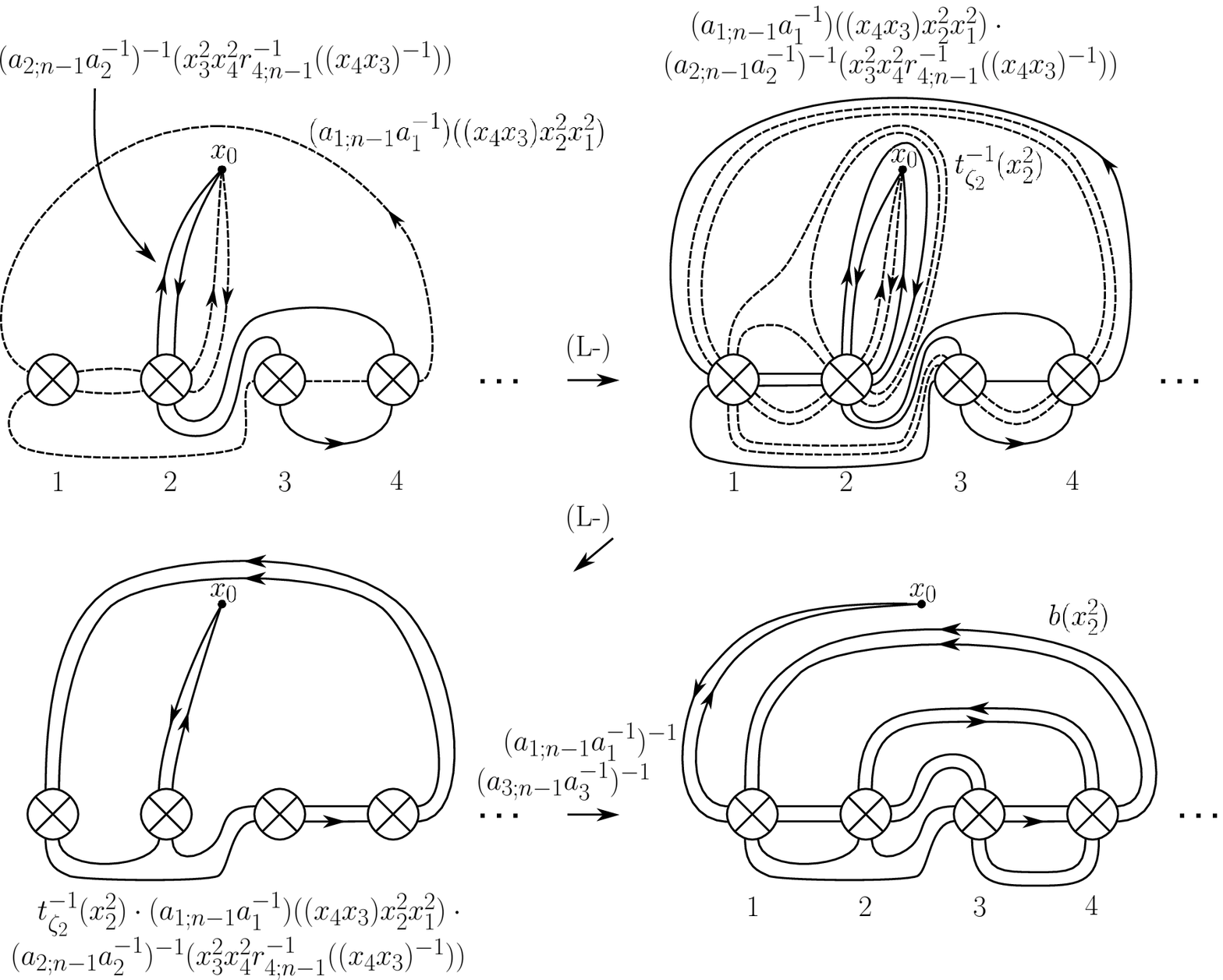}
\caption{Relations $\Delta ((a_{1;n-1}a_1^{-1})((x_4x_3)x_2^2x_1^2))\Delta ((a_{2;n-1}a_2^{-1})^{-1}$ $(x_3^2x_4^2r_{4;n-1}^{-1}((x_4x_3)^{-1})))=\Delta (a_{1;n-1}a_1^{-1})((x_4x_3)x_2^2x_1^2)\cdot $ $(a_{2;n-1}a_2^{-1})^{-1}(x_3^2x_4^2r_{4;n-1}^{-1}((x_4x_3)^{-1})))d_{n-1}^{-1}$ and $\Delta (t_{\zeta _2}^{-1}(x_2^2))\Delta (a_{1;n-1}a_1^{-1})((x_4x_3)x_2^2x_1^2)\cdot (a_{2;n-1}a_2^{-1})^{-1}(x_3^2x_4^2r_{4;n-1}^{-1}((x_4x_3)^{-1})))=$ $\Delta (t_{\zeta _2}^{-1}(x_2^2)\cdot a_{1;n-1}a_1^{-1})((x_4x_3)x_2^2x_1^2)\cdot (a_{2;n-1}a_2^{-1})^{-1}(x_3^2x_4^2r_{4;n-1}^{-1}((x_4x_3)^{-1})))d_{n-1}^{-1}$ and loop $b(x_2^2)$.}\label{rel_d2c2prime_epsilon5}
\end{figure}

For Relation~(D1d) when $m=i-1$ and $k=n-1$, we have
\begin{eqnarray*}
&&{[(a_{i-1;n-1}a_{i-1}^{-1})^{-1},(s_{l,n-1}d_l^{-1})^{-1}]}(a_{i;n-1}a_i^{-1})(s_{l,n-1}d_l^{-1})(a_{i-1;n-1}a_{i-1}^{-1}) \\
&=&\{ (s_{l,n-1}d_l^{-1})(a_{i-1;n-1}a_{i-1}^{-1})\}^{-1}\Delta (x_ix_{i-1})\underline{\Delta (y_l)\Delta (x_{i+1}x_i)}\{ (s_{l,n-1}d_l^{-1})\\
&&(a_{i-1;n-1}a_{i-1}^{-1})\}\\
&\stackrel{(\text{L+})}{=}&\{ (s_{l,n-1}d_l^{-1})(a_{i-1;n-1}a_{i-1}^{-1})\}^{-1}\underline{\Delta (x_ix_{i-1})\Delta (y_l(x_{i+1}x_i))}\{ (s_{l,n-1}d_l^{-1})\\
&&(a_{i-1;n-1}a_{i-1}^{-1})\}d_{n-1} \\
&\stackrel{(\text{L0})}{=}&\{ (s_{l,n-1}d_l^{-1})(a_{i-1;n-1}a_{i-1}^{-1})\}^{-1}\Delta ((x_ix_{i-1})y_l(x_{i+1}x_i))\{ (s_{l,n-1}d_l^{-1})\\
&&(a_{i-1;n-1}a_{i-1}^{-1})\}d_{n-1}\\
&\stackrel{(\text{I})}{=}&\Delta (\{ (s_{l,n-1}d_l^{-1})(a_{i-1;n-1}a_{i-1}^{-1})\}^{-1}((x_ix_{i-1})y_l(x_{i+1}x_i)))d_{n-1}\\
&=&\Delta (a_{i-1;l}(x_{i+1}x_i))d_{n-1}\\
&\stackrel{(\text{I})}{=}&a_{i-1;l}(a_{i;n-1}a_i^{-1})a_{i-1;l}^{-1}d_{n-1}.
\end{eqnarray*}
Thus $\varepsilon _r=-1$ for Relation~(D1d)$\prime $ when $m=i-1$ and $k=n-1$, and we obtain Relation~(D1d) when $m=i-1$ and $k=n-1$ by Relations~(I) and (I\hspace{-0.04cm}I\hspace{-0.04cm}I).

For Relation~(D1e) when $m=i+1$ and $k=n-1$, we have
\begin{eqnarray*}
&&r_{i+1;n-1}^{-1}(s_{l,n-1}d_l^{-1})^{-1}r_{i+1;n-1}(\bar{s}_{l,n-1;i+1}d_l^{-1})(a_{i;n-1}a_i^{-1}) \\
&=&\underline{r_{i+1;n-1}^{-1}\Delta (y_{l})r_{i+1;n-1}}\underline{\Delta (\bar{y}_{l;i+1})\Delta (x_{i+1}x_i)}\\
&\stackrel{(\text{I}),(\text{L-})}{=}&\Delta (r_{i+1;n-1}^{-1}(y_{l}))\Delta (\bar{y}_{l;i+1}(x_{i+1}x_i))d_{n-1}^{-1}\\
&\stackrel{(\text{L-})}{=}&\Delta (r_{i+1;n-1}^{-1}(y_{l})\bar{y}_{l;i+1}(x_{i+1}x_i))d_{n-1}^{-2}\\
&=&\Delta (r_{i+1;l}(x_{i+1}x_i))d_{n-1}^{-2}\\
&\stackrel{(\text{I})}{=}&r_{i+1;l}(a_{i;n-1}a_i^{-1})r_{i+1;l}^{-1}d_{n-1}^{-2}.
\end{eqnarray*}
Thus $\varepsilon _r=-2$ for Relation~(D1e)$\prime $ when $m=i+1$ and $k=n-1$, and we obtain Relation~(D1e) when $m=i+1$ and $k=n-1$ by Relations~(I) and (I\hspace{-0.04cm}I\hspace{-0.04cm}I).

For Relation~(D1g) when $i=1$ and $k=n-1$, we have
\begin{eqnarray*}
&&[(\bar{s}_{l,k}d_l^{-1})^{-1},s_{t,k}d_t^{-1}]^{-1}(s_{l,k}d_l^{-1})(a_{1;k}a_1^{-1})r_{1;k}(\bar{s}_{t,k}d_t^{-1})r_{1;k}^{-1}(s_{l,k}d_l^{-1})^{-1}\\
&&r_{1;k}(\bar{s}_{t,k}d_t^{-1})^{-1}r_{1;k}^{-1}[(\bar{s}_{l,k}d_l^{-1})^{-1},(s_{t,k}d_t^{-1})]\\
&=&[(\bar{s}_{l,k}d_l^{-1})^{-1},s_{t,k}d_t^{-1}]^{-1}\underline{\Delta (y_l)\Delta (x_2x_1)}\underline{r_{1;k}\Delta (\bar{y}_t)r_{1;k}^{-1}}\Delta (y_l)^{-1}\\
&&\underline{r_{1;k}\Delta (\bar{y}_t)^{-1}r_{1;k}^{-1}}[(\bar{s}_{l,k}d_l^{-1})^{-1},(s_{t,k}d_t^{-1})]\\
&\stackrel{(\text{L+}),(\text{I})}{=}&[(\bar{s}_{l,k}d_l^{-1})^{-1},s_{t,k}d_t^{-1}]^{-1}\Delta (y_l(x_2x_1))\underline{\Delta (r_{1;k}(\bar{y}_t))\Delta (y_l)^{-1}\Delta (r_{1;k}(\bar{y}_t))^{-1}}\\
&&[(\bar{s}_{l,k}d_l^{-1})^{-1},(s_{t,k}d_t^{-1})]d_{n-1}\\
&\stackrel{(\text{I})}{=}&[(\bar{s}_{l,k}d_l^{-1})^{-1},s_{t,k}d_t^{-1}]^{-1}\underline{\Delta (y_l(x_2x_1))\Delta (\Delta (r_{1;k}(\bar{y}_t))(y_l))^{-1}}\\
&&[(\bar{s}_{l,k}d_l^{-1})^{-1},(s_{t,k}d_t^{-1})]d_{n-1}\\
&\stackrel{(\text{L-})}{=}&[(\bar{s}_{l,k}d_l^{-1})^{-1},s_{t,k}d_t^{-1}]^{-1}\Delta (y_l(x_2x_1)\Delta (r_{1;k}(\bar{y}_t))(y_l)^{-1})\\
&&[(\bar{s}_{l,k}d_l^{-1})^{-1},(s_{t,k}d_t^{-1})]\\
&\stackrel{(\text{I})}{=}&\Delta ([(\bar{s}_{l,k}d_l^{-1})^{-1},s_{t,k}d_t^{-1}]^{-1}(y_l(x_2x_1)\Delta (r_{1;k}(\bar{y}_t))(y_l)^{-1}))\\
&=&\Delta (\bar{s}_{l;t}(x_{i+1}x_i))\\
&\stackrel{(\text{I})}{=}&\bar{s}_{l;t}(a_{i;n-1}a_i^{-1})\bar{s}_{l;t}^{-1}.
\end{eqnarray*}
Thus $\varepsilon _r=0$ for Relation~(D1g)$\prime $ when $i=1$ and $k=n-1$, and we obtain Relation~(D1g) when $i=1$ and $k=n-1$ by Relations~(I) and (I\hspace{-0.04cm}I\hspace{-0.04cm}I).

For Relation~(D2c) when $i=1$ and $k=n-1$, we have
\begin{eqnarray*}
&&(a_{1;n-1}a_1^{-1})^{-1}(a_{3;n-1}a_3^{-1})^{-1}(a_{2;n-1}a_2^{-1})^{-1}r_{4;n-1}^{-1}(a_{3;n-1}a_3^{-1})\\
&&r_{3;n-1}^{-1}(a_{2;n-1}a_2^{-1})r_{2;n-1}^{-1}(a_{1;n-1}a_1^{-1})\\
&=&(a_{1;n-1}a_1^{-1})^{-1}(a_{3;n-1}a_3^{-1})^{-1}\underline{\Delta (x_3x_2)^{-1}\Delta (x_4^2)^{-1}}(a_{3;n-1}a_3^{-1})\\
&&\underline{r_{3;n-1}^{-1}\Delta (x_3x_2)r_{3;n-1}}\underline{\Delta (x_3^2)^{-1}\Delta (x_2^2)^{-1}}(a_{1;n-1}a_1^{-1})\\
&\stackrel{(\text{L-}),(\text{I})}{=}&(a_{1;n-1}a_1^{-1})^{-1}\underline{(a_{3;n-1}a_3^{-1})^{-1}\Delta ((x_3x_2)^{-1}x_4^{-2})(a_{3;n-1}a_3^{-1})}\\
&&\underline{\Delta (r_{3;n-1}^{-1}(x_3x_2))\Delta (x_3^{-2}x_2^{-2})}(a_{1;n-1}a_1^{-1})d_{n-1}^{-2}\\
&\stackrel{(\text{I}),(\text{L+})}{=}&(a_{1;n-1}a_1^{-1})^{-1}\underline{\Delta ((a_{3;n-1}a_3^{-1})^{-1}((x_3x_2)^{-1}x_4^{-2}))\Delta (r_{3;n-1}^{-1}(x_3x_2)x_3^{-2}x_2^{-2})}\\
&&(a_{1;n-1}a_1^{-1})d_{n-1}^{-1}\\
&\stackrel{(\text{L+})}{=}&(a_{1;n-1}a_1^{-1})^{-1}\Delta ((a_{3;n-1}a_3^{-1})^{-1}((x_3x_2)^{-1}x_4^{-2})r_{3;n-1}^{-1}(x_3x_2)x_3^{-2}x_2^{-2})\\
&&(a_{1;n-1}a_1^{-1})\\
&\stackrel{(\text{I})}{=}&\Delta ((a_{1;n-1}a_1^{-1})^{-1}((a_{3;n-1}a_3^{-1})^{-1}((x_3x_2)^{-1}x_4^{-2})r_{3;n-1}^{-1}(x_3x_2)x_3^{-2}x_2^{-2}))\\
&=&\Delta (b(x_1^2))\\
&\stackrel{(\text{I})}{=}&b(r_{1;n-1})b^{-1}.
\end{eqnarray*}
Thus $\varepsilon _r=0$ for Relation~(D2c)$\prime $ when $i=1$ and $k=n-1$, and we obtain Relation~(D2c) when $i=1$ and $k=n-1$ by Relations~(I) and (I\hspace{-0.04cm}I\hspace{-0.04cm}I).

For Relation~(D2c) when $i=3$ and $k=n-1$, we have
\begin{eqnarray*}
&&\{ (a_{3;n-1}a_3^{-1})(a_{1;n-1}a_1^{-1})\} ^{-1}r_{4;n-1}^{-1}(a_{3;n-1}a_3^{-1})r_{3;n-1}^{-1}(a_{2;n-1}a_2^{-1})\\
&&r_{2;n-1}^{-1}(a_{1;n-1}a_1^{-1})r_{1;n-1}^{-1}(a_{2;n-1}a_2^{-1})^{-1}r_{3;n-1}(a_{3;n-1}a_3^{-1})^{-1}\\
&&(a_{1;n-1}a_1^{-1})^{-1}\{ (a_{3;n-1}a_3^{-1})(a_{1;n-1}a_1^{-1})\}\\
&=&\{ (a_{3;n-1}a_3^{-1})(a_{1;n-1}a_1^{-1})\} ^{-1}r_{4;n-1}^{-1}(a_{3;n-1}a_3^{-1})r_{3;n-1}^{-1}(a_{2;n-1}a_2^{-1})\\
&&\underline{r_{2;n-1}^{-1}\Delta (x_2x_1)r_{2;n-1}}\underline{\Delta (x_2^2)^{-1}\Delta (x_1^2)^{-1}}(a_{2;n-1}a_2^{-1})^{-1}r_{3;n-1}(a_{3;n-1}a_3^{-1})^{-1}\\
&&\Delta (x_2x_1)^{-1}\{ (a_{3;n-1}a_3^{-1})(a_{1;n-1}a_1^{-1})\}\\
&\stackrel{(\text{I}),(\text{L-})}{=}&\{ (a_{3;n-1}a_3^{-1})(a_{1;n-1}a_1^{-1})\} ^{-1}r_{4;n-1}^{-1}(a_{3;n-1}a_3^{-1})r_{3;n-1}^{-1}(a_{2;n-1}a_2^{-1})\\
&&\underline{\Delta (r_{2;n-1}^{-1}(x_2x_1))\Delta (x_2^{-2}x_1^{-2})}(a_{2;n-1}a_2^{-1})^{-1}r_{3;n-1}(a_{3;n-1}a_3^{-1})^{-1}\\
&&\Delta (x_2x_1)^{-1}\{ (a_{3;n-1}a_3^{-1})(a_{1;n-1}a_1^{-1})\}d_{n-1}^{-1}\\
&\stackrel{(\text{L+})}{=}&\{ (a_{3;n-1}a_3^{-1})(a_{1;n-1}a_1^{-1})\} ^{-1}r_{4;n-1}^{-1}\underline{(a_{3;n-1}a_3^{-1})r_{3;n-1}^{-1}(a_{2;n-1}a_2^{-1})}\\
&&\underline{\Delta (r_{2;n-1}^{-1}(x_2x_1)x_2^{-2}x_1^{-2})(a_{2;n-1}a_2^{-1})^{-1}r_{3;n-1}(a_{3;n-1}a_3^{-1})^{-1}}\\
&&\Delta (x_2x_1)^{-1}\{ (a_{3;n-1}a_3^{-1})(a_{1;n-1}a_1^{-1})\} \\
&\stackrel{(\text{I})}{=}&\{ (a_{3;n-1}a_3^{-1})(a_{1;n-1}a_1^{-1})\} ^{-1}\\
&&\underline{r_{4;n-1}^{-1}\Delta (\{ (a_{3;n-1}a_3^{-1})r_{3;n-1}^{-1}(a_{2;n-1}a_2^{-1})\} (r_{2;n-1}^{-1}(x_2x_1)x_2^{-2}x_1^{-2}))r_{4;n-1}}\\
&&\underline{\Delta (x_4^2)^{-1}\Delta (x_2x_1)^{-1}}\{ (a_{3;n-1}a_3^{-1})(a_{1;n-1}a_1^{-1})\} \\
&\stackrel{(\text{I}),(\text{L-})}{=}&\{ (a_{3;n-1}a_3^{-1})(a_{1;n-1}a_1^{-1})\} ^{-1}\\
&&\underline{\Delta (\{ r_{4;n-1}^{-1}(a_{3;n-1}a_3^{-1})r_{3;n-1}^{-1}(a_{2;n-1}a_2^{-1})\} (r_{2;n-1}^{-1}(x_2x_1)x_2^{-2}x_1^{-2}))}\\
&&\underline{\Delta (x_4^{-2}(x_2x_1)^{-1})}\{ (a_{3;n-1}a_3^{-1})(a_{1;n-1}a_1^{-1})\} \\
&\stackrel{(\text{L+})}{=}&\{ (a_{3;n-1}a_3^{-1})(a_{1;n-1}a_1^{-1})\} ^{-1}\\
&&\Delta (\{ r_{4;n-1}^{-1}(a_{3;n-1}a_3^{-1})r_{3;n-1}^{-1}(a_{2;n-1}a_2^{-1})\} (r_{2;n-1}^{-1}(x_2x_1)x_2^{-2}x_1^{-2})x_4^{-2}\\
&&(x_2x_1)^{-1})\{ (a_{3;n-1}a_3^{-1})(a_{1;n-1}a_1^{-1})\} \\
&\stackrel{(\text{I})}{=}&\Delta (\{ (a_{3;n-1}a_3^{-1})(a_{1;n-1}a_1^{-1})\} ^{-1}(\{ r_{4;n-1}^{-1}(a_{3;n-1}a_3^{-1})r_{3;n-1}^{-1}(a_{2;n-1}a_2^{-1})\} \\
&&(r_{2;n-1}^{-1}(x_2x_1)x_2^{-2}x_1^{-2})x_4^{-2}(x_2x_1)^{-1}))\\
&=&\Delta (b(x_3^2))\\
&\stackrel{(\text{I})}{=}&b(r_{3;n-1})b^{-1}.
\end{eqnarray*}
Thus $\varepsilon _r=0$ for Relation~(D2c)$\prime $ when $i=3$ and $k=n-1$, and we obtain Relation~(D2c) when $i=3$ and $k=n-1$ by Relations~(I) and (I\hspace{-0.04cm}I\hspace{-0.04cm}I).

For Relation~(D2c) when $i=4$ and $k=n-1$, we have
\begin{eqnarray*}
&&r_{4;n-1}(a_{2;n-1}a_2^{-1})r_{1;n-1}(a_{1;n-1}a_1^{-1})^{-1}r_{2;n-1}\\
&&(a_{2;n-1}a_2^{-1})^{-1}r_{3;n-1}(a_{3;n-1}a_3^{-1})^{-1}r_{4;n-1}(a_{3;n-1}a_3^{-1})(a_{1;n-1}a_1^{-1})\\
&=&\Delta (x_4^2)\Delta (x_3x_2)\Delta (x_1^2)\underline{(a_{1;n-1}a_1^{-1})^{-1}\Delta (x_2^2)(a_{1;n-1}a_1^{-1})}\Delta (x_2x_1)^{-1}\\
&&\underline{(a_{2;n-1}a_2^{-1})^{-1}\Delta (x_3^2)(a_{2;n-1}a_2^{-1})}\ \underline{\Delta (x_3x_2)^{-1}\Delta (x_4x_3)^{-1}}\\
&&r_{4;n-1}\Delta (x_4x_3)\Delta (x_2x_1)\\
&\stackrel{(\text{I}),(\text{L0})}{=}&\underline{\Delta (x_4^2)\Delta (x_3x_2)\Delta (x_1^2)}\Delta ((a_{1;n-1}a_1^{-1})^{-1}(x_2^2))\Delta (x_2x_1)^{-1}\\
&&\Delta ((a_{2;n-1}a_2^{-1})^{-1}(x_3^2)) \Delta ((x_3x_2)^{-1}(x_4x_3)^{-1})r_{4;n-1}\underline{\Delta (x_4x_3)\Delta (x_2x_1)}\\
&\stackrel{(\text{L+})}{=}&\Delta (x_4^2(x_3x_2)x_1^2)\Delta ((a_{1;n-1}a_1^{-1})^{-1}(x_2^2))\Delta (x_2x_1)^{-1}\\
&&\underline{\Delta ((a_{2;n-1}a_2^{-1})^{-1}(x_3^2)) \Delta ((x_3x_2)^{-1}(x_4x_3)^{-1})}r_{4;n-1}\Delta (x_4x_3x_2x_1)d_{n-1}^3\\
&\stackrel{(\text{L-})}{=}&\Delta (x_4^2(x_3x_2)x_1^2)\Delta ((a_{1;n-1}a_1^{-1})^{-1}(x_2^2))\underline{\Delta (x_2x_1)^{-1}}\\
&&\underline{\Delta ((a_{2;n-1}a_2^{-1})^{-1}(x_3^2)(x_3x_2)^{-1}(x_4x_3)^{-1})}r_{4;n-1}\Delta (x_4x_3x_2x_1)d_{n-1}^2\\
&\stackrel{(\text{L0})}{=}&\Delta (x_4^2(x_3x_2)x_1^2)\underline{\Delta ((a_{1;n-1}a_1^{-1})^{-1}(x_2^2))\Delta ((x_2x_1)^{-1}}\\
&&\underline{(a_{2;n-1}a_2^{-1})^{-1}(x_3^2)(x_3x_2)^{-1}(x_4x_3)^{-1})}r_{4;n-1}\Delta (x_4x_3x_2x_1)d_{n-1}^2\\
&\stackrel{(\text{L-})}{=}&\underline{\Delta (x_4^2(x_3x_2)x_1^2)\Delta ((a_{1;n-1}a_1^{-1})^{-1}(x_2^2)(x_2x_1)^{-1}}\\
&&\underline{(a_{2;n-1}a_2^{-1})^{-1}(x_3^2)(x_3x_2)^{-1}(x_4x_3)^{-1})}r_{4;n-1}\Delta (x_4x_3x_2x_1)d_{n-1}\\
&\stackrel{(\text{L-})}{=}&\underline{\Delta (x_4^2(x_3x_2)x_1^2(a_{1;n-1}a_1^{-1})^{-1}(x_2^2)(x_2x_1)^{-1}(a_{2;n-1}a_2^{-1})^{-1}(x_3^2)(x_3x_2)^{-1}}\\
&&\underline{(x_4x_3)^{-1})\Delta (x_4x_3x_2x_1)}\ \underline{\Delta (x_4x_3x_2x_1)^{-1}\Delta (x_4^2)\Delta (x_4x_3x_2x_1)}\\
&\stackrel{(\text{L+}),(\text{I})}{=}&\underline{\Delta (x_4^2(x_3x_2)x_1^2(a_{1;n-1}a_1^{-1})^{-1}(x_2^2)(x_2x_1)^{-1}(a_{2;n-1}a_2^{-1})^{-1}(x_3^2)(x_3x_2)^{-1}}\\
&&\underline{(x_4x_3)^{-1}x_4x_3x_2x_1)\Delta (\Delta (x_4x_3x_2x_1)^{-1}(x_4^2))}d_{n-1}\\
&\stackrel{(\text{L-})}{=}&\Delta (x_4^2(x_3x_2)x_1^2(a_{1;n-1}a_1^{-1})^{-1}(x_2^2)(x_2x_1)^{-1}(a_{2;n-1}a_2^{-1})^{-1}(x_3^2)(x_3x_2)^{-1}\\
&&(x_4x_3)^{-1}x_4x_3x_2x_1\Delta (x_4x_3x_2x_1)^{-1}(x_4^2))\\
&=&\Delta (b(x_4^2))\\
&\stackrel{(\text{I})}{=}&b(r_{4;n-1})b^{-1}.
\end{eqnarray*}
Thus $\varepsilon _r=0$ for Relation~(D2c)$\prime $ when $i=4$ and $k=n-1$, and we obtain Relation~(D2c) when $i=4$ and $k=n-1$ by Relations~(I) and (I\hspace{-0.04cm}I\hspace{-0.04cm}I).

For Relation~(D2d) when $m=i-1$ and $k=n-1$, we have
\begin{eqnarray*}
&&\{ (s_{l,n-1}d_l^{-1})(a_{i-1;n-1}a_{i-1}^{-1})\} ^{-1}(a_{i-1;n-1}a_{i-1}^{-1})(s_{l,n-1}d_l^{-1})r_{i;n-1}r_{i-1;n-1}\\
&&(a_{i-1;n-1}a_{i-1}^{-1})^{-1}r_{i;n-1}(\bar{s}_{l,n-1;i}d_l^{-1})\{ (s_{l,n-1}d_l^{-1})(a_{i-1;n-1}a_{i-1}^{-1})\} \\
&=&\{ (s_{l,n-1}d_l^{-1})(a_{i-1;n-1}a_{i-1}^{-1})\} ^{-1}(a_{i-1;n-1}a_{i-1}^{-1})\underline{\Delta (y_l)\Delta (x_i^2)\Delta (x_{i-1}^2)}\\
&&(a_{i-1;n-1}a_{i-1}^{-1})^{-1}\Delta (x_i^2)(\bar{s}_{l,n-1;i}d_l^{-1})\{ (s_{l,n-1}d_l^{-1})(a_{i-1;n-1}a_{i-1}^{-1})\} \\
&\stackrel{(\text{L+})}{=}&\{ (s_{l,n-1}d_l^{-1})(a_{i-1;n-1}a_{i-1}^{-1})\} ^{-1}\underline{(a_{i-1;n-1}a_{i-1}^{-1})\Delta (y_lx_i^2x_{i-1}^2)(a_{i-1;n-1}a_{i-1}^{-1})^{-1}}\\
&\stackrel{(\text{I})}{=}&\{ (s_{l,n-1}d_l^{-1})(a_{i-1;n-1}a_{i-1}^{-1})\} ^{-1}\Delta ((a_{i-1;n-1}a_{i-1}^{-1})(y_lx_i^2x_{i-1}^2))\Delta (\bar{y}_{l;i})\\
&&\underline{(\bar{s}_{l,n-1;i}d_l^{-1})^{-1}\Delta (x_i^2)(\bar{s}_{l,n-1;i}d_l^{-1})}\{ (s_{l,n-1}d_l^{-1})(a_{i-1;n-1}a_{i-1}^{-1})\} d_{n-1}^2\\
&\stackrel{(\text{I})}{=}&\{ (s_{l,n-1}d_l^{-1})(a_{i-1;n-1}a_{i-1}^{-1})\} ^{-1}\underline{\Delta ((a_{i-1;n-1}a_{i-1}^{-1})(y_lx_i^2x_{i-1}^2))\Delta (\bar{y}_{l;i})}\\
&&\Delta ((\bar{s}_{l,n-1;i}d_l^{-1})^{-1}(x_i^2))\{ (s_{l,n-1}d_l^{-1})(a_{i-1;n-1}a_{i-1}^{-1})\} d_{n-1}^2\\
&\stackrel{(\text{L-})}{=}&\{ (s_{l,n-1}d_l^{-1})(a_{i-1;n-1}a_{i-1}^{-1})\} ^{-1}\underline{\Delta ((a_{i-1;n-1}a_{i-1}^{-1})(y_lx_i^2x_{i-1}^2)\bar{y}_{l;i})}\\
&&\underline{\Delta ((\bar{s}_{l,n-1;i}d_l^{-1})^{-1}(x_i^2))}\{ (s_{l,n-1}d_l^{-1})(a_{i-1;n-1}a_{i-1}^{-1})\} d_{n-1}\\
&\stackrel{(\text{L-})}{=}&\{ (s_{l,n-1}d_l^{-1})(a_{i-1;n-1}a_{i-1}^{-1})\} ^{-1}\Delta ((a_{i-1;n-1}a_{i-1}^{-1})(y_lx_i^2x_{i-1}^2)\bar{y}_{l;i}\\
&&(\bar{s}_{l,n-1;i}d_l^{-1})^{-1}(x_i^2))\{ (s_{l,n-1}d_l^{-1})(a_{i-1;n-1}a_{i-1}^{-1})\} \\
&\stackrel{(\text{I})}{=}&\Delta (\{ (s_{l,n-1}d_l^{-1})(a_{i-1;n-1}a_{i-1}^{-1})\} ^{-1}((a_{i-1;n-1}a_{i-1}^{-1})(y_lx_i^2x_{i-1}^2)\bar{y}_{l;i}\\
&&(\bar{s}_{l,n-1;i}d_l^{-1})^{-1}(x_i^2)))\\
&=&\Delta (a_{i-1;n-1}(x_i^2))\\
&\stackrel{(\text{I})}{=}&a_{i-1;n-1}(r_{i;n-1})a_{i-1;n-1}^{-1}.
\end{eqnarray*}
Thus $\varepsilon _r=0$ for Relation~(D2d)$\prime $ when $m=i-1$ and $k=n-1$, and we obtain Relation~(D2d) when $m=i-1$ and $k=n-1$ by Relations~(I) and (I\hspace{-0.04cm}I\hspace{-0.04cm}I).

For Relation~(D2g) when $i=1$ and $k=n-1$, we have
\begin{eqnarray*}
&&[s_{t,n-1}d_t^{-1},(\bar{s}_{l,n-1}d_l^{-1})^{-1}][r_{1;n-1}(\bar{s}_{t,n-1}d_t^{-1})r_{1;n-1}^{-1},(s_{l,n-1}d_l^{-1})^{-1}]r_{1;n-1} \\
&=&\underline{(s_{t,n-1}d_t^{-1})\Delta (\bar{y}_l)^{-1}(s_{t,n-1}d_t^{-1})^{-1}}\Delta (\bar{y}_l)\Delta (x_1^2)(\bar{s}_{t,n-1}d_t^{-1})\\
&&\underline{r_{1;n-1}^{-1}(s_{l,n-1}d_l^{-1})^{-1}r_{1;n-1}\Delta (\bar{y}_t)^{-1}r_{1;n-1}^{-1}(s_{l,n-1}d_l^{-1})r_{1;n-1}}\\
&\stackrel{(\text{I})}{=}&\Delta ((s_{t,n-1}d_t^{-1})(\bar{y}_l))^{-1}\Delta (\bar{y}_l)\Delta (\bar{y_t})\underline{(\bar{s}_{t,n-1}d_t^{-1})^{-1}\Delta (x_1^2)(\bar{s}_{t,n-1}d_t^{-1})}\\
&&\Delta (r_{1;n-1}^{-1}(s_{l,n-1}d_l^{-1})^{-1}r_{1;n-1}(\bar{y}_t))^{-1}\\
&\stackrel{(\text{I})}{=}&\Delta ((s_{t,n-1}d_t^{-1})(\bar{y}_l))^{-1}\underline{\Delta (\bar{y}_l)\Delta (\bar{y_t})}\Delta ((\bar{s}_{t,n-1}d_t^{-1})^{-1}(x_1^2))\\
&&\Delta (r_{1;n-1}^{-1}(s_{l,n-1}d_l^{-1})^{-1}r_{1;n-1}(\bar{y}_t))^{-1}\\
&\stackrel{(\text{L-})}{=}&\Delta ((s_{t,n-1}d_t^{-1})(\bar{y}_l))^{-1}\underline{\Delta (\bar{y}_l\bar{y_t})\Delta ((\bar{s}_{t,n-1}d_t^{-1})^{-1}(x_1^2))}\\
&&\Delta (r_{1;n-1}^{-1}(s_{l,n-1}d_l^{-1})^{-1}r_{1;n-1}(\bar{y}_t))^{-1}d_{n-1}^{-1}\\
&\stackrel{(\text{L-})}{=}&\underline{\Delta ((s_{t,n-1}d_t^{-1})(\bar{y}_l))^{-1}\Delta (\bar{y}_l\bar{y_t}(\bar{s}_{t,n-1}d_t^{-1})^{-1}(x_1^2))}\\
&&\Delta (r_{1;n-1}^{-1}(s_{l,n-1}d_l^{-1})^{-1}r_{1;n-1}(\bar{y}_t))^{-1}d_{n-1}^{-2}\\
&\stackrel{(\text{L+})}{=}&\underline{\Delta ((s_{t,n-1}d_t^{-1})(\bar{y}_l)^{-1}\bar{y}_l\bar{y_t}(\bar{s}_{t,n-1}d_t^{-1})^{-1}(x_1^2))}\\
&&\underline{\Delta (r_{1;n-1}^{-1}(s_{l,n-1}d_l^{-1})^{-1}r_{1;n-1}(\bar{y}_t))^{-1}}d_{n-1}^{-1}\\
&\stackrel{(\text{L+})}{=}&\Delta ((s_{t,n-1}d_t^{-1})(\bar{y}_l)^{-1}\bar{y}_l\bar{y_t}(\bar{s}_{t,n-1}d_t^{-1})^{-1}(x_1^2)r_{1;n-1}^{-1}(s_{l,n-1}d_l^{-1})^{-1}r_{1;n-1}(\bar{y}_t)^{-1})\\
&=&\Delta (\bar{s}_{l,t}(x_1^2))\\
&\stackrel{(\text{I})}{=}&\bar{s}_{l,t}(r_{1;n-1})\bar{s}_{l,t}^{-1}.
\end{eqnarray*}
Thus $\varepsilon _r=0$ for Relation~(D2g)$\prime $ when $i=1$ and $k=n-1$, and we obtain Relation~(D2g) when $i=1$ and $k=n-1$ by Relations~(I) and (I\hspace{-0.04cm}I\hspace{-0.04cm}I).

\par
{\bf Acknowledgement: } The authors would like to express his gratitude to Hisaaki Endo, for his encouragement and helpful advices. 
The second author was supported by JSPS KAKENHI Grant number 15J10066.

\end{document}